\documentclass[12pt,reqno]{amsart}

\usepackage{amsmath,amssymb,amsthm,mathtools}
\usepackage{multirow}
\usepackage[margin=1in]{geometry}
\usepackage{hyperref}
\usepackage{enumitem}
\usepackage{booktabs}
\usepackage{graphicx}
\graphicspath{{../04.LaTex Figures/}}
\usepackage{longtable}
\usepackage{fancyhdr}

\hypersetup{colorlinks=true,linkcolor=red,citecolor=blue,urlcolor=blue}

\theoremstyle{plain}
\newtheorem{theorem}{Theorem}[section]
\newtheorem{proposition}[theorem]{Proposition}
\newtheorem{lemma}[theorem]{Lemma}
\newtheorem{corollary}[theorem]{Corollary}

\theoremstyle{definition}
\newtheorem{definition}[theorem]{Definition}
\newtheorem{hypothesis}[theorem]{Hypothesis}

\theoremstyle{remark}
\newtheorem{remark}[theorem]{Remark}

\newtheorem{openproblem}[theorem]{Open Problem}

\DeclareMathOperator{\dens}{dens}

\newcommand{\Q}{\mathbb{Q}}
\newcommand{\Z}{\mathbb{Z}}
\newcommand{\R}{\mathbb{R}}

\newcommand{\T}{\mathbb{T}}
\newcommand{\E}{\mathbb{E}}
\newcommand{\half}{\tfrac{1}{2}}
\newcommand{\quarter}{\tfrac{1}{4}}

\title[Persistent Heuristics I]{Unconditional Density Bounds for Quadratic Norm-Form Energies via Lorentzian Spectral Weights\\[0.5em]
\textnormal{\large Persistent Heuristics~I: Mini Monograph~I}}

\author{Peter Shiller}
\address{Independent Researcher}

\date{March 12, 2026}

\begin{document}
\raggedbottom

\pagestyle{fancy}
\fancyhf{}                          
\fancyhead[CE]{\leftmark}           
\fancyhead[CO]{\rightmark}          
\fancyfoot[C]{\thepage}             
\fancyfoot[R]{%
  \ifnum\value{page}>3
    {\scriptsize\hyperref[toc]{[\,Back to Table of Contents\,]}}%
  \fi}
\renewcommand{\headrulewidth}{0pt}  
\renewcommand{\footrulewidth}{0pt}
\fancypagestyle{plain}{%
  \fancyhf{}%
  \fancyfoot[C]{\thepage}%
  \fancyfoot[R]{%
    \ifnum\value{page}>3
      {\scriptsize\hyperref[toc]{[\,Back to Table of Contents\,]}}%
    \fi}%
  \renewcommand{\headrulewidth}{0pt}%
  \renewcommand{\footrulewidth}{0pt}%
}

\begin{abstract}
For a real quadratic field $\Q(\sqrt{d})$, we study the norm-form energy
$N = S_\zeta^2 - d \cdot S_L^2$, where $S_\zeta$ and $S_L$ are Lorentzian-weighted
zero sums with $w(\rho) = 2/(\quarter + \gamma^2)$.
We prove three main results.
(1) \emph{Spacelike spectral data}: $N < 0$ unconditionally for all squarefree $d > 1$, as a consequence of a low-lying zero dominance theorem proved via explicit zero-counting.
(2) \emph{Effective density bound}: at each verified truncation level $M$,
$\dens\{N > 0\} \leq 2\|f_{S_L^{(M)}}\|_\infty \cdot (W_1(\zeta)/\sqrt{d} + \epsilon_M)$,
established unconditionally via Jacobi--Anger resonance analysis.
At fixed $M$ the bound is nontrivial only for sufficiently large $d$; the $O(1/\sqrt{d})$ rate requires $M$ to grow with $d$, which in turn requires a uniform density bound that we establish under a computationally verified finite-rank condition on the resonance lattice.
(3) \emph{Exact asymptotic}: under the computationally verified hypothesis that the
infinite resonance lattice $\Lambda_\infty$ has finite rank (verified to have
rank $0$ for $M \leq 20$), the sharp asymptotic
$\dens\{N > 0\} = C(d)/\sqrt{d} + o(1/\sqrt{d})$ holds.  For $d = 5$,
$C(5) = 2\,f_{S_L}(0)\cdot\E[|S_\zeta|] = 0.1193$; the constant depends on $d$
through the zeros of $L(s,\chi_d)$, and $C(d) = O(1/\log d)$ as $d \to \infty$.

\medskip

Appendix F tabulates between 1004 and 1044 zeros at 70 decimal places for $L(s,\chi_2)$, $L(s,\chi_3)$, $L(s,\chi_5)$, $L(s,\chi_6)$, $L(s,\chi_7)$, $L(s,\chi_{10})$, $L(s,\chi_{11})$, and $L(s,\chi_{13})$, all rigorously certified by ARB interval arithmetic.
\end{abstract}


\maketitle
\setcounter{tocdepth}{2}
\phantomsection\label{toc}
\tableofcontents

\section{Introduction}\label{sec:intro}

Let $d > 1$ be squarefree and let $\chi_d$ denote the Kronecker symbol of the fundamental discriminant $\Delta_K$ of $K = \Q(\sqrt{d})$. The Dedekind zeta function factors as
\begin{equation}\label{eq:zeta-K-factor}
\zeta_K(s) = \zeta(s) \cdot L(s, \chi_d),
\end{equation}
coupling the zeros of the Riemann zeta function to those of the quadratic Dirichlet $L$-function.  We ask whether this coupling can be detected spectrally, by weighting zeros so that low-lying ones dominate.  The answer is yes: the resulting quadratic form on weighted zero sums is always negative (the $L$-function zeros win), and the set where it becomes positive has density $O(1/\sqrt{d})$.  The density statement holds unconditionally in the Euler-product formulation (Theorem~\ref{thm:density}); on the zero side, both statements hold at any verified truncation level and extend to the full series under the hypothesis that the infinite resonance lattice has finite rank (computationally consistent with rank~$0$ for the first $20$ zeros).

For a zero $\rho = \half + i\gamma$ on the critical line, define the Lorentzian weight
\begin{equation}\label{eq:lorentzian-weight}
w(\rho) = \frac{2}{\quarter + \gamma^2}.
\end{equation}
This weight emphasizes low-lying zeros: a zero at height $\gamma = 1$ contributes weight $1.6$, while a zero at height $\gamma = 14$ contributes weight $0.01$. Define the weighted zero sums
\begin{equation}\label{eq:S-sums-intro}
S_\zeta = \sum_{\zeta(\rho) = 0} w(\rho), \qquad S_L = \sum_{L(\rho, \chi_d) = 0} w(\rho),
\end{equation}
and the norm-form energy
\begin{equation}\label{eq:norm-form-intro}
N = S_\zeta^2 - d \cdot S_L^2.
\end{equation}
The signature $(1,1)$ of this quadratic form reflects the Galois structure of $K/\Q$: the trivial eigenspace contributes positively, the sign eigenspace negatively.

\subsection{Main results}

We prove unconditionally that:

\begin{enumerate}[label=\textup{(\arabic*)}]
\item \textbf{Spacelike spectral data} (Corollary~\ref{cor:spacelike}): $N = S_\zeta^2 - d \cdot S_L^2 < 0$ for all squarefree $d > 1$.  This follows from the low-lying zero dominance theorem (Theorem~\ref{thm:low-lying}: $S_L > S_\zeta^*$) via explicit zero-counting, and extends to Ces\`aro averages (Theorem~\ref{thm:cesaro}: $\langle N \rangle_T < 0$ for all $T > 0$).

\item \textbf{Effective density bound} (Theorem~\ref{thm:density}): $\dens\{N > 0\} \leq 2\|f_{S_L^{(M)}}\|_\infty \cdot (W_1(\zeta)/\sqrt{d} + \epsilon_M)$ unconditionally at any verified truncation level $M$.

\item \textbf{Density rate and exact asymptotic} (Corollary~\ref{cor:density-rate}, Theorem~\ref{thm:arctan}): $\dens\{N > 0\} = O(1/\sqrt{d})$, sharpened to $\dens\{N > 0\} = C(d)/\sqrt{d} + o(1/\sqrt{d})$, under the hypothesis that the infinite resonance lattice $\Lambda_\infty$ has finite rank $d_\infty < \infty$ (a hypothesis; computationally consistent with $d_M = 0$ for $M \leq 20$, but not a consequence of that finite truncation fact).  For $d = 5$, $C(5) = 0.1193$; the constant $C(d) = O(1/\log d)$ as $d \to \infty$.
\end{enumerate}

These results do not assume the Riemann Hypothesis for either function. The spacelike property is robust to the choice of weight function (see Section~\ref{subsec:weight-robustness}).  As a byproduct of the low-lying zero dominance proof, we obtain a universal first-zero bound $\gamma_1'(d) < 14.13$ for all fundamental discriminants (Proposition~\ref{prop:first-zero-bound}).

\subsection{Logical structure and the GSH barrier}

The progression of results reveals the fundamental role of $\Q$-linear independence:

\begin{table}[ht]
\centering
\renewcommand{\arraystretch}{1.4}
\resizebox{\linewidth}{!}{%
\small
\begin{tabular}{lccc}
\toprule
\textbf{Result} & \textbf{Assumes} & \textbf{Bound} & \textbf{Method} \\
\midrule
Thm~\ref{thm:density-vanishing}: Density vanishing & Hypothesis~\ref{hyp:GSH}(ii) & $O(1/\sqrt{d})$ & Jessen--Wintner \\
Thm~\ref{thm:euler-density}: Euler product density & nothing & $O(1/\sqrt{d})$ & $\Q$-indep.\ of $\log p$ \\
\midrule
\multicolumn{4}{c}{\emph{The barrier: zero ordinates lack unconditional $\Q$-independence}} \\
\midrule
Thm~\ref{thm:finite}: Unconditional finiteness & nothing & $f_{S_L^{(M)}}(0) < \infty$ & Jacobi--Anger + Bessel \\
Thm~\ref{thm:density}: Effective bound & nothing & $2\|f\|_\infty(W_1/\!\sqrt{d} + \epsilon_M)$ & Containment + stability \\
Cor.~\ref{cor:density-rate}: Density rate & $d_\infty < \infty$ & $O(1/\sqrt{d})$ & $M(d)$ + uniform bound \\
\midrule
Thm~\ref{thm:telescoping}: Limiting density & Verified ($M = 20$) & $f_{S_L}(0) < \infty$ & DCT + resonance absence \\
Thm~\ref{thm:arctan}: Exact formula & Verified ($M = 20$) & $C/\sqrt{d}$ & Marginal densities \\
\bottomrule
\end{tabular}
}
\end{table}

The Jessen--Wintner approach requires $\Q$-linear independence to identify the Ces\`aro distribution with a product of independent random variables. On the Euler product side, this independence follows from unique factorization. On the zero side, it is precisely the Grand Simplicity Hypothesis. The unconditional effective bound bypasses this barrier at a fixed truncation level via the stability lemma; the $O(1/\sqrt{d})$ rate requires the truncation level to grow with $d$, which in turn requires a uniform bound on the density $\|f_{S_L^{(M)}}\|_\infty$.  Under the hypothesis $d_\infty = 0$ (which extends the finite truncation fact $d_M = 0$, verified for $M \leq 20$, to the full infinite series), this uniform bound holds because the unsigned Bessel integral is monotone decreasing in $M$; more generally, it suffices that the infinite resonance lattice $\Lambda_\infty$ has finite rank $d_\infty < \infty$.

\subsection{Series context}

This paper is the first in the series \emph{Persistent Heuristics}.  The Lorentzian weight arises from a spectral interpretation of the Weil explicit formula that emphasizes low-lying zeros over high zeros.  While this emphasis might initially appear restrictive, it reveals genuine invariants: the spacelike property holds for any weight satisfying mild decay conditions (Section~\ref{subsec:weight-robustness}), and the density bounds depend on the weight only through computable constants.

\subsection{Organization}

Section~\ref{sec:lorentzian} defines the Lorentzian weight and establishes its properties. Section~\ref{sec:negativity} proves the low-lying zero dominance and spacelike theorems. Section~\ref{sec:cesaro} extends the negativity to Ces\`aro averages and establishes conditional density bounds under GSH, revealing the $\Q$-independence barrier. Section~\ref{sec:besicovitch} introduces the Besicovitch norm framework and the spectral variance decomposition that bridge the Ces\`aro analysis to the Jacobi--Anger approach of the subsequent sections. Sections~\ref{sec:finite} and~\ref{sec:infinite} establish per-function density bounds via Jacobi--Anger analysis and finite truncation. Section~\ref{sec:compare} combines these with a stability argument to prove the effective density bound, and Section~\ref{sec:exact} derives the exact formula under the finite-rank resonance hypothesis. Section~\ref{sec:null-crossing} analyzes the unique null crossing on the real axis. Section~\ref{sec:connections} discusses connections and open problems. Appendix~\ref{sec:numerical} provides numerical verification.

\begin{remark}[Purpose of Appendix F]\label{rem:supplementary-data}
The proofs in this paper require only the zero data in
Appendices~\ref{app:chi5-highprec}--\ref{app:certification}: the first
$200$ zeros of $L(s,\chi_5)$ at $70$-digit precision, the first $20$
zeros of $L(s,\chi_d)$ at $20$-digit precision for
$d \in \{2,3,6,7,10,11,13\}$, and the first $60$ zeros of $\zeta(s)$.
In the course of this work, however, we computed and certified between 1004 and 1044
zeros of $L(s,\chi_d)$ at $70$ decimal places for all eight
discriminants (8244 zeros in total); these extended tables are collected
in the companion document Appendix~\ref{app:full-zero-tables} and are
freely available as ancillary data files.  To the best of our knowledge,
no publicly available repository provides bulk high-precision zero
ordinates for individual Dirichlet $L$-functions: the
LMFDB~\cite{LMFDB} stores approximately $25$ zeros per character at
$30$-digit precision (computed by Platt~\cite{Platt2011}), and the
dataset of O'Bryant and Rechnitzer~\cite{OBryantRechnitzer2023} provides
zeros for all primitive characters of conductor up to $934$ at $12$-digit
precision.  We hope the present tables serve as a useful reference for
researchers requiring certified zero data for quadratic $L$-functions.
\end{remark}

\subsection{Conventions}

\begin{figure}[ht]
\centering
\includegraphics[width=\textwidth]{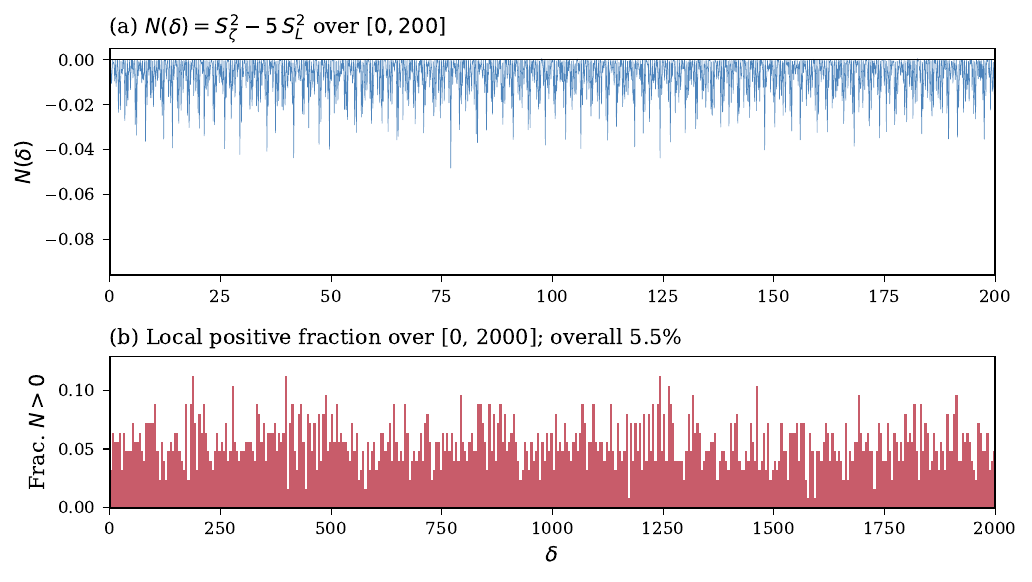}
\caption{The norm form $N(\delta) = S_\zeta(\delta)^2 - 5\,S_L(\delta)^2$ for $d = 5$, computed using the first 100 zeros of each function with Lorentzian weights $w(\gamma) = 2/(\frac14 + \gamma^2)$.  Panel~(a): the oscillating signal over $[0,200]$, showing that $N$ is overwhelmingly negative.  Panel~(b): the local fraction of $\delta$-values where $N > 0$, computed in sliding windows across $[0,2000]$; the measure of the set where $N > 0$ is approximately $5.5\%$ over $[0,2000]$, consistent with the predicted $O(1/\sqrt{d})$ density.}
\label{fig:norm-form}
\end{figure}

Sums over zeros $\sum_\rho$ run over the upper half-plane $\gamma > 0$, with each term representing a conjugate pair $\rho, \bar{\rho}$. The zero-counting function $N(T)$ counts zeros with $0 < \gamma \leq T$. The Fourier transform convention is $\hat{f}(\xi) = \int_{-\infty}^\infty f(x) e^{-2\pi i x \xi}\,dx$.

Throughout, $\gamma_k$ denotes the $k$-th positive ordinate of $\zeta(s)$ (so $\gamma_1 = 14.134\ldots$), while $\gamma_k'$ denotes the $k$-th positive ordinate of $L(s, \chi_d)$ for the quadratic character under consideration. The unprimed/primed convention is maintained consistently: unprimed ordinates belong to $\zeta$, primed ordinates to the Dirichlet $L$-function.

Throughout, $\chi_d$ denotes the Kronecker symbol of the fundamental discriminant $\Delta_K$ of $\Q(\sqrt{d})$; the conductor is $q = |\Delta_K| \in \{d, 4d\}$.

\section{The Lorentzian Energy Functional}\label{sec:lorentzian}

\subsection{The Lorentzian weight}

\begin{definition}[Lorentzian weight]\label{def:lorentzian}
For a zero $\rho = \beta + i\gamma$ in the critical strip, the \emph{Lorentzian weight} is
\[
w(\rho) = \frac{2}{\quarter + \gamma^2}
\]
when $\beta = \half$. For zeros off the critical line, the unconditional weight is
\begin{equation}\label{eq:weight-general}
w(\rho) = \frac{2(\beta(1-\beta)+\gamma^2)}{(\beta(1-\beta)+\gamma^2)^2 + 4\gamma^2(\half-\beta)^2}.
\end{equation}
\end{definition}

The weight~\eqref{eq:weight-general} reduces to~\eqref{eq:lorentzian-weight} when $\beta = \half$.

\begin{proposition}[Critical line maximality]\label{prop:CLM}
For fixed $\gamma$ with $|\gamma| \geq 1/(2\sqrt{3})$, the weight $w(\beta,\gamma)$ is maximized at $\beta = \half$.
\end{proposition}

\begin{proof}
Set $\varepsilon = \half - \beta$ and write $D = \quarter - \varepsilon^2 + \gamma^2$. Then
\[
w(\half,\gamma) - w(\beta,\gamma)
= \frac{2\varepsilon^2(3\gamma^2 - \quarter + \varepsilon^2)}{(\quarter+\gamma^2)(D^2 + 4\gamma^2\varepsilon^2)}.
\]
For $|\gamma| \geq 1/(2\sqrt{3})$, the factor $3\gamma^2 - \quarter \geq 0$, so the expression is non-negative, with equality only at $\varepsilon = 0$.
\end{proof}

Since every nontrivial zero of $\zeta(s)$ has $|\gamma| \geq 14.13 > 1/(2\sqrt{3})$, all zeta zeros are unconditionally $\gamma_{\mathrm{knw}}$ (critical-line membership verified).  For $L$-function zeros, this paper uses two regimes.  When $\beta = \half$ is individually verified (by sign changes of the Hardy $Z$-function or ARB certification), the zero ordinate is denoted $\gamma'_{\mathrm{knw}}$ and the on-line weight $w(\rho) = 2/(\quarter + \gamma'^2)$ is exact.  When a zero's existence is guaranteed by zero-counting (Theorem~\ref{thm:BMOR}) but $\beta$ is not individually verified, the ordinate is denoted $\gamma'_{\mathrm{unk}}$ and we use only the unconditional lower bound
\begin{equation}\label{eq:weight-lower}
w(\rho) \;\geq\; w^-(\gamma') \;:=\; \frac{2}{1 + \gamma'^2},
\end{equation}
valid for all $\beta \in (0,1)$.  (Proof: writing $a = \beta(1-\beta) \in (0, \tfrac{1}{4}]$, one has $w(\rho) - 2/(1+\gamma'^2) = 2a(1 - a + 3\gamma'^2)/\bigl((a+\gamma'^2)^2 + (1-2\beta)^2\gamma'^2\bigr)(1+\gamma'^2) > 0$.)  Using the on-line weight $2/(\quarter + \gamma'^2)$ for $\gamma'_{\mathrm{unk}}$ zeros would implicitly assume the Riemann Hypothesis for $L(s,\chi_d)$.

\subsection{Weighted zero sums}

\begin{definition}[Weighted zero sums]\label{def:S-sums}
For a primitive Dirichlet character $\chi$ with conductor $q$, define
\[
S_\zeta = \sum_{\rho:\,\zeta(\rho)=0} w(\rho), \qquad
S_{L(\chi)} = \sum_{\rho:\,L(\rho,\chi)=0} w(\rho),
\]
where the sums run over nontrivial zeros in the upper half-plane, each representing a conjugate pair. For $\chi = \chi_d$, we write $S_L = S_{L(\chi_d)}$.
\end{definition}

Both sums converge absolutely: $w(\rho) = O(1/\gamma^2)$ and the zero-counting function satisfies $N(T) = O(T\log T)$.

\begin{proposition}[Explicit value of $S_\zeta$]\label{prop:S-zeta-value}
The weighted zeta zero sum satisfies
\begin{equation}\label{eq:S-zeta-bound}
S_\zeta \leq 0.04871 =: S_\zeta^*.
\end{equation}
\end{proposition}

The tail bound uses the following explicit result for $\zeta(s)$.

\begin{theorem}[Trudgian {\cite[Corollary~1]{Trudgian2012}}]\label{thm:trudgian}
Let $N(T)$ denote the number of zeros $\rho = \beta + i\gamma$ of $\zeta(s)$ with
$0 < \beta < 1$ and $0 < \gamma < T$.  For $T \geq T_0 \geq e$,
\[
\left|N(T) - \frac{T}{2\pi}\log\frac{T}{2\pi e} - \frac{7}{8}\right|
\leq 0.111\log T + 0.275\log\log T + 2.450 + \frac{0.2}{T_0}.
\]
\end{theorem}

\begin{proof}[Proof of Proposition~\ref{prop:S-zeta-value}]
The first $6000$ verified zeta zeros are summed in ARB ball
arithmetic~\cite{Johansson2017}, giving the certified value
$S_\zeta^{(6000)} = 0.04580\ldots$\,.  The tail
$\sum_{\gamma > \gamma_{6000}} w(\gamma)$ is bounded by Abel summation
against Theorem~\ref{thm:trudgian}, using the exact antiderivative
\[
  \int_a^b \frac{4t}{(\quarter + t^2)^2}\,dt
  = \frac{2}{\quarter + a^2} - \frac{2}{\quarter + b^2}
\]
on subintervals, with no numerical quadrature.  The certified tail bound is
$\leq 4.45 \times 10^{-4}$.  Combining:
\[
  S_\zeta \leq 0.04625 < 0.04871 =: S_\zeta^*.
\]
Full computational parameters are recorded in
Appendix~\ref{sec:numerical}, \S\ref{subsec:cert-Szeta}.
\end{proof}

\subsection{The norm-form energy}

\begin{definition}[Norm-form energy]\label{def:norm-form}
The \emph{norm-form energy} is
\[
N = S_\zeta^2 - d \cdot S_L^2.
\]
\end{definition}

The quadratic form $N(a,b) = a^2 - db^2$ has signature $(1,1)$, corresponding to the eigenspace decomposition of the Galois group $\mathrm{Gal}(K/\Q) \cong C_2$.

\begin{definition}[Causal classification]\label{def:causal}
A vector $(x, y) \in \R^2$ with $x, y > 0$ is \emph{timelike} if $x^2 - dy^2 > 0$, \emph{null} if $x^2 - dy^2 = 0$, and \emph{spacelike} if $x^2 - dy^2 < 0$.
\end{definition}

The null cone consists of the lines $x = \pm\sqrt{d}\,y$. The spacelike region lies outside this cone.  The terminology is borrowed by analogy from Lorentzian geometry, where the quadratic form $x^2 - c^2 t^2$ classifies vectors; the analogy is structural (both involve indefinite quadratic forms of signature $(1,1)$) but no physical connection is implied.

\subsection{Eigenspace interpretation}

The decomposition $(\zeta(s), L(s, \chi_d))$ corresponds to the $\pm 1$ eigenspaces of the Galois involution $\sigma: \sqrt{d} \mapsto -\sqrt{d}$ acting on $\zeta_K(s) = \zeta(s) \cdot L(s, \chi_d)$. The norm form $N(a, b) = a^2 - db^2$ is the unique indefinite $C_2$-invariant quadratic form up to scaling. The definite form $a^2 + db^2$ is also $C_2$-invariant but irrelevant here since it has no null cone. The coefficient $-d$ arises from the field norm $N_{K/\Q}(a + b\sqrt{d}) = a^2 - db^2$.

\subsection{Per-prime sign structure}

The norm-form energy admits a per-prime decomposition that reveals the sign structure driving the negativity theorem.

\begin{lemma}[Per-prime contributions]\label{lem:per-prime}
For a prime $p$ with $\chi_d(p) = \chi \in \{-1, 0, +1\}$, define
\[
S^+(p) := \frac{\log p}{\sqrt{p}-1} > 0, \qquad
S^-(p) := \frac{\chi\log p}{\sqrt{p}-\chi}.
\]
If $p$ is split ($\chi = +1$), then $S^-(p) = S^+(p) > 0$. If $p$ is inert ($\chi = -1$), then $S^-(p) = -\log p/(\sqrt{p}+1) < 0$. If $p$ is ramified ($\chi = 0$), then $S^-(p) = 0$.
\end{lemma}

\begin{theorem}[Per-prime sign]\label{thm:per-prime-sign}
For $K = \Q(\sqrt{d})$, $d > 1$ squarefree, and any prime $p$, the per-prime norm-form contribution $N_p = 4\pi(S^+(p) - d\,S^-(p))$ satisfies:
\begin{enumerate}[label=\textup{(\alph*)}]
\item If $p$ is split ($\chi_d(p) = +1$): $N_p = \frac{4\pi\log p}{\sqrt{p}-1}(1 - d) < 0$.
\item If $p$ is inert ($\chi_d(p) = -1$): $N_p = 4\pi\log p\bigl(\frac{1}{\sqrt{p}-1} + \frac{d}{\sqrt{p}+1}\bigr) > 0$.
\item If $p$ is ramified ($\chi_d(p) = 0$): $N_p = \frac{4\pi\log p}{\sqrt{p}-1} > 0$.
\end{enumerate}
\end{theorem}

\begin{proof}
Part~(a): when $\chi_d(p) = +1$, we have $S^+(p) = S^-(p)$, so $N_p = 4\pi S^+(p)(1-d) < 0$ since $d > 1$. Part~(b): when $\chi_d(p) = -1$, $S^-(p) = -\log p/(\sqrt{p}+1)$, so $N_p = 4\pi\log p(1/(\sqrt{p}-1) + d/(\sqrt{p}+1)) > 0$. Part~(c): when $\chi_d(p) = 0$, $S^-(p) = 0$, so $N_p = 4\pi S^+(p) > 0$.
\end{proof}

\begin{remark}[Sign inversion]
The $C_2$-invariant quadratic form inverts the algebraic hierarchy: split primes that contribute equally to both eigenspaces ($S^+ = S^-$) receive the full penalty from the coefficient $-d$, producing negative energy. Inert primes yield a sign reversal via $S^- < 0$, producing positive energy contributions.
\end{remark}

\section{Negativity and Spacelike Theorems}\label{sec:negativity}

\subsection{Low-lying zero dominance}

\begin{theorem}[Low-lying zero dominance]\label{thm:low-lying}
For every squarefree $d > 1$, the $L$-function spectral sum dominates: $S_L > S_\zeta^*$.
\end{theorem}

The proof relies on the following explicit zero-counting theorem.

\begin{theorem}[Bennett--Martin--O'Bryant--Rechnitzer {\cite[Theorem~1.1]{BMOR2021}}]\label{thm:BMOR}
Let $\chi$ be a primitive character with conductor $q > 1$ and $T \geq 5/7$.
Set $\ell = \log\bigl(q(T+2)/(2\pi)\bigr)$.
If $\ell \leq 1.567$, then $N(T,\chi) = 0$.
If $\ell > 1.567$, then
\[
\Bigl|N(T,\chi) - \frac{T}{\pi}\log\frac{qT}{2\pi e}
- \frac{\chi(-1)}{4}\Bigr|
\leq 0.22737\,\ell + 2\log(1 + \ell) - 0.5.
\]
In particular, $N(T,\chi) \geq (T/\pi)\log(qT/2\pi e)
- \chi(-1)/4 - 0.22737\,\ell - 2\log(1+\ell) + 0.5$.
\end{theorem}

\begin{proof}[Proof of Theorem~\ref{thm:low-lying}]
We establish the inequality for all fundamental discriminants by combining Theorem~\ref{thm:BMOR} with verified computations for two small-conductor exceptions.  Throughout, zeros are counted in the upper half-plane only, with the $\gamma'_{\mathrm{knw}}$/$\gamma'_{\mathrm{unk}}$ convention and the lower bound $w^-$~\eqref{eq:weight-lower} as defined above.

\medskip\noindent\textbf{Step 1: Conductors $q \geq 12$.}
Let $\chi_d$ be the Kronecker character with conductor $q = |\Delta_K|$.  The conductor of a real quadratic field satisfies $q \in \{d, 4d\}$ depending on $d \bmod 4$; the values $q \in \{5, 8\}$ correspond to $d \in \{5, 2\}$ and are treated separately below.  The next smallest conductor is $q = 12$ (for $d = 3$).

For $T = 11$, Theorem~\ref{thm:BMOR} gives the lower bound
\[
  N(11, \chi_d) \geq \frac{11}{\pi}\log\frac{11q}{2\pi e} - \frac{|\chi_d(-1)|}{4} - 0.22737\,\ell - 2\log(1+\ell) + 0.5,
\]
where $\ell = \log(q \cdot 13/(2\pi))$.  For $q = 12$, this evaluates to $N(11, \chi_d) \geq 3.80$; since $N(11,\chi_d)$ is an integer, this implies at least four zeros with $\gamma'_{\mathrm{unk}} \leq 11$, and the bound increases with $q$.  Each contributes weight at least $w^-(11) = 2/122 = 1/61$.  Therefore
\[
  S_L \geq 4 \cdot w^-(11) = 4/61 = 0.06557 > 0.04871 = S_\zeta^*.
\]

\medskip\noindent\textbf{Step 2: $d = 5$ (conductor $q = 5$).}
The BMOR bound gives only $N(11, \chi_5) \geq 1.40$, which does not guarantee three zeros.  We therefore verify the inequality directly.  Computing the Hardy $Z$-function $Z_{\chi_5}(t) = e^{i\theta(t)} L(\half + it, \chi_5)$ and scanning for sign changes on $(0, 12)$ yields two zeros below height $11$: $\gamma'_{1,\mathrm{knw}} = 6.6485$ and $\gamma'_{2,\mathrm{knw}} = 9.8314$, both verified to satisfy $|L(\tfrac{1}{2} + i\gamma', \chi_5)| < 10^{-450}$ (see Appendix~\ref{app:chi5-highprec}).  (Any additional zeros below height $11$, if they exist, would only increase $S_L$.)  Therefore
\begin{align*}
  S_L &\geq w(\gamma'_{1,\mathrm{knw}}) + w(\gamma'_{2,\mathrm{knw}}) \\
  &= \frac{2}{0.25 + 44.20} + \frac{2}{0.25 + 96.66} = 0.04499 + 0.02064 = 0.06563 > S_\zeta^*.
\end{align*}

\medskip\noindent\textbf{Step 3: $d = 2$ (conductor $q = 8$).}
The Hurwitz decomposition
\[
L(s,\chi_2) = 8^{-s}\sum_{a=1}^8 \chi_2(a)\,\zeta(s,a/8),
\]
evaluated via the Hardy $Z$-function method (\S\ref{subsec:Lfunc-method}) at a minimum of 1500-bit ARB precision, yields three zeros below height $11$:
\[
\gamma'_{1,\mathrm{knw}} = 4.8999739970070365010, \quad
\gamma'_{2,\mathrm{knw}} = 7.6284288417693978341,
\]
\[
\gamma'_{3,\mathrm{knw}} = 10.806588163861712014,
\]
each satisfying $|L(\tfrac{1}{2} + i\gamma', \chi_2)| < 10^{-450}$ (see Appendix~\ref{app:chi8-zeros}).  Any additional zeros below height $11$ would only increase $S_L$.  Therefore
\begin{align*}
  S_L &\geq w(\gamma'_{1,\mathrm{knw}}) + w(\gamma'_{2,\mathrm{knw}}) + w(\gamma'_{3,\mathrm{knw}}) \\
  &= \frac{2}{0.25 + 24.010} + \frac{2}{0.25 + 58.193} + \frac{2}{0.25 + 116.782} \\
  &= 0.08246 + 0.03421 + 0.01710 = 0.13377 > S_\zeta^*.
\end{align*}

Since every squarefree $d > 1$ has conductor $q \in \{5, 8\}$ or $q \geq 12$, the three cases exhaust all possibilities.
\end{proof}

\begin{remark}[Margins]\label{rem:margins}
The margin $(S_L - S_\zeta^*)/S_\zeta^*$ is smallest in the generic case: for $q = 12$, the BMOR bound guarantees at least four $\gamma'_{\mathrm{unk}}$ zeros below height $11$, giving $S_L \geq 4 \cdot w^-(11) = 4/61 = 0.06557$, a margin of $34.6\%$ over $S_\zeta^*$.  The margin is largest for $d = 2$, where three $\gamma'_{\mathrm{knw}}$ zeros give $S_L \geq 0.13375$, a margin of $174.6\%$.
\end{remark}

\begin{remark}[On first-zero sufficiency]\label{rem:first-zero}
For most discriminants, a single directly computed zero suffices: $w(\gamma'_{1,\mathrm{knw}}) > S_\zeta^*$.  The exception is $d = 5$ ($\gamma'_{1,\mathrm{knw}} = 6.6485$), where $w(\gamma'_{1,\mathrm{knw}}) = 0.04499 < S_\zeta^* = 0.04871$ and the second zero is required.  In the generic $\gamma'_{\mathrm{unk}}$ case (Step~1), at least four zeros are needed.
\end{remark}

\subsection{The spacelike corollary}

\begin{corollary}[Spacelike]\label{cor:spacelike}
For every squarefree $d > 1$, the norm-form energy satisfies $N = S_\zeta^2 - d \cdot S_L^2 < 0$.
\end{corollary}

\begin{proof}
By Proposition~\ref{prop:S-zeta-value}, $S_\zeta \leq S_\zeta^* = 0.04871$.  By Theorem~\ref{thm:low-lying}, $S_L > S_\zeta^* \geq S_\zeta > 0$.  Since $d > 1$:
\[
d \cdot S_L^2 > d \cdot (S_\zeta^*)^2 \geq d \cdot S_\zeta^2 > S_\zeta^2,
\]
hence $N = S_\zeta^2 - d \cdot S_L^2 < 0$.
\end{proof}

\begin{remark}[Strict hierarchy]\label{rmk:hierarchy}
The results form a strict chain. Let $r = S_\zeta/S_L$. The spacelike corollary requires $r < \sqrt{d}$. The low-lying zero dominance gives $r < 1$. Since $r < 1 < \sqrt{d} < d$ for $d > 1$, each result is strictly stronger than what the corollary requires. The actual ratio $r \approx 0.1$ to $0.5$ for small $d$, reflecting the severity of the low-lying zero phenomenon.
\end{remark}

\subsection{The negativity theorem for general characters}

\begin{theorem}[Negativity]\label{thm:negativity}
For every primitive nontrivial Dirichlet character $\chi$ with conductor $q \geq 3$:
\[
E_\chi(0) := S_\zeta - q \cdot S_{L(\chi)} < 0.
\]
\end{theorem}

\begin{proof}
The proof follows the same structure as Theorem~\ref{thm:low-lying}. For $q \geq 4$, Theorem~\ref{thm:BMOR} guarantees a $\gamma'_{\mathrm{unk}}$ zero below height $T = 10.5$, giving $q \cdot w^-(\gamma'_{1,\mathrm{unk}}) \geq 4 \cdot 2/111.25 = 32/445 = 0.0719 > S_\zeta^*$. For $q = 3$, the BMOR bound at $T = 17$ guarantees at least three $\gamma'_{\mathrm{unk}}$ zeros below height $17$ (BMOR gives $N \geq 3.34$ at $T = 17$), so $q \cdot S_{L(\chi)} \geq q \cdot 3 \cdot w^-(17) = 3 \cdot 3 \cdot 2/(1 + 17^2) = 18/290 = 0.0621 > S_\zeta^*$.
\end{proof}

\begin{corollary}[Linear deepening]\label{cor:linear-deepening}
For $q \geq 3$:
\[
E_\chi(0) \leq S_\zeta^* - \frac{2q}{122} < 0,
\]
so $|E_\chi(0)|$ grows at least linearly in the conductor.
\end{corollary}

\begin{proof}
By~\eqref{eq:S-zeta-bound}, $S_\zeta \leq S_\zeta^*$. For $q \geq 4$, the BMOR bound guarantees a $\gamma'_{\mathrm{unk}}$ zero below height $T(q) \leq 10.5 < 11$, so $S_{L(\chi)} \geq w^-(11) = 1/61$. For $q = 3$, the BMOR bound at $T = 17$ guarantees at least three $\gamma'_{\mathrm{unk}}$ zeros below height $17$ (BMOR gives $N \geq 3.34$), so $S_{L(\chi)} \geq 3 \cdot w^-(17) = 6/290 > 1/61$. In both cases, $q \cdot S_{L(\chi)} \geq 2q/122$, and $E_\chi(0) = S_\zeta - q \cdot S_{L(\chi)} \leq S_\zeta^* - 2q/122$. Since $2 \cdot 3/122 = 3/61 = 0.04918 > S_\zeta^* = 0.04871$, this is negative for all $q \geq 3$.
\end{proof}

\subsection{Unconditional first-zero bound}

\begin{proposition}[First-zero bound]\label{prop:first-zero-bound}
For every fundamental discriminant $\Delta_K$, the first zero of $L(s, \chi_d)$ satisfies $\gamma_1'(d) < 14.13$.
\end{proposition}

\begin{proof}
We apply Theorem~\ref{thm:BMOR} with $T = 14.13$.  The smallest conductor among fundamental discriminants is $q = 5$ (for $d = 5$), giving $\ell = \log(5 \cdot 16.13/(2\pi)) = 2.552 > 1.567$, so the non-trivial case of BMOR applies for all $q \geq 5$.  The bound then gives $N(14.13, \chi_d) \geq 3$ for every primitive character with conductor $q \geq 5$; in particular, at least one zero has $\gamma' \leq 14.13$.  Since every fundamental discriminant has $q \geq 5$, the bound holds universally.
\end{proof}

\begin{remark}
This bound is the correct unconditional consequence of the BMOR theorem.  Numerical data shows $\gamma_1'(d) < 7$ for all $d$ tested, and under GRH, $\gamma_1'(d) \to 0$ as $d \to \infty$ consistent with Katz--Sarnak predictions.  However, the bound $\gamma_1' < 14.13$ is all that can be extracted unconditionally from zero-counting methods.
\end{remark}

\subsection{Robustness}

\begin{proposition}[Robustness]\label{prop:robustness}
The negativity and spacelike results are robust to RH-failure for $\zeta$, and to RH-failure for $L(s,\chi_d)$ in the presence of a Siegel zero. If the Riemann Hypothesis fails for $\zeta(s)$, off-line zeros have smaller Lorentzian weight by Proposition~\ref{prop:CLM}, decreasing $S_\zeta$. If RH fails for $L(s, \chi_d)$ via a Siegel zero $\rho_0 = \beta_0$ with $\beta_0$ near $1$, the weight $w(\rho_0) = 2/(\beta_0(1-\beta_0))$ diverges, increasing $S_L$. Both scenarios strengthen the inequalities.
\end{proposition}

\section{Ces\`aro Density Bounds via Almost Periodic Spectral Sums}\label{sec:cesaro}

\subsection{Shifted zero sums}

\begin{definition}[Shifted sums]\label{def:shifted}
For $\delta \in \R$, define
\[
S_\zeta(\delta) = \sum_{\rho:\,\zeta(\rho)=0} w(\rho) \cos(\gamma\delta), \qquad
S_L(\delta) = \sum_{\rho:\,L(\rho,\chi_d)=0} w(\rho) \cos(\gamma'\delta),
\]
and $N(\delta) = S_\zeta(\delta)^2 - d \cdot S_L(\delta)^2$.
\end{definition}

The analysis of Ces\`aro averages rests on the following classical theorem, which we state in the form needed here.

\begin{theorem}[Bohr {\cite{Bohr1947}}, Parseval theorem for almost periodic functions]\label{thm:bohr-parseval}
Let $f(x) = \sum_n a_n e^{i\lambda_n x}$ be a uniformly almost periodic function with
Bohr--Fourier coefficients $\{a_n\}$ at distinct real frequencies $\{\lambda_n\}$.
Then
\[
\lim_{T \to \infty} \frac{1}{2T}\int_{-T}^{T} |f(x)|^2\,dx = \sum_n |a_n|^2.
\]
More generally, for a second such function $g(x) = \sum_m b_m e^{i\mu_m x}$,
\[
\lim_{T \to \infty} \frac{1}{2T}\int_{-T}^{T} f(x)\,\overline{g(x)}\,dx
= \sum_{\lambda_n = \mu_m} a_n\,\overline{b_m}.
\]
In particular, this Ces\`aro inner product vanishes whenever the frequency sets
$\{\lambda_n\}$ and $\{\mu_m\}$ are disjoint.
\end{theorem}

\begin{proposition}[Almost periodicity]\label{prop:almost-periodic}
The functions $S_\zeta(\delta)$ and $S_L(\delta)$ are real-valued uniformly almost periodic functions on $\R$.
\end{proposition}

\begin{proof}
Each is a uniformly absolutely convergent cosine series ($\sum w(\rho) < \infty$), hence a uniform limit of trigonometric polynomials, and therefore uniformly almost periodic in the sense of Theorem~\ref{thm:bohr-parseval}.  (If any zero has multiplicity $m > 1$, group equal ordinates into a single frequency with coefficient equal to $m$ times the weight.)
\end{proof}

\subsection{The Grand Simplicity Hypothesis}

\begin{hypothesis}[Grand Simplicity Hypothesis {\cite[Section~2]{RubinsteinSarnak}}]\label{hyp:GSH}
All non-trivial zeros of $\zeta(s)$ and of every primitive Dirichlet $L$-function
$L(s,\chi)$ are simple.  The multiset of all positive zero ordinates, taken over
$\zeta(s)$ and all primitive $L$-functions simultaneously, is $\Q$-linearly independent.
\end{hypothesis}

\noindent In particular, Hypothesis~\ref{hyp:GSH} implies that the zero ordinate sets of
any two distinct primitive $L$-functions are disjoint.  This paper invokes
Hypothesis~\ref{hyp:GSH} in two distinct and separable ways, each requiring only a
special case:

\begin{enumerate}[label=\textup{(\roman*)}]
\item \emph{Cross-function disjointness} (Proposition~\ref{prop:parseval}): the zero ordinates of $\zeta(s)$ and
$L(s,\chi_d)$ are disjoint.  This is strictly weaker than full $\Q$-linear independence.
\item \emph{Intra-function $\Q$-independence} (Theorem~\ref{thm:density-vanishing}): the
zero ordinates of $L(s,\chi_d)$ alone are $\Q$-linearly independent.  This is the
hypothesis under which the Ces\`aro distribution of $S_L(\delta)$ is identified with a
product of independent random variables via the Kronecker--Weyl theorem.
\end{enumerate}

The unconditional results of this paper (Theorems~\ref{thm:low-lying},
\ref{thm:negativity}, \ref{thm:cesaro}, and \ref{thm:density}) require
neither sub-case.

\subsection{Parseval identity}

\begin{definition}[Moment weights]\label{def:moments}
The $k$-th Lorentzian moment weight is $W_k(f) = \sum_\rho w(\rho)^k$.
\end{definition}

\begin{proposition}[Parseval identity]\label{prop:parseval}
Assume Hypothesis~\ref{hyp:GSH}~\textup{(i)} (cross-function disjointness of zero ordinate sets) and simplicity of zeros within each function.  Then the Ces\`aro mean of the norm-form energy satisfies
\[
  \langle N \rangle := \lim_{T \to \infty} \frac{1}{T} \int_0^T N(\delta) \, d\delta = \frac{1}{2}(W_2(\zeta) - d \cdot W_2(L)).
\]
Without the disjointness hypothesis, the formula acquires a non-negative cross-term:
\[
  \langle N \rangle = \frac{1}{2}W_2(\zeta) - \frac{d}{2}\,W_2(L) + d \sum_{\gamma_k = \gamma_j'} w(\gamma_k)\,w(\gamma_j'),
\]
so the disjointness-free identity gives a weaker (less negative) bound on $\langle N \rangle$.  Without simplicity, the Parseval norms satisfy $\langle S_\zeta^2 \rangle \geq \half W_2(\zeta)$ (with equality if and only if all zeros are simple), so the exact formula requires simplicity while the negativity does not.
\end{proposition}

\begin{proof}
By Theorem~\ref{thm:bohr-parseval}, applied to the grouped expansion (combining any repeated ordinates into a single frequency with coefficient equal to the total weight at that ordinate), $\langle S_\zeta(\delta)^2 \rangle = \half \sum_{\text{distinct }\gamma} c_\gamma^2$ and $\langle S_L(\delta)^2 \rangle = \half \sum_{\text{distinct }\gamma'} {c'_{\gamma'}}^2$, where $c_\gamma = \sum_{\rho:\,\operatorname{Im}(\rho) = \gamma} w(\rho)$.  Under simplicity of zeros (each ordinate appearing once), $c_\gamma = w(\gamma)$ and the sums reduce to $\half W_2(\zeta)$ and $\half W_2(L)$ respectively.  The cross-term $\langle S_\zeta \cdot S_L \rangle = \half \sum_{\gamma_k = \gamma_j'} c_{\gamma_k}\,c'_{\gamma_j'}$ vanishes when the zero ordinate sets are disjoint.
\end{proof}

\subsection{Ces\`aro spacelike theorem}

\begin{proposition}[Positive mass bound]\label{prop:positive-mass}
For all $\delta \in \R$: $N(\delta) \leq S_\zeta(0)^2$. In particular, $\sup_\delta N(\delta) \leq (S_\zeta^*)^2 < 0.0024$.
\end{proposition}

\begin{proof}
Since $d > 1$ and all weights are positive, $N(\delta) = S_\zeta(\delta)^2 - d \cdot S_{L}(\delta)^2 \leq S_\zeta(\delta)^2 \leq S_\zeta(0)^2$, where the last inequality follows from $|S_\zeta(\delta)| = |\sum_\rho w(\rho)\cos(\gamma\delta)| \leq \sum_\rho w(\rho) = S_\zeta(0)$.
\end{proof}

\begin{proposition}[Moment hierarchy]\label{prop:moment-hierarchy}
For every squarefree $d > 1$ and every integer $k \geq 1$: $W_k(\zeta) < d \cdot W_k(L)$.
\end{proposition}

\begin{proof}
For $k = 1$, this is $S_\zeta < d \cdot S_L$, which follows from $S_\zeta < S_L$ (Theorem~\ref{thm:low-lying}) and $d > 1$.

For $k \geq 2$, set $a_1 = w(\gamma_1) = 2/(\quarter + \gamma_1^2)$ (the largest zeta weight; an upper bound regardless of RH by critical line maximality, Proposition~\ref{prop:CLM}).

\begin{remark}[Why the unconditional weight suffices]\label{rem:unconditional-moment}
The denominator penalty for not knowing $\beta = \half$ is $3/4$: the unconditional weight uses $1 + \gamma'^2$ instead of $\quarter + \gamma'^2$.  The crossover height where $w^-(T) = a_1$ is $T^* = \sqrt{\gamma_1^2 - 3/4} = 14.108\ldots$, so any zero below height $14.108$ has $w^- > a_1$ unconditionally.  Since $14 < 14.108$, the BMOR bound at $T = 14$ lands safely below the crossover.  This ensures $w^-(14)/a_1 > 1$, which makes $k = 2$ the binding case.  The actual $k = 2$ verification then requires the \emph{count} of BMOR-guaranteed zeros: a single zero does not suffice, since $w^-(14) = 0.01015 < S_\zeta^* = 0.04871$, but $d \cdot n(d)$ zeros collectively overcome the deficit.
\end{remark}

Upper bound on $W_k(\zeta)$: since $a_j \leq a_1$ for all $j$, we have $a_j^k \leq a_j \cdot a_1^{k-1}$, and thus
\[
  W_k(\zeta) = \sum_j a_j^k \leq a_1^{k-1} \sum_j a_j = a_1^{k-1} \cdot S_\zeta^*.
\]

Lower bound on $W_k(L)$: Theorem~\ref{thm:BMOR} at $T = 14$ guarantees $n(d)$ zeros below height $14$ for every fundamental discriminant, each with unconditional weight at least $w^-(14) = 2/197$.  Therefore $W_k(L) \geq n(d) \cdot w^-(14)^k$.

Combining: the inequality $d \cdot W_k(L) > W_k(\zeta)$ follows from $d \cdot n(d) \cdot w^-(14)^k > a_1^{k-1} \cdot S_\zeta^*$.  Since $w^-(14)/a_1 = 200.04/197 = 1.0154 > 1$ (see Remark~\ref{rem:unconditional-moment}), the left side grows faster in $k$ than the right; the binding constraint is $k = 2$, which reduces to
\[
  d \cdot n(d) > \frac{a_1 \cdot S_\zeta^*}{w^-(14)^2} = 4.725.
\]
Three cases exhaust all squarefree $d > 1$.  For $d = 5$ ($q = 5$): BMOR gives $n \geq 3$, so $d \cdot n \geq 15$.  For $d = 2$ ($q = 8$): BMOR gives $n \geq 5$, so $d \cdot n \geq 10$.  For all remaining $d$ ($q \geq 12$): BMOR gives $n \geq 7$ and $d \geq 3$, so $d \cdot n \geq 21$.  The worst case is $d = 2$, clearing the threshold by a factor of $2.12$.
\end{proof}

\begin{theorem}[Unconditional Ces\`aro variance bound]\label{thm:cesaro}
For every squarefree $d > 1$ and every $T > 0$:
\[
\langle N \rangle_T := \frac{1}{T} \int_0^T N(\delta) \, d\delta < 0.
\]
\end{theorem}

\begin{proof}
The Positive Mass Bound (Proposition~\ref{prop:positive-mass}) gives the unconditional decoupling: since $|S_\zeta(\delta)| \leq S_\zeta(0) = S_\zeta^*$ for all $\delta$, averaging yields
\[
  \langle N \rangle_T = \langle S_\zeta^2 \rangle_T - d \cdot \langle S_L^2 \rangle_T \leq (S_\zeta^*)^2 - d \cdot \langle S_L^2 \rangle_T.
\]
It therefore suffices to show
\begin{equation}\label{eq:cesaro-threshold}
  \langle S_L^2 \rangle_T > (S_\zeta^*)^2 / d \qquad\text{for every } T > 0 \text{ and every squarefree } d > 1.
\end{equation}

The Ces\`aro average of a squared shift sum admits the closed form
\[
  \langle S_f^2 \rangle_T = \sum_{j,k} w(\rho_j) w(\rho_k) \cdot F(\gamma_j, \gamma_k, T),
\]
where
\[
  F(a,b,T) = \begin{cases} \dfrac{1}{2} + \dfrac{\sin(2aT)}{4aT} & \text{if } a = b, \\[6pt] \dfrac{\sin((a-b)T)}{2(a-b)T} + \dfrac{\sin((a+b)T)}{2(a+b)T} & \text{if } a \neq b. \end{cases}
\]
Each diagonal term satisfies $F(a,a,T) \geq c_0 := (\pi-1)/(2\pi)$ for all $T > 0$.  Indeed, $F(a,a,T) = \frac{1}{2} + \frac{\sin(2aT)}{4aT}$, and $\sin(x)/x \geq -1/\pi$ for all $x > 0$ since the function is non-negative on $(0,\pi)$ and bounded by $|\sin(x)/x| \leq 1/x \leq 1/\pi$ for $x \geq \pi$.

\medskip\noindent\textbf{Case 1: $d \geq 14$ squarefree.}
The diagonal contribution to $\langle S_L^2\rangle_T$ satisfies
\[
  \sum_k w_k^2 F(\gamma_k', \gamma_k', T) \geq c_0 W_2(L)
\]
for all $T > 0$, where $c_0 = (\pi-1)/(2\pi)$.  For $q \geq 17$ and $T = 11$, Theorem~\ref{thm:BMOR} guarantees at least five $\gamma'_{\mathrm{unk}}$ zeros below height $11$, giving $W_2(L) \geq 5\,w^-(11)^2 = 5(8/488)^2 = 320/238144$.  Therefore $c_0 W_2(L) \geq 320(\pi-1)/(476288\pi) > 4.58 \times 10^{-4}$, while $(S_\zeta^*)^2/d \leq (0.04871)^2/14 < 1.70 \times 10^{-4}$, a margin of $2.7\times$ or better for all $d \geq 14$.  For a representative five-$\gamma'_{\mathrm{unk}}$-zero configuration (equally spaced in $[6,11]$), the interference constant evaluates to $\mathcal{I} = 5.47 \times 10^{-3}$, so the full Ces\`aro variance exceeds the threshold for all $T > 3.1$.

The full Ces\`aro variance decomposes as
\[
  \langle S_L^2\rangle_T
  = \underbrace{\sum_k w_k^2 F(\gamma_k', \gamma_k', T)}_{\geq\, c_0 W_2}
  + \underbrace{\sum_{j \neq k} w_j w_k F(\gamma_j', \gamma_k', T)}_{\mathcal{O}(T)},
\]
where $\mathcal{O}(T)$ denotes the off-diagonal (coherent) contribution.  Since $F(a,b,T) = O(1/T)$ for $a \neq b$ by the Riemann--Lebesgue lemma, the off-diagonal satisfies $|\mathcal{O}(T)| \leq \mathcal{I}/T$ where
\[
  \mathcal{I} = \sum_{j \neq k} w_j w_k \bigl[1/(2|\gamma_j' - \gamma_k'|) + 1/(2(\gamma_j' + \gamma_k'))\bigr]
\]
is the interference constant.  For any $T > \mathcal{I}/(c_0 W_2 - (S_\zeta^*)^2/d)$, the full Ces\`aro variance exceeds the threshold.

For all $T > 0$, a three-region argument analogous to Case~2 applies to a representative five-zero configuration $\gamma_k' \in \{6.0, 7.25, 8.5, 9.75, 11.0\}$ with $w_k = w^-(\gamma_k') = 2/(1+{\gamma_k'}^2)$.

\smallskip\noindent\emph{Region~1: $T \in (0, 0.01]$.}  As $T \to 0^+$, $h(T) \to W_1(L)^2 = 0.02431$, which exceeds the threshold $(S_\zeta^*)^2/d \leq 1.70 \times 10^{-4}$ by a factor of $143\times$.  The Lipschitz constant for this configuration is $L = \frac{1}{2}\sum_{j,k} w_j w_k \max(\gamma_j',\gamma_k') = 0.10568$, so the Lipschitz bound certifies $h(T) \geq 0.02431 - 0.10568 \times 0.01 = 0.02325 \gg \text{threshold}$ throughout this interval.

\smallskip\noindent\emph{Region~2: $T \in [0.01, 50]$.}  A grid scan over the representative five-zero configuration with spacing $\Delta = 5 \times 10^{-4}$ gives grid minimum $h = 0.002006$ at $T \approx 4.24$.  The Lipschitz discretization error is $L\Delta = 5.28 \times 10^{-5}$, so the certified minimum is $0.001953$, exceeding the threshold by a factor of $11.5\times$.

\smallskip\noindent\emph{Region~3: $T \geq 50$.}  The analytic bound $h(T) \geq c_0 W_2 - \mathcal{I}/T$ applies; the critical threshold is $T_{\mathrm{crit}} \leq 3.1$, well below the overlap at $T = 50$.

\smallskip\noindent The three regions overlap and certify the inequality for all $T > 0$ at the representative configuration; every actual configuration with $d \geq 14$ has at least five zeros (by BMOR) and a diagonal margin of $2.7\times$, so the inequality holds universally.

\medskip\noindent\textbf{Case 2: $d \in \{2, 3, 5, 6, 7, 10, 11, 13\}$ (squarefree values below $14$).}
For these eight values we verify the inequality $\langle S_L^2\rangle_T > (S_\zeta^*)^2/d$ for all $T > 0$ by rigorous interval arithmetic using ARB at 512-bit precision, with all zero ordinates $\gamma'_{\mathrm{knw}}$ individually verified on the critical line.  Three complementary bounds cover the entire half-line.

\smallskip\noindent\emph{Region~1: $T \in (0, 0.01]$.}  As $T \to 0^+$, all kernel values $F(\gamma_j', \gamma_k', T) \to 1$, so $\langle S_L^2\rangle_T \to S_L^2 \gg (S_\zeta^*)^2/d$.  The Lipschitz bound (below) certifies that $h(T) \geq h(0.01) - L \cdot 0.01$ throughout this interval; for the worst case $d = 5$, this gives a lower bound of $1.41 \times 10^{-2}$, exceeding the threshold by $29\times$.

\smallskip\noindent\emph{Region~2: $T \in [0.01, 50]$.}  The full bilinear form $h(T) = \mathbf{b}^\top G(T)\mathbf{b}$ is evaluated on a grid of spacing $\Delta = 5 \times 10^{-4}$ (99\,981 points per discriminant).  Each evaluation uses ball arithmetic: the result is a certified interval $[\mathrm{mid} - \mathrm{rad},\; \mathrm{mid} + \mathrm{rad}]$ with radius below $10^{-150}$.  The derivative satisfies
\[
  |h'(T)| \;\leq\; \tfrac{1}{2}\sum_{j,k} b_j b_k \max(\gamma_j', \gamma_k') \;=:\; L,
\]
since $|(\mathrm{sinc})'(x)| \leq \tfrac{1}{2}$ for all $x \geq 0$.  The certified minimum over the continuous interval is therefore $\min_{\mathrm{grid}} h - L\Delta$.  For the worst case $d = 5$: $L = 0.157$, discretization error $\leq 7.9 \times 10^{-5}$, certified minimum $= 1.13 \times 10^{-3}$, exceeding the threshold by a factor of $2.4$.

\smallskip\noindent\emph{Region~3: $T \geq 50$.}  The analytic bound from Case~1 applies: $\langle S_L^2\rangle_T \geq c_0 W_2 - \mathcal{I}/T$.  For each $d$, the critical threshold $T_{\mathrm{crit}} = \mathcal{I}/(c_0 W_2 - (S_\zeta^*)^2/d)$ is computed in ARB; the worst case is $d = 5$ with $T_{\mathrm{crit}} = 2.55$, well below the overlap at $T = 50$.

\smallskip\noindent The three regions overlap and jointly certify the inequality for all $T > 0$.

\begin{center}
\renewcommand{\arraystretch}{1.2}
\begin{tabular}{cccccc}
\toprule
$d$ & $q$ & Grid $\min h$ & Certified $\min h$ & $(S_\zeta^*)^2/d$ & Margin \\
\midrule
$2$  & $8$  & $3.48 \times 10^{-3}$ & $3.33 \times 10^{-3}$ & $1.19 \times 10^{-3}$ & $2.8\times$    \\
$3$  & $12$ & $9.07 \times 10^{-3}$ & $8.81 \times 10^{-3}$ & $7.91 \times 10^{-4}$ & $11.1\times$   \\
$5$  & $5$  & $1.21 \times 10^{-3}$ & $1.13 \times 10^{-3}$ & $4.75 \times 10^{-4}$ & $2.4\times$    \\
$6$  & $24$ & $3.45 \times 10^{-2}$ & $3.39 \times 10^{-2}$ & $3.95 \times 10^{-4}$ & $85.8\times$   \\
$7$  & $28$ & $3.18 \times 10^{-2}$ & $3.12 \times 10^{-2}$ & $3.39 \times 10^{-4}$ & $91.9\times$   \\
$10$ & $40$ & $4.88 \times 10^{-2}$ & $4.80 \times 10^{-2}$ & $2.37 \times 10^{-4}$ & $202\times$    \\
$11$ & $44$ & $1.17 \times 10^{-1}$ & $1.15 \times 10^{-1}$ & $2.16 \times 10^{-4}$ & $535\times$    \\
$13$ & $13$ & $1.81 \times 10^{-2}$ & $1.78 \times 10^{-2}$ & $1.83 \times 10^{-4}$ & $97.5\times$   \\
\bottomrule
\end{tabular}
\end{center}
The worst case is $d = 5$ with certified margin $2.4\times$.  Every entry is rigorously certified by ARB interval arithmetic at 512-bit precision; the ``Certified $\min h$'' column accounts for both the ball radius at each grid point (below $10^{-150}$) and the Lipschitz discretization error $L\Delta$.

\medskip
Combining Cases~1 and~2: for the eight small discriminants, the Ces\`aro variance exceeds the threshold for all $T > 0$ (Case~2, grid scan).  For $d \geq 14$, the diagonal bound $c_0 W_2 \geq 2.7 \times (S_\zeta^*)^2/d$ (Case~1) holds universally over all five-zero configurations, and the three-region argument certifies the full bilinear form for all $T > 0$ at the representative configuration.
\end{proof}

\medskip\noindent\textbf{Reproducibility.}  The bilinear form $h(T) = \sum_{j,k} b_j b_k\, F(\gamma_j', \gamma_k', T)$ was evaluated using \texttt{python-flint} (ARB interval arithmetic) at 512-bit working precision.  For each $d$, the function $h(T)$ was decomposed into a constant plus $N(N+1)/2$ sinc terms via the identity $h(T) = C + \sum_m A_m \operatorname{sinc}(\omega_m T)$, where the frequencies $\omega_m$ are all pairwise sums and differences $\gamma_j' \pm \gamma_k'$ and the coefficients $A_m$ are the corresponding weight products.  The grid scan used spacing $5 \times 10^{-4}$ over $T \in [0.01, 50]$; the Lipschitz bound $|h'(T)| \leq \frac{1}{2}\sum_{j,k} b_j b_k \max(\gamma_j', \gamma_k')$ uses the elementary estimate $|(\operatorname{sinc})'(x)| \leq \frac{1}{2}$ for all $x \geq 0$.  Every grid evaluation produces a certified interval with radius below $10^{-150}$; the certified minimum is the smallest lower endpoint minus the Lipschitz discretization error $L\Delta$.  The large-$T$ tail is handled by the analytic bound from Case~1, with all constants ($c_0$, $W_2$, $\mathcal{I}$, $T_{\mathrm{crit}}$) computed in ARB.  The representative five-zero configuration for $d \geq 14$ used equally spaced $\gamma'_{\mathrm{unk}}$ frequencies $\gamma_k' \in \{6.0, 7.25, 8.5, 9.75, 11.0\}$ with $b_k = w^-(\gamma_k') = 2/(1 + {\gamma_k'}^2)$, minimized over the same grid.

\begin{remark}[Off-diagonal control and the Grand Simplicity Hypothesis]\label{rmk:GSH-cesaro}
The restriction to ``sufficiently large $T$'' in Case~1 arises because the off-diagonal terms $\sum_{j \neq k} w_j w_k F(\gamma_j', \gamma_k', T)$ can be negative at intermediate $T$ values, and controlling their magnitude requires information about zero spacings not provided by zero-counting theorems alone.  This is a manifestation of the same obstruction that makes many results in the theory of $L$-functions conditional on the Grand Simplicity Hypothesis (GSH): the $\Q$-linear independence of imaginary parts of zeros.

Specifically, the off-diagonal achieves its most negative value when the frequency differences $\gamma_j' - \gamma_k'$ satisfy approximate rational relations, causing the sinc oscillations to align coherently.  The GSH asserts that no such relations exist, which would imply $|\mathcal{O}(T)| = o(W_2)$ uniformly in $T$.  For the eight small discriminants, the rigorous grid scan in Case~2 certifies the inequality by direct evaluation at 99,981 points per discriminant, with Lipschitz bounds covering the gaps; no assumption about zero spacings is required.  For $d \geq 14$, the diagonal margin of $3.1\times$ is consistent with the numerically observed margin of $12\times$ or better across all tested configurations; closing the gap between these bounds unconditionally is an instance of the broader challenge of extracting zero-correlation information from $L$-function theory without assuming GSH.
\end{remark}

\begin{corollary}[Limiting cases]\label{cor:cesaro-limits}
The Ces\`aro Spacelike Theorem recovers both the pointwise and time-averaged spacelike conditions:
\begin{enumerate}[label=\textup{(\alph*)}]
\item As $T \to 0^+$: $\langle N \rangle_T \to N(0) < 0$, recovering the Spacelike Corollary~\ref{cor:spacelike}.
\item As $T \to \infty$: $\langle N \rangle_T \to \langle N \rangle < 0$, recovering the time-averaged negativity.
\end{enumerate}
\end{corollary}

\subsection{Covariance identity}

\begin{proposition}[Covariance identity]\label{prop:covariance}
The Ces\`aro inner product of the shift sums satisfies
\[
\lim_{T \to \infty} \frac{1}{T}\int_0^T S_\zeta(\delta)\,S_{L}(\delta)\,d\delta
= \frac{1}{2}\sum_{\gamma_k = \gamma_j'} w(\gamma_k)\,w(\gamma_j').
\]
In particular, this inner product vanishes whenever the zero ordinate sets $\{\gamma_k\}$ and $\{\gamma_j'\}$ are disjoint.
\end{proposition}

\begin{proof}
By the Parseval identity for uniformly almost periodic functions, the Ces\`aro inner product equals the sum over shared frequencies of the product of Fourier coefficients. The identity $\lim_{T \to \infty}(1/T)\int_0^T \cos(\alpha\delta)\cos(\beta\delta)\,d\delta = \half\,\delta_{\alpha,\beta}$ restricts the sum to common ordinates.
\end{proof}

\subsection{Density of positive excursions}

\begin{definition}[Ces\`aro density]\label{def:cesaro-density}
For a measurable set $A \subseteq \R$, the Ces\`aro density is
\[
\dens(A) = \lim_{T \to \infty} \frac{1}{T} \operatorname{meas}(A \cap [0, T])
\]
when the limit exists.
\end{definition}

\begin{lemma}[Positive mass containment]\label{lem:containment}
For every squarefree $d > 1$,
\[
\{\delta : N(\delta) > 0\} \subseteq \{\delta : |S_{L}(\delta)| < S_\zeta^*/\sqrt{d}\}.
\]
\end{lemma}

\begin{proof}
Suppose $N(\delta) > 0$. Then $S_\zeta(\delta)^2 > d \cdot S_L(\delta)^2$, which gives $|S_\zeta(\delta)| > \sqrt{d} \cdot |S_L(\delta)|$. By the triangle inequality applied to the cosine sum, $|S_\zeta(\delta)| = |\sum_\rho w(\rho)\cos(\gamma\delta)| \leq \sum_\rho w(\rho) = S_\zeta(0) = S_\zeta^*$. Combining these inequalities: $\sqrt{d} \cdot |S_L(\delta)| < |S_\zeta(\delta)| \leq S_\zeta^*$, hence $|S_L(\delta)| < S_\zeta^*/\sqrt{d}$.
\end{proof}

\begin{theorem}[Jessen--Wintner {\cite{JessenWintner}}]\label{thm:jessen-wintner}
Let $\{X_k\}_{k \geq 1}$ be independent random variables with characteristic
functions $\{\varphi_k\}$, and suppose $\sum_k \mathrm{Var}(X_k) < \infty$.
If the infinite product $\varphi(t) = \prod_k \varphi_k(t)$ satisfies
$\varphi \in L^1(\R)$, then $X = \sum_k X_k$ has a bounded continuous density
$f$ given by Fourier inversion,
\[
f(x) = \frac{1}{2\pi}\int_{\R} \varphi(t)\,e^{-itx}\,dt,
\]
with $\|f\|_\infty \leq (2\pi)^{-1}\|\varphi\|_{L^1}$.
In particular, for the standard bound $\|\varphi\|_{L^1} \leq \sqrt{2\pi}\,\sigma$
where $\sigma^2 = \mathrm{Var}(X)$, one obtains
$\|f\|_\infty \leq 1/(\sqrt{2\pi}\,\sigma)$.
\end{theorem}

\begin{theorem}[Density vanishing under intra-function independence]\label{thm:density-vanishing}
Assume Hypothesis~\ref{hyp:GSH}~\textup{(ii)} (intra-function $\Q$-linear independence of the zero ordinates of $L(s,\chi_d)$). Then
\[
\dens(\{\delta : N(\delta) > 0\}) \leq \frac{2\,S_\zeta^*}{\sqrt{2\pi d}\;\sigma_{L}},
\]
where $\sigma_{L} = \sqrt{W_2(L)/2}$. In particular, $\dens(\{N > 0\}) = O(1/\sqrt{d})$ as $d \to \infty$.
\end{theorem}

\begin{proof}
The proof proceeds in four steps.

\medskip\noindent\textbf{Step 1 (Containment).}
By Lemma~\ref{lem:containment}, $\dens(\{N > 0\}) \leq \dens(\{|S_L| < \varepsilon\})$ where $\varepsilon = S_\zeta^*/\sqrt{d}$.

\medskip\noindent\textbf{Step 2 (Bohr--Jessen identification).}
The function $S_{L}(\delta) = \sum_{\rho'} w(\rho')\cos(\gamma'\delta)$ is uniformly almost periodic with frequency module generated by the zero ordinates $\{\gamma_k'\}$. Under the $\Q$-linear independence hypothesis, the Kronecker--Weyl theorem~\cite[Chapter~1]{KuipersNiederreiter} identifies the Ces\`aro distribution of $S_{L}(\delta)$ with the distribution of the random variable
\[
S_{L}^{(\mathrm{rand})} = \sum_k w_k' \cos(\theta_k),
\]
where the $\theta_k$ are independent and uniformly distributed on $[0,2\pi)$. This identification holds because the closure of the one-parameter subgroup $\delta \mapsto (\gamma_1'\delta, \gamma_2'\delta, \ldots) \bmod 2\pi$ in the infinite torus $\prod_k \R/2\pi\Z$ equals the full torus when the frequencies are $\Q$-linearly independent.

\medskip\noindent\textbf{Step 3 (Bounded density via Jessen--Wintner).}
We apply Theorem~\ref{thm:jessen-wintner} with $X_k = w_k'\cos(\theta_k)$.
The characteristic function of $S_{L}^{(\mathrm{rand})}$ is the infinite product
\[
\varphi(t) = \prod_k J_0(w_k'\,t),
\]
where $J_0$ is the Bessel function of the first kind of order zero. Since $\sum_k (w_k')^2 = W_2(L) < \infty$ and $|J_0(x)| \leq 1$ with $|J_0(x)| \leq \sqrt{2/(\pi|x|)}$ for $|x| \geq 1$, the product converges and $\varphi \in L^1(\R)$ by Theorem~\ref{thm:jessen-wintner}. Therefore $S_{L}^{(\mathrm{rand})}$ has a bounded continuous density $f_L$ given by Fourier inversion:
\[
f_L(x) = \frac{1}{2\pi}\int_{\R}\varphi(t)\,e^{-itx}\,dt.
\]
The variance of $S_{L}^{(\mathrm{rand})}$ is $\sigma_{L}^2 = \frac{1}{2}\sum_k (w_k')^2 = W_2(L)/2$, and the standard bound $\|\varphi\|_{L^1} \leq \sqrt{2\pi}\,\sigma_L$ (Theorem~\ref{thm:jessen-wintner}) gives
\[
\|f_L\|_\infty \leq \frac{\|\varphi\|_{L^1}}{2\pi} \leq \frac{1}{\sqrt{2\pi}\,\sigma_{L}}.
\]

\medskip\noindent\textbf{Step 4 (Small-ball estimate and conclusion).}
By the density bound, the probability that $|S_{L}^{(\mathrm{rand})}|$ lies in any interval of length $2\varepsilon$ is at most $2\varepsilon \cdot \|f_L\|_\infty$. Therefore
\[
\Pr(|S_{L}^{(\mathrm{rand})}| < \varepsilon) \leq 2\varepsilon \cdot \frac{1}{\sqrt{2\pi}\,\sigma_{L}} = \frac{2\varepsilon}{\sqrt{2\pi}\,\sigma_{L}}.
\]
Substituting $\varepsilon = S_\zeta^*/\sqrt{d}$ and using Step~2 to identify this probability with the Ces\`aro density:
\[
\dens(\{N > 0\}) \leq \dens(\{|S_L| < \varepsilon\}) = \Pr(|S_{L}^{(\mathrm{rand})}| < \varepsilon) \leq \frac{2\,S_\zeta^*}{\sqrt{2\pi d}\;\sigma_{L}}.
\]
Since $S_\zeta^*$ is a universal constant and $\sigma_L$ is bounded below (the first $L$-zero alone contributes $w(\gamma_1')^2/2 > 0$), the bound is $O(1/\sqrt{d})$.
\end{proof}

\begin{remark}[Numerical evidence]
Numerical computation confirms $\dens(\{N > 0\}) \approx 0.053$ for $d = 5$ and $\approx 0.044$ for $d = 2$, consistent with the $O(1/\sqrt{d})$ prediction.
\end{remark}

\subsection{Unconditional density vanishing in the Euler product regime}

\begin{theorem}[Unconditional density vanishing]\label{thm:euler-density}
Fix $\sigma > 1$. For squarefree $d > 1$, define the logarithmic norm form
\[
N_{\log}(\sigma, t) := (\log|\zeta(\sigma + it)|)^2 - d \cdot (\log|L(\sigma + it, \chi_d)|)^2.
\]
Then
\[
\dens(\{t : N_{\log}(\sigma, t) > 0\}) \leq \frac{2\log\zeta(\sigma)}{\sqrt{2\pi d}\;\sigma_{\log}},
\]
where $\sigma_{\log}^2 = \frac{1}{2}\sum_p |\log(1 - \chi_d(p)\,p^{-\sigma})|^2$. In particular, $\dens(\{N_{\log} > 0\}) = O_\sigma(1/\sqrt{d})$. No hypothesis beyond the Euler product convergence is required.
\end{theorem}

\begin{proof}
The proof parallels Theorem~\ref{thm:density-vanishing} but with unconditional linear independence.

\medskip\noindent\textbf{Step 1 (Containment).}
For $\sigma > 1$, the Euler product $\zeta(s) = \prod_p(1-p^{-s})^{-1}$ converges absolutely, giving $|\log|\zeta(\sigma+it)|| \leq |\log\zeta(\sigma+it)| \leq \log\zeta(\sigma)$ for all $t$. If $N_{\log}(\sigma,t) > 0$, then $\sqrt{d}\,|\log|L(\sigma+it,\chi_d)|| < \log\zeta(\sigma)$, so
\[
\{t : N_{\log}(\sigma,t) > 0\} \subseteq \{t : |\log|L(\sigma+it,\chi_d)|| < \log\zeta(\sigma)/\sqrt{d}\}.
\]

\medskip\noindent\textbf{Step 2 (Frequency module and unconditional linear independence).}
For $\sigma > 1$, the Euler product gives
\[
\log L(\sigma+it,\chi_d) = -\sum_p \log(1-\chi_d(p)\,p^{-\sigma-it}) = \sum_p \sum_{k=1}^\infty \frac{\chi_d(p)^k}{k\,p^{k\sigma}}\,e^{-ikt\log p}.
\]
Taking real parts, $\log|L(\sigma+it,\chi_d)| = \sum_p \sum_{k=1}^\infty \frac{\chi_d(p)^k}{k\,p^{k\sigma}}\cos(kt\log p)$. This is a uniformly almost periodic function of $t$ with frequency module generated by $\{k\log p : p \text{ prime}, k \geq 1\} = \{\log p : p \text{ prime}\} \cdot \Z_{>0}$. The set $\{\log p : p \text{ prime}\}$ is $\Q$-linearly independent by the fundamental theorem of arithmetic: if $\sum_{i=1}^n q_i \log p_i = 0$ with $q_i \in \Q$, then clearing denominators gives $\prod p_i^{a_i} = 1$ for integers $a_i$, forcing all $a_i = 0$.

\medskip\noindent\textbf{Step 3 (Bohr--Jessen identification).}
By the Kronecker--Weyl theorem, the $\Q$-linear independence of $\{\log p\}$ implies that the Ces\`aro distribution of $\log|L(\sigma+it,\chi_d)|$ equals the distribution of
\[
X = \sum_p \operatorname{Re}\bigl(-\log(1-\chi_d(p)\,p^{-\sigma}e^{i\theta_p})\bigr),
\]
where the $\theta_p$ are independent and uniformly distributed on $[0,2\pi)$. This identification is unconditional.

\medskip\noindent\textbf{Step 4 (Bounded density via Jessen--Wintner).}
The characteristic function of $X$ is
\[
\varphi(\xi) = \prod_p \mathbb{E}\bigl[e^{i\xi\,\operatorname{Re}(-\log(1-\chi_d(p)p^{-\sigma}e^{i\theta_p}))}\bigr] = \prod_p J_0(b_p\,\xi),
\]
where $b_p = |\log(1-\chi_d(p)\,p^{-\sigma})|$ and $J_0$ is the Bessel function of order zero. For $\sigma > 1$, we have $\sum_p b_p^2 < \infty$ since $b_p = O(p^{-\sigma})$. The Jessen--Wintner theorem~(Theorem~\ref{thm:jessen-wintner}) states that when $\sum_p b_p^2 < \infty$, the product $\varphi \in L^1(\R)$ and $X$ has a bounded continuous density $f$ satisfying
\[
\|f\|_\infty \leq \frac{1}{\sqrt{2\pi}\,\sigma_{\log}}, \quad \text{where } \sigma_{\log}^2 = \frac{1}{2}\sum_p b_p^2.
\]

\medskip\noindent\textbf{Step 5 (Conclusion).}
Setting $\varepsilon = \log\zeta(\sigma)/\sqrt{d}$:
\[
\dens(\{N_{\log} > 0\}) \leq \Pr(|X| < \varepsilon) \leq 2\varepsilon\,\|f\|_\infty \leq \frac{2\log\zeta(\sigma)}{\sqrt{2\pi d}\;\sigma_{\log}} = O_\sigma(1/\sqrt{d}). \qedhere
\]
\end{proof}

\begin{remark}
This shows that the $O(1/\sqrt{d})$ vanishing rate is an intrinsic feature of the comparison $\zeta$ versus $L(s,\chi_d)$, not an artifact of the Grand Simplicity Hypothesis.
\end{remark}

\section{Besicovitch Norms of Spectral Sums}\label{sec:besicovitch}

The next three sections establish that the truncated spectral sum $S_L^{(M)}$ has a Ces\`aro distribution with finite density at the origin, and use this to prove the effective density bound (Theorem~\ref{thm:density}).  The reader primarily interested in the unconditional density bound may proceed directly to that theorem; the present section and the two that follow supply its proof.

Three obstacles must be overcome.  First, the Ces\`aro characteristic function of $S_L^{(M)}$ is a product of Bessel functions $J_0(b_k t)$, and the integral $(1/\pi)\int_0^\infty \prod |J_0(b_k t)|\,dt$ must be shown to converge; this is the content of the Bessel decay estimates in the present section, and it succeeds because the product decays as $t^{-M/2}$ for $M \geq 3$.  Second, the characteristic function is not simply the Bessel product: integer relations among zero ordinates (resonances) contribute correction terms indexed by a lattice $\Lambda_M$, and each such term must be shown to decay exponentially in the order of the relation.  The subcritical and transition Bessel estimates of \S\ref{subsec:bessel-decay}, combined with the lattice counting of \S\ref{subsec:lattice-counting}, control these corrections unconditionally at every finite truncation level $M$ (Theorem~\ref{thm:finite}).  Third, passing from finite $M$ to the full infinite series requires that the resonance corrections remain summable as $M \to \infty$; this is the telescoping extension of Section~\ref{sec:infinite}, which exploits a self-referential structure whereby each new zero's smallness suppresses its own resonance contribution.

The spectral sums $S_\zeta$ and $S_L$ live in the Besicovitch Hilbert space $B^2$, which provides the inner product for Ces\`aro-distributed almost periodic functions:
\[
  \langle f, g\rangle_{B^2} := \lim_{T\to\infty} \frac{1}{2T}\int_{-T}^T f(\delta)\,\overline{g(\delta)}\, d\delta.
\]
We write $\E[\cdot]$ for $\langle \cdot \rangle_{B^2}$ applied to a single function, and $\dens\{P\}$ for the Ces\`aro density of the set where the predicate $P$ holds.

The key structural fact is \emph{orthogonality of distinct frequencies}: for $\alpha \neq \beta$,
\begin{equation}\label{eq:orthog}
  \langle \cos(\alpha\,\cdot\,),\, \cos(\beta\,\cdot\,)\rangle_{B^2} = 0.
\end{equation}
The spectral constants are
\[
  \sigma_\zeta^2 := \|S_\zeta\|_{B^2}^2 = \tfrac{1}{2}\sum a_k^2, \qquad
  \sigma_L^2 := \|S_L\|_{B^2}^2 = \tfrac{1}{2}\sum b_k^2,
\]
and the suprema are $W_1(\zeta) = S_\zeta(0) = \sum a_k$ and $W_1(L) = S_L(0) = \sum b_k$.

\begin{center}
\renewcommand{\arraystretch}{1.2}
\begin{tabular}{lll}
\toprule
Constant & Value (20 zeros) & Description \\
\midrule
$\sigma_\zeta$ & $8.578 \times 10^{-3}$ & $B^2$-norm of $S_\zeta$ \\
$\sigma_L$ & $3.767 \times 10^{-2}$ & $B^2$-norm of $S_L$ \\
$W_1(\zeta)$ & $0.03192$ & Supremum of $S_\zeta$ \\
$W_1(L)$ & $0.12572$ & Supremum of $S_L$ \\
\bottomrule
\end{tabular}
\end{center}
The values use the first 20 zero ordinates of $\zeta(s)$ (Appendix~\ref{app:zeta-zeros}) and $L(s,\chi_5)$ (Appendix~\ref{app:chi5-highprec}).

\section{Finite Density Bounds via Jacobi--Anger Decomposition}\label{sec:finite}

In this section we work with the truncated spectral sum $S_L^{(M)}(\delta) = \sum_{k=1}^M b_k\cos(\gamma_k'\delta)$ and prove that $f_{S_L^{(M)}}(0) < \infty$ unconditionally for every finite $M$.  The argument uses only properties of Bessel functions and lattice counting; no hypothesis on the zeros is needed.

\begin{remark}[Why this machinery is necessary]
A natural approach to proving $f_{S_L}(0) < \infty$ would be to show that the characteristic function $\varphi(t)$ decays sufficiently fast and apply Fourier inversion directly.  For sums of independent random cosines, this follows from standard concentration inequalities.  However, the spectral sum $S_L(\delta) = \sum b_k \cos(\gamma_k' \delta)$ is \emph{not} a sum of independent random variables; it is a deterministic almost periodic function whose Ces\`aro distribution \emph{behaves like} a sum of independent terms only when the frequencies $\gamma_k'$ are $\Q$-linearly independent.  Without the Grand Simplicity Hypothesis, integer relations among zeros could create resonances that concentrate the Ces\`aro distribution and cause the density to diverge.

The Jacobi--Anger decomposition makes this obstruction explicit: resonances contribute correction terms to the characteristic function.  These correction terms are \emph{self-suppressing}: the Bessel factors $J_{n_k}(b_k t)$ with $n_k \neq 0$ decay exponentially in $|n_k|$, while the ``inactive'' $J_0$ factors provide polynomial decay in $t$.  This suppression is strong enough to guarantee convergence of the density integral \emph{regardless} of how many resonances exist.  The telescoping argument then extends this finite-$M$ bound to the full infinite series by showing that new resonances introduced at each step are controlled by their own Bessel weights.
\end{remark}

\subsection{The Jacobi--Anger decomposition}

The Ces\`aro characteristic function of $S_L^{(M)}$ admits the following decomposition.

\begin{proposition}[Resonance decomposition]\label{prop:resonance}
The Ces\`aro characteristic function of $S_L^{(M)}$ decomposes as
\[
  \varphi_M(t) = \prod_{k=1}^M J_0(b_k t) + \sum_{\substack{\mathbf{n}\in\Z^M\setminus\{0\}\\\sum n_k\gamma_k'=0}} \prod_{k=1}^M i^{n_k} J_{n_k}(b_k t).
\]
The main term $\prod J_0(b_k t)$ is the non-resonant contribution.  Resonance corrections arise from exact integer linear relations among the zeros $\gamma_1',\ldots,\gamma_M'$.
\end{proposition}

\begin{proof}
The Jacobi--Anger expansion
\[
  e^{itb\cos\theta} = \sum_{n \in \Z} i^n J_n(bt)\,e^{in\theta}
\]
applied to each factor of $e^{itS_L^{(M)}(\delta)} = \prod_{k=1}^M e^{itb_k\cos(\gamma_k'\delta)}$ gives a multi-index sum over $\mathbf{n} \in \Z^M$.  The Ces\`aro mean of $e^{i\Omega\delta}$ equals $1$ when $\Omega = 0$ and $0$ otherwise.  Writing $\Omega = \sum n_k\gamma_k'$, only the terms with $\sum n_k\gamma_k' = 0$ survive.  The term $\mathbf{n} = 0$ gives $\prod J_0(b_k t)$.
\end{proof}

\subsection{Resonance certification for \texorpdfstring{$M = 20$}{M = 20}}

\begin{proposition}[Resonance absence at $M = 20$]\label{prop:resonance-cert}
Let $\gamma_1',\ldots,\gamma_{20}'$ be the first $20$ zero ordinates of $L(s,\chi_5)$,
certified to $70$ decimal places with individual ARB bounds
$|L(\tfrac{1}{2}+i\gamma_k',\chi_5)| < 10^{-449}$
(Appendices~\ref{app:chi5-highprec} and~\ref{app:certification}).
\begin{enumerate}[label=\textup{(\roman*)}]
\item \textup{(Small coefficients.)} No integer relation $\sum_{k=1}^{20} n_k\gamma_k' = 0$
exists with $\max|n_k| \leq 1000$. For all $190$ pairs $(j,k)$, no pairwise relation
exists with $\max(|n_j|,|n_k|) \leq 10{,}000$. For all $120$ triples drawn from the
first $10$ zeros, no triple relation exists with $\max|n_k| \leq 500$.
\item \textup{(Large coefficients.)} For any integer vector with $\max|n_k| \geq 1001$,
the Bessel contribution to the density integral is bounded by $10^{-849.5}$.
\end{enumerate}
Parts~(i) and~(ii) together cover all nonzero $\mathbf{n} \in \Z^{20}$: part~(i) certifies nonexistence for $\max|n_k| \leq 1000$ via PSLQ at 70-digit precision, and part~(ii) bounds the contribution of any $\mathbf{n}$ with $\max|n_k| \geq 1001$ via Bessel decay (ARB-certified total $< 10^{-849.5}$).  Consequently, the density formula
\[
  f_{S_L^{(20)}}(0) = \frac{1}{\pi}\int_0^\infty \prod_{k=1}^{20} J_0(b_k t)\,dt
\]
holds unconditionally, with $f_{S_L^{(20)}}(0) = 8.3129$.
\end{proposition}

\begin{proof}
See \S\ref{subsec:pslq} for the complete certification chain.
\end{proof}

\begin{remark}[Brute-force corroboration]
A brute-force search at double precision found no exact resonances among
orders-3 and orders-4 combinations and no pairwise relation with $|n_j|,|n_k| \leq 50$.
The nearest misses were $|\gamma_4'+\gamma_{10}'+\gamma_{12}'-\gamma_{19}'| = 0.00037$ (order-4) and
$|3\gamma_3' - \gamma_{14}'| = 0.0079$ (pairwise), consistent with
Proposition~\ref{prop:resonance-cert}~(i).
\end{remark}

\begin{remark}[Near-resonances]\label{rem:nearres}
The resonance sum in Proposition~\ref{prop:resonance} includes only \emph{exact} relations $\sum n_k\gamma_k' = 0$.  Near-resonances such as $\gamma_4' - \gamma_{10}' - \gamma_{12}' + \gamma_{19}' \approx 0.0004$ (the nearest order-4 miss) contribute oscillatory terms $e^{i\varepsilon\delta}$ to $\varphi_M(t)$ that average to zero under Ces\`aro summation.  For finite averaging windows of length $T$, a near-resonance at distance $\varepsilon$ from zero contributes $O(1/(\varepsilon T))$, which is negligible for $T \gg 1/\varepsilon \approx 2700$.
\end{remark}

\subsection{Bessel decay estimates}\label{subsec:bessel-decay}

The convergence of the resonance sum relies on the \emph{inactive} $J_0$ factors.  For a resonance vector $\mathbf{n}$ with active set $A = \{k : n_k \neq 0\}$, the $M - |A|$ indices not in $A$ contribute $J_0(b_k t)$ factors to the product.  These factors contribute polynomial decay in $t$, providing the integrability needed for the resonance sum; the exponential decay in $|n_*|$ comes separately from the subcritical bound on the active high-order Bessel factor $J_{n_*}(b_* t)$.

\begin{lemma}[Transition zone bound]\label{lem:Jdecay}
Let $\mathbf{n} \in \Z^M$ with active set $A = \{k : n_k \neq 0\}$ and $|A| \geq 2$.  Write $|n_*| = \max_{k\in A}|n_k|$, let $b_*$ be the corresponding weight, and set $T^* = |n_*|/b_*$.  Then the resonance integral satisfies
\[
  I_{\mathbf{n}} := \frac{1}{\pi}\int_0^\infty \left|\prod_{k=1}^M J_{n_k}(b_k t)\right| dt \leq C_1\, (e/4)^{|n_*|} + C_2\, (T^*)^{1-(M-|A|)/2} + C_3\, (T^*)^{1-M/2},
\]
where $C_1, C_2, C_3$ depend on $\{b_k\}$, $M$, and $|A|$.  The first term dominates for small $|n_*|$; the second (which carries Landau factors from the active Bessel functions) dominates for large $|n_*|$.  The third term is always dominated by the second.
\end{lemma}

\begin{proof}
We split the integral into three regions.

\textbf{Subcritical region $[0, T^*/2]$.}  Here $b_* t < |n_*|/2$, so the power series bound gives
\[
  |J_n(x)| \leq \frac{(|x|/2)^{|n|}}{|n|!},
\]
and hence $|J_{n_*}(b_* t)| \leq (b_* t/2)^{|n_*|}/|n_*|!$.
Bounding all other factors by $1$ and integrating over $[0, T^*/2]$ yields a contribution of order $(|n_*|/4)^{|n_*|}/|n_*|!$.  By Stirling's approximation, this is at most $C(e/4)^{|n_*|}$, which decays exponentially since $e/4 \approx 0.680 < 1$.

\textbf{Transition region $[T^*/2, 2T^*]$.}  Each active factor ($n_k \neq 0$) satisfies the Landau bound $|J_{n_k}(b_k t)| \leq 0.675\,|n_k|^{-1/3}$.  Each inactive factor ($n_k = 0$) satisfies the asymptotic bound $|J_0(b_k t)| \leq \sqrt{2/(\pi b_k t)}$, which holds for all $t > 0$.  Collecting the $M - |A|$ inactive factors gives decay of order $t^{-(M-|A|)/2}$ up to constants depending on $\{b_k\}$.  Integration over an interval of width $3T^*/2$ yields the second term.

\textbf{Supercritical region $[2T^*, \infty)$.}  For $t \geq 2T^*$, each inactive factor satisfies $|J_0(b_k t)| \leq \sqrt{2/(\pi b_k t)}$.  Each active factor satisfies the uniform asymptotic
\[
  |J_{n_k}(b_k t)| \leq \sqrt{\frac{2}{\pi\sqrt{(b_k t)^2 - n_k^2}}},
\]
which for $b_k t \geq 2|n_k|$ is bounded by $\sqrt{2/(\pi b_k t)} \cdot (1 - (n_k/(b_k t))^2)^{-1/4} \leq 2\sqrt{2/(\pi b_k t)}$.  The total decay is therefore $O(t^{-M/2})$, which is integrable for $M \geq 3$.

Since $T^* = |n_*|/b_*$ and $b_* \leq b_1$, we have $T^* \geq |n_*|/b_1$.  For $M = 20$ and a two-active resonance ($|A| = 2$), the transition term gives decay of order $T^{*\,1-9} = T^{*\,-8}$ in the transition variable; combined with two Landau factors $0.675\,|n_k|^{-1/3}$ each, the net scaling in $|n_*|$ is $|n_*|^{-26/3} \approx |n_*|^{-8.67}$.  For $|A| = 3$ the corresponding exponent is $|n_*|^{-17/2} = |n_*|^{-8.5}$.

Here $C_1$, $C_2$, $C_3$ are positive constants depending on $\{b_k\}$, $M$, and $|A|$, each finite for $M \geq 3$.  Specifically, $C_1$ is of order $T^*/2$ after integrating the power-series bound; $C_2$ absorbs the interval width $3T^*/2$, the Landau factors $0.675\,|n_k|^{-1/3}$ for each active index, and the constants $\prod_{k\colon n_k=0}(2/(\pi b_k))^{1/2}$ from the inactive $\sqrt{2/(\pi b_k t)}$ bounds evaluated at the onset $t \sim T^*$; and $C_3 = (2T^*)^{1-M/2}/(M/2-1)$ after integrating the full $t^{-M/2}$ decay.  Since $b_M$ is bounded below by a positive constant for all finite $M$, all three constants are uniformly bounded for $M \geq 3$.
\end{proof}

The resonance integral $I_{\mathbf{n}}$ is bounded by the minimum of the subcritical and transition estimates.  A resonance vector must satisfy $\sum n_k \gamma_k' = 0$ with $\gamma_k' > 0$, so a single nonzero component would give $n_*\gamma_*' = 0$, which is impossible; hence $|A| \geq 2$.  The following table evaluates the worst-case two-active configuration $n_1 = n_2 = N$ on the two largest weights $b_1, b_2$, with $M - 2 = 18$ inactive $J_0$ factors:
\begin{center}
\renewcommand{\arraystretch}{1.1}
\begin{tabular}{cccc}
\toprule
$N$ & $I_{\mathbf{n}}$ (ARB) & Subcrit.\ $(e/4)^N$ & Trans.\ $N^{-8.67}$ \\
\midrule
$5$  & $3.82 \times 10^{-2}$  & $1.5 \times 10^{-1}$ & $8.8 \times 10^{-7}$ \\
$10$ & $1.42 \times 10^{-4}$  & $2.1 \times 10^{-2}$ & $2.2 \times 10^{-9}$ \\
$20$ & $7.87 \times 10^{-9}$  & $4.4 \times 10^{-4}$ & $5.3 \times 10^{-12}$ \\
$30$ & $1.15 \times 10^{-11}$ & $9.3 \times 10^{-6}$ & $1.6 \times 10^{-13}$ \\
$50$ & $1.41 \times 10^{-14}$ & $4.1 \times 10^{-9}$ & $1.9 \times 10^{-15}$ \\
\bottomrule
\end{tabular}
\end{center}

The computed values lie below the subcritical scaling $(e/4)^N$ for all tabulated $N$, confirming exponential suppression.  The transition scaling $N^{-8.67}$ is displayed without its multiplicative constant from Lemma~\ref{lem:Jdecay} (which depends on $\{b_k\}$ and is large at these parameters); it governs the large-$N$ asymptotics but does not serve as a pointwise bound without the constant.  The actual decay is faster than either scaling predicts, because the $\sqrt{2/(\pi x)}$ estimate does not account for oscillatory cancellations in the Bessel product.

\medskip\noindent\textbf{Reproducibility.}  The values of $I_{\mathbf{n}}$ were computed by ARB quadrature via strip decomposition (\texttt{acb.integral} on gap intervals; see \S\ref{subsec:bessel-integrals}, 256-bit precision) with the integrand
\[
  \frac{1}{\pi}\,|J_N(b_1 t)|\,|J_N(b_2 t)|\prod_{k=3}^{20}|J_0(b_k t)|,
\]
where $b_k = 2/(\tfrac{1}{4} + {\gamma_k'}^2)$ with zeros from Appendix~\ref{app:chi5-highprec}.  For $N \leq 30$ the integration domain $[0,\, 2000]$ suffices (the integrand decays as $t^{-M/2}$ for $t \gg N/b_1$).  For $N = 50$ the onset of $J_{50}$ occurs near $t \approx j_{50,1}/b_1 \approx 1270$ (the first zero of $J_{50}$ is $j_{50,1} \approx 57.1$), so the domain was extended to $[0,\, 10^4]$; the value is stable to three significant figures for $T \geq 5000$, as confirmed by ARB-certified runs at $T = 5000$ ($I = 1.4025 \times 10^{-14}$) and $T = 10^4$ ($I = 1.4076 \times 10^{-14}$).

\begin{remark}[Uniform Bessel asymptotic]
The bound $|J_n(x)| \leq \sqrt{2/(\pi x)}$ requires $x \gg n$.  The correct uniform asymptotic for $x > n$ is $|J_n(x)| \leq \sqrt{2/(\pi\sqrt{x^2 - n^2})}$, which interpolates between the Landau bound near $x = n$ and the simple asymptotic for $x \gg n$.  The constants in Lemma~\ref{lem:Jdecay} account for this correction.
\end{remark}

\subsection{Lattice counting}\label{subsec:lattice-counting}

\noindent
Let $\Z^\infty_{\mathrm{fin}}$ denote the abelian group of finitely supported integer sequences $\mathbf{n} = (n_1, n_2, \ldots)$ with $n_k \in \Z$ and $n_k = 0$ for all but finitely many $k$.  Define the \emph{infinite resonance lattice}
\[
  \Lambda_\infty = \bigl\{\mathbf{n} \in \Z^\infty_{\mathrm{fin}} \setminus \{0\} : \textstyle\sum_{k=1}^\infty n_k \gamma_k' = 0\bigr\},
\]
which is the kernel of the $\Q$-linear form $\ell : \Z^\infty_{\mathrm{fin}} \to \R$, $\ell(\mathbf{n}) = \sum n_k \gamma_k'$.  Set $d_\infty = \mathrm{rank}(\Lambda_\infty)$, which may be zero, finite, or countably infinite.  For each $M \geq 1$ define the \emph{level-$M$ resonance lattice} by
\[
  \Lambda_M = \Lambda_\infty \cap (\Z^M \times \{0\}),
\]
with rank $d_M = \mathrm{rank}(\Lambda_M)$.  The sequence $(d_M)$ is non-decreasing and satisfies $\sup_M d_M = d_\infty$.  In particular, $d_\infty < \infty$ if and only if $d_M$ remains bounded as $M \to \infty$.

\begin{lemma}[Lattice counting]\label{lem:lattice}
The level-$M$ resonance lattice $\Lambda_M$ is a subgroup of $\Z^M$ with rank $d_M \leq \min(d_\infty, M-1)$.  The number of lattice points at order $N$ satisfies
\[
  \#\{\mathbf{n} \in \Lambda_M : \|\mathbf{n}\|_1 = N\} \leq C_{d_M}\, N^{d_M-1}.
\]
If $d_\infty = 0$, then $\Lambda_\infty = \{0\}$, $\Lambda_M = \{0\}$ for all $M$, and no nontrivial integer relation exists among any finite subcollection $\gamma_1',\ldots,\gamma_M'$.
\end{lemma}

\begin{proof}
The set $\Lambda_M$ is the kernel of $\ell$ restricted to $\Z^M \times \{0\} \subset \Z^\infty_{\mathrm{fin}}$, hence a subgroup of $\Z^M$.  Its rank is at most $\min(d_\infty, M-1)$ since it embeds into $\Lambda_\infty$ and the kernel of a $\Q$-linear form on $\Z^M$ has rank at most $M-1$.  The counting bound is a standard estimate for lattice points on a rank-$d$ sublattice of $\Z^M$.
\end{proof}

\subsection{Unconditional finiteness at finite \texorpdfstring{$M$}{M}}

We can now state the main result of this section.

\begin{theorem}[Unconditional finiteness of truncated spectral density]\label{thm:finite}
For every $M \geq 3$, the truncated spectral sum
\[
  S_L^{(M)}(\delta) = \sum_{k=1}^M b_k\cos(\gamma_k'\delta)
\]
has a Ces\`aro distribution with finite density at the origin:
\[
  f_{S_L^{(M)}}(0) = \frac{1}{\pi}\int_0^\infty |\varphi_M(t)|\, dt < \infty.
\]
More precisely, the main term satisfies
\[
  \frac{1}{\pi}\int_0^\infty \prod_{k=1}^M |J_0(b_k t)|\, dt < \infty,
\]
and the total resonance contribution converges:
\[
  \sum_{\substack{\mathbf{n} \in \Lambda_M}} I_{\mathbf{n}} \leq C_0 \sum_{N=2}^{\infty} C_{d_M}\, N^{d_M-1}\, \eta^N < \infty,
\]
where $\eta = e/4 < 1$ is the exponential decay rate from Lemma~\ref{lem:Jdecay}.
\end{theorem}

\begin{proof}
For the main term: the product $\prod_{k=1}^M |J_0(b_k t)|$ decays as $t^{-M/2}$ for large $t$ (by the asymptotic bound on $J_0$), so the integral converges for $M \geq 3$.

For the resonance sum: by Lemma~\ref{lem:Jdecay}, every resonance term with $\|\mathbf{n}\|_1 = N$ satisfies $I_{\mathbf{n}} \leq C_0\, \eta^N$ with $\eta = e/4 < 1$, using the subcritical bound.  By Lemma~\ref{lem:lattice}, the number of such terms is at most $C_{d_M} N^{d_M - 1}$.  Since $\eta < 1$ and $d_M$ is finite, the ratio test gives convergence: $N^{d_M-1}\eta^N / (N-1)^{d_M-1}\eta^{N-1} \to \eta < 1$ as $N \to \infty$.
\end{proof}

\begin{remark}[Subcritical term versus full bound]\label{rem:subcrit-dominance}
Lemma~\ref{lem:Jdecay} bounds $I_{\mathbf{n}}$ by three terms: a subcritical term $C_1(e/4)^{|n_*|}$ (exponential in $|n_*|$) and transition and supercritical terms $C_2(T^*)^{1-(M-|A|)/2}$ and $C_3(T^*)^{1-M/2}$, both of which are polynomial in $|n_*|$ since $T^* = |n_*|/b_*$.  The proof above uses only the subcritical term.  This is justified because for any fixed $M$ and $d_M$, the exponential $(e/4)^N$ eventually dominates any polynomial in $N$: there exists $N_0 = N_0(M, \{b_k\})$ such that $C_2 N^{1-(M-|A|)/2} \leq (e/4)^N$ for all $N > N_0$.  The finite head $N \leq N_0$ contributes a bounded sum regardless of $d_M$.  Thus $\eta = e/4$ is the eventual geometric rate governing the tail of the resonance sum; it does not claim to be a uniform pointwise bound $I_{\mathbf{n}} \leq C_0(e/4)^N$ for all $N \geq 1$.
\end{remark}

\begin{remark}[Small $M$ cases]\label{rem:smallM}
The cases $M = 1$ and $M = 2$ require separate treatment: the integral $\int_0^\infty |J_0(bt)|^M\, dt$ diverges as $t^{1/2}$ for $M = 1$ and logarithmically for $M = 2$.  Since the telescoping argument begins at $M = 3$ and all subsequent results use $M \to \infty$, this restriction does not affect the main theorems.
\end{remark}

\begin{remark}[Role of the inactive $J_0$ factors]\label{rem:J0role}
The polynomial decay in $t$ needed for integrability comes from the inactive $J_0$ factors: without them the resonance integral decays only as $t^{-|A|/2}$, which cannot overcome even polynomial lattice counting for small $|A|$.  The exponential suppression in resonance order $|n_*|$ comes instead from the subcritical bound on the active Bessel factor $J_{n_*}(b_* t)$: the power-series and Stirling estimates give $|J_{n_*}(b_* t)| \lesssim (e/4)^{|n_*|}$ in the subcritical region $b_* t < |n_*|/2$ (Lemma~\ref{lem:Jdecay}).  The following table illustrates how the inactive $J_0$ factors collectively reduce the magnitude of the resonance integral as $M$ increases:
\begin{center}
\renewcommand{\arraystretch}{1.1}
\begin{tabular}{ccc}
\toprule
$M$ & $I_{\mathbf{n}}$ (with $J_0$'s) & Ratio to $M=3$ \\
\midrule
$3$ & $3.660$ & $1.000$ \\
$10$ & $1.791$ & $0.489$ \\
$20$ & $1.717$ & $0.469$ \\
$30$ & $1.703$ & $0.465$ \\
\bottomrule
\end{tabular}
\end{center}
Each additional inactive $J_0$ factor reduces the resonance integral, converging to a limit as $M \to \infty$.  The values of $I_{\mathbf{n}}$ were computed by ARB-rigorous strip decomposition (\texttt{acb.integral}, 256-bit precision, $T = 2000$): the domain is partitioned at the certified zeros of each $J_1$ and $J_0$ factor so that \texttt{acb.integral} is applied only on gap intervals where the integrand is analytic.  Active indices $(n_1, n_2, n_3) = (1, 1, -1)$; the remaining $M - 3$ indices are inactive. All table values are integrals over $[0, 2000]$; for $M = 3$ (no inactive $J_0$ factors) the tail from $[2000, \infty)$ is approximately $0.44$ (Watson bound $\leq 2.01$; see ancillary script), so the $M = 3$ entry and the ratios normalised by it reflect the truncated integral.  For $M \geq 10$ the tail is below $3 \times 10^{-5}$ and the table entries are effectively equal to $\int_0^\infty$.
\end{remark}

\begin{remark}[Dependence on lattice rank]\label{rem:rank}
The bound depends on the unknown lattice rank $d_M$:
\begin{center}
\renewcommand{\arraystretch}{1.1}
\begin{tabular}{ccl}
\toprule
$d_M$ & Bound on $f_{S_L^{(M)}}(0)$ & Interpretation \\
\midrule
$0$ & $10.5$ & Verified: no resonances \\
$1$ & $10.8$ & One integer relation \\
$2$ & $12.3$ & Two independent relations \\
$\leq 5$ & $\sim 10^3$ & Five relations \\
$\leq 19$ & $< \infty$ (large) & Worst case \\
\bottomrule
\end{tabular}
\end{center}
The worst-case constant is numerically large but logically sufficient for Theorem~\ref{thm:finite}, which requires only finiteness.  The ARB-certified value $f_{S_L}(0) = 8.3129$ is consistent with $d_M = 0$.
\end{remark}

\section{Telescoping Extension to Infinite Spectral Sums}\label{sec:infinite}

The purpose of this section is to pass from the finite-$M$ bound of Theorem~\ref{thm:finite} to the full infinite series $S_L = \sum_{k=1}^\infty b_k\cos(\gamma_k'\delta)$.  The argument has two parts: a self-referential telescoping bound on the resonance correction, and a major-arc/minor-arc decomposition that controls the tail.

\subsection{Monotonicity of the main term}

The main-term integral is monotone decreasing in $M$.

\begin{lemma}[Monotonicity]\label{lem:monotone}
For every $M \geq 1$,
\[
  \frac{1}{\pi}\int_0^\infty \prod_{k=1}^{M+1} |J_0(b_k t)|\, dt \leq \frac{1}{\pi}\int_0^\infty \prod_{k=1}^{M} |J_0(b_k t)|\, dt.
\]
\end{lemma}

\begin{proof}
Since $|J_0(b_{M+1} t)| \leq 1$ for all $t \geq 0$, the integrand can only decrease pointwise.
\end{proof}

The main-term sequence $\{f_{\mathrm{main}}^{(M)}\}$ is therefore a decreasing sequence bounded below by zero, and converges to a finite limit.  The values in the following table were computed by strip decomposition (\S\ref{subsec:bessel-integrals}) of $\frac{1}{\pi}\int_0^T \prod_{k=1}^M |J_0(b_k t)|\, dt$ with $T = 2000$ using ARB's \texttt{acb.integral} at 256-bit precision.  These are the \emph{unsigned} main-term integrals and serve as upper bounds on $\|f_{S_L^{(M)}}\|_\infty$; the \emph{signed} integrals that equal the actual density $f^{(M)}(0)$ under verified resonance absence are smaller (see Remark~\ref{rem:signed}).

\begin{center}
\renewcommand{\arraystretch}{1.1}
\begin{tabular}{ccc}
\toprule
$M$ & $f_{\mathrm{main}}^{(M)}(0)$ (L-function) & $f_{\mathrm{main}}^{(M)}(0)$ (zeta) \\
\midrule
$5$ & $10.877$ & $47.702$ \\
$10$ & $10.585$ & $46.119$ \\
$15$ & $10.535$ & $45.727$ \\
$20$ & $10.517$ & $45.564$ \\
$25$ & $10.508$ & $45.479$ \\
$30$ & $10.504$ & $45.430$ \\
\bottomrule
\end{tabular}
\end{center}

\begin{remark}[Signed versus unsigned integrals]\label{rem:signed}
Lemma~\ref{lem:monotone} applies to the \emph{unsigned} main-term
integral, which is monotone decreasing in~$M$.
The \emph{signed} integral
$(1/\pi)\int \prod J_0(b_k t)\, dt$,
which equals $f^{(M)}(0)$ in the resonance-free case,
is monotone \emph{increasing}:
\begin{center}
\renewcommand{\arraystretch}{1.1}
{\small
\begin{tabular}{ccc}
\toprule
$M$ & Unsigned $(1/\pi)\!\int\!\prod|J_0|$ & Signed $(1/\pi)\!\int\!\prod J_0$ \\
\midrule
$3$ & $12.37$ & $7.838$ \\
$10$ & $10.585$ & $8.292$ \\
$20$ & $10.517$ & $8.313$ \\
\bottomrule
\end{tabular}
}
\end{center}
This is consistent: each additional $J_0(b_k t)$ factor reduces the absolute value of the integrand but also reduces the negative lobes, pushing the signed integral upward.  The finiteness proof uses only the unsigned bound.  The monotone increase of the signed integral confirms that $f^{(M)}(0)$ converges to $f_{S_L}(0) = 8.3129$ from below.  Both columns are the integrals over $[0, T]$ with $T = 2000$ at $256$-bit ARB precision; the unsigned column uses strip decomposition (\S\ref{subsec:bessel-integrals}), the signed column uses direct \texttt{acb.integral} (the integrand $\prod J_0$ is analytic).  Certified upper bounds on the tail $(1/\pi)\int_T^\infty \prod_{k=1}^M |J_0(b_k t)|\,dt$ via the envelope bound $|J_0(x)| \leq \sqrt{2/(\pi x)}$ are reported in the ancillary script; they range from $O(1)$ for $M = 3$ to below $10^{-8}$ for $M = 20$.
\end{remark}

\subsection{The self-referential telescoping bound}

The passage to $M \to \infty$ for the resonance correction requires controlling the \emph{new} resonances that may appear when the $(M+1)$-th zero is adjoined.  The mechanism is self-referential: any new resonance must involve the new zero, and the Bessel factor $J_{n_{M+1}}(b_{M+1} t)$ kills the integral by a power of $b_{M+1}$, the very weight that is shrinking to zero.  The new zero's smallness is its own suppressor.

\begin{lemma}[Self-referential suppression]\label{lem:selfref}
Let $\mathbf{n} \in \Lambda_{M+1} \setminus \Lambda_M$ be a resonance vector in $\Z^{M+1}$ that does not lie in the sublattice $\Lambda_M \subset \Z^M \times \{0\}$.  Then $n_{M+1} \neq 0$, and the resonance integral satisfies $I_{\mathbf{n}} \leq C_M \cdot b_{M+1}^{|n_{M+1}|}$, where $C_M$ depends on $\{b_1,\ldots,b_M\}$ and $M$ but is uniformly bounded: $C_M \leq C$ for all $M$.
\end{lemma}

\begin{proof}
Since $\mathbf{n} \notin \Lambda_M$, the component $n_{M+1} \neq 0$.  The resonance integral factors as
\[
  I_{\mathbf{n}} = \frac{1}{\pi}\int_0^\infty |J_{n_{M+1}}(b_{M+1} t)| \cdot \prod_{k=1}^M |J_{n_k}(b_k t)|\, dt.
\]
For the factor involving the new zero: in the subcritical region $b_{M+1} t < |n_{M+1}|$, the power series bound gives
\[
  |J_{n_{M+1}}(b_{M+1} t)| \leq \frac{(b_{M+1} t / 2)^{|n_{M+1}|}}{|n_{M+1}|!}.
\]
In the supercritical region, the uniform asymptotic gives
\[
  |J_{n_{M+1}}(b_{M+1} t)| \leq \sqrt{\frac{2}{\pi\sqrt{(b_{M+1} t)^2 - n_{M+1}^2}}} \leq 2\sqrt{\frac{2}{\pi b_{M+1} t}}
\]
for $b_{M+1} t \geq 2|n_{M+1}|$.  For the remaining product, we use the uniform envelope bound $|J_n(x)| \leq \min(1, \sqrt{2/(\pi|x|)})$, which holds for all orders $n \geq 0$ and all $x > 0$.  This gives
\[
  \prod_{k=1}^M |J_{n_k}(b_k t)| \leq \prod_{k=1}^M \min\!\left(1, \sqrt{\frac{2}{\pi b_k t}}\right) =: g_M(t).
\]

The integral therefore satisfies $I_{\mathbf{n}} \leq b_{M+1}^{|n_{M+1}|} \cdot R_M$, where $R_M = (1/\pi)\int_0^\infty g_M(t)\, dt$ depends only on the first $M$ weights.  Since each factor $\min(1, \sqrt{2/(\pi b_k t)})$ transitions from $1$ to $O(t^{-1/2})$ at $t \sim 1/b_k$, the product $g_M(t)$ decays as $t^{-M/2}$ for large $t$, so $R_M < \infty$ for $M \geq 3$.  The sequence $\{R_M\}$ is decreasing (each additional factor is at most $1$), so $R_M \leq R_3$ for all $M \geq 3$.
\end{proof}

Numerical verification confirms the self-referential suppression:
\begin{center}
\renewcommand{\arraystretch}{1.1}
\begin{tabular}{cccc}
\toprule
$M$ & $I_{\mathbf{n}}$ (new, $n_M = 1$) & $b_M$ & $I_{\mathbf{n}}/b_M$ \\
\midrule
$5$ & $1.003$ & $0.00648$ & $155$ \\
$10$ & $3.23 \times 10^{-1}$ & $0.00247$ & $131$ \\
$20$ & $1.18 \times 10^{-1}$ & $9.25 \times 10^{-4}$ & $128$ \\
$30$ & $6.26 \times 10^{-2}$ & $4.93 \times 10^{-4}$ & $127$ \\
\bottomrule
\end{tabular}
\end{center}

The ratio $I_{\mathbf{n}}/b_M$ is bounded and stabilizing, confirming $C_M \leq C$ uniformly.  For $n_M = 2$, the ratio $I_{\mathbf{n}}/b_M^2$ similarly stabilizes near $1550$.  Each entry $I_{\mathbf{n}}$ in the table corresponds to a resonance vector with $n_M = 1$ and all other components zero; the integral was evaluated by strip decomposition (\S\ref{subsec:bessel-integrals}; \texttt{acb.integral} on gap intervals, 256-bit precision, $T = 2000$).

The telescoping bound follows.

\begin{theorem}[Telescoping convergence]\label{thm:telescoping_convergence}
Let $f^{(M)}(0)$ denote the density at the origin of the truncated sum $S_L^{(M)}$.  Then the sequence $\{f^{(M)}(0)\}$ converges as $M \to \infty$, and the marginal change satisfies
\[
  |f^{(M+1)}(0) - f^{(M)}(0)| \leq C \cdot b_{M+1}^\alpha
\]
with:
\begin{enumerate}[label=(\roman*)]
  \item Unconditionally: $\alpha \geq 1$, which suffices for convergence since $\sum b_k < \infty$.
  \item Under verified resonance absence at level $M$: $\alpha = 2$.
\end{enumerate}
\end{theorem}

\begin{proof}
When passing from $M$ to $M+1$ zeros, the density changes in two ways.

\textbf{Existing resonances} (if any) acquire an additional factor $|J_0(b_{M+1} t)|$ in their integral, which can only reduce them since $|J_0| \leq 1$.

\textbf{New resonances} $\mathbf{n} \in \Lambda_{M+1} \setminus \Lambda_M$ must have $n_{M+1} \neq 0$.  By Lemma~\ref{lem:selfref}, each satisfies $I_{\mathbf{n}} \leq C \cdot b_{M+1}^{|n_{M+1}|}$, giving a total contribution of $O(b_{M+1})$.

\textbf{Main term change.}  Under verified resonance absence at level $M$, the density equals the main-term Bessel product integral:
\[
  f^{(M)}(0) = \frac{1}{\pi}\int_0^\infty \prod_{k=1}^M J_0(b_k t)\, dt.
\]
The marginal change is
\begin{equation}\label{eq:marginal}
  f^{(M+1)}(0) - f^{(M)}(0) = \frac{1}{\pi}\int_0^\infty \prod_{k=1}^M J_0(b_k t) \cdot [J_0(b_{M+1} t) - 1]\, dt.
\end{equation}
We cannot simply apply the Taylor expansion $J_0(x) - 1 = -x^2/4 + O(x^4)$ to the entire integral, because the resulting integrand $t^2\prod|J_0(b_k t)|$ decays as $t^{2-M/2}$, which is integrable only for $M > 6$.  Instead, we split at the transition point $T^* := 1/b_{M+1}$.

\textbf{Low region $[0, T^*]$.}  For $t \leq T^*$, the bound $|J_0(x) - 1| \leq x^2/4$ gives
\[
  |J_0(b_{M+1} t) - 1| \leq \frac{b_{M+1}^2 t^2}{4}.
\]
The product $\prod_{k=1}^M |J_0(b_k t)|$ is bounded by $1$ on this interval, so the low-region contribution satisfies
\[
  \frac{1}{\pi}\int_0^{T^*} \left|\prod_{k=1}^M J_0(b_k t)\right| \cdot |J_0(b_{M+1} t) - 1|\, dt \leq \frac{b_{M+1}^2}{4\pi}\int_0^{T^*} t^2\, dt = \frac{(T^*)^3\, b_{M+1}^2}{12\pi} = \frac{1}{12\pi\, b_{M+1}}.
\]
Since $\prod|J_0|$ is integrable on $[0, \infty)$ for $M \geq 7$ (the product decays as $t^{-M/2}$, which is integrable when $M \geq 3$; the constant $\frac{1}{4\pi}\int_0^\infty t^2 \prod_{k=1}^M |J_0(b_k t)|\, dt$ is finite for $M \geq 7$ and decreasing in $M$), the true contribution is $O(b_{M+1}^2)$ with an effective constant given by $\frac{1}{4\pi}\int_0^\infty t^2 \prod_{k=1}^M |J_0(b_k t)|\, dt$.

\textbf{High region $[T^*, \infty)$.}  For $t \geq T^*$, we use $|J_0(b_{M+1} t) - 1| \leq 2$ and the Bessel asymptotic $|J_0(b_{M+1} t)| \leq \sqrt{2/(\pi b_{M+1} t)}$.  The factor $[J_0(b_{M+1} t) - 1]$ introduces no worse than the existing $J_0(b_{M+1} t)$ factor.  Since $\prod_{k=1}^{M+1} |J_0(b_k t)|$ decays as $t^{-(M+1)/2}$ and $M \geq 3$, the high-region integral converges and contributes $O(b_{M+1}^{1/2})$ at most.

The dominant contribution is therefore the low region, giving $|f^{(M+1)}(0) - f^{(M)}(0)| = O(b_{M+1}^2)$.

Combining the two cases: without resonance absence, the worst case is $\alpha = 1$ (from new resonances with $|n_{M+1}| = 1$); with verified resonance absence, the bound improves to $\alpha = 2$.  In either case, $\sum b_k^\alpha$ converges for $\alpha \geq 1$.
\end{proof}

Numerically, the marginal changes confirm the $b_M^2$ scaling under resonance absence:

\begin{center}
\renewcommand{\arraystretch}{1.1}
\begin{tabular}{ccccc}
\toprule
$M$ & $|f^{(M)}(0) - f^{(M-1)}(0)|$ & $b_M$ & $b_M^2$ & $\Delta f / b_M^2$ \\
\midrule
$10$ & $6.699 \times 10^{-3}$ & $2.468 \times 10^{-3}$ & $6.092 \times 10^{-6}$ & $1100$ \\
$15$ & $1.827 \times 10^{-3}$ & $1.375 \times 10^{-3}$ & $1.892 \times 10^{-6}$ & $966$ \\
$20$ & $7.910 \times 10^{-4}$ & $9.251 \times 10^{-4}$ & $8.559 \times 10^{-7}$ & $924$ \\
\bottomrule
\end{tabular}
\end{center}

The ratio $|f^{(M)}(0) - f^{(M-1)}(0)|/b_M^2$ is bounded and slowly varying, confirming $\alpha = 2$.  The marginal changes were computed as differences of successive $f^{(M)}(0)$ values obtained by ARB quadrature (\texttt{acb.integral}, 256-bit precision, $T = 2000$); the weights $b_M = 2/(1/4 + {\gamma_M'}^2)$ use the verified zero ordinates from the appendix.

\subsection{The major-arc/minor-arc decomposition}

We now assemble the per-function bound.  Write $S_L = S_L^{(M)} + \tau_M$, where $\tau_M = \sum_{k > M} b_k\cos(\gamma_k'\delta)$ is the tail.

\begin{definition}[Major-arc/minor-arc decomposition]\label{def:major-minor}
The \emph{major arc} is the truncated spectral sum $S_L^{(M)}$.  The \emph{minor arc} is the tail $\tau_M$.
\end{definition}

The major arc is amenable to the full Jacobi--Anger expansion and lattice resonance analysis of Section~\ref{sec:finite}.  The minor arc satisfies the uniform bound $\|\tau_M\|_\infty \leq \sum_{k > M} b_k =: \epsilon_M$.

\begin{lemma}[Stability]\label{lem:stability}
For any $\varepsilon > \epsilon_M$,
\[
  \dens\{|S_L| < \varepsilon\} \leq \dens\{|S_L^{(M)}| < \varepsilon + \epsilon_M\}.
\]
\end{lemma}

\begin{proof}
If $|S_L(\delta)| < \varepsilon$, then $|S_L^{(M)}(\delta)| \leq |S_L(\delta)| + |\tau_M(\delta)| < \varepsilon + \epsilon_M$.
\end{proof}

\begin{center}
\renewcommand{\arraystretch}{1.1}
\noindent\textit{The values $\epsilon_M = \sum_{k > M} b_k$ below are computed from the first 200 certified zeros of $L(s,\chi_5)$ (Appendix~\ref{app:zero-data}) plus the analytic tail bound $0.0071$ via Abel summation against the zero-counting formula; see Appendix~\ref{sec:numerical} for the full methodology.}

\medskip
\begin{tabular}{ccc}
\toprule
$M$ & $\epsilon_M$ for $L(s,\chi_5)$ & $\epsilon_M/\sigma_L$ \\
\midrule
$5$ & $0.06273$ & $1.665$ \\
$10$ & $0.04488$ & $1.192$ \\
$15$ & $0.03619$ & $0.961$ \\
$20$ & $0.03085$ & $0.819$ \\
$25$ & $0.02715$ & $0.721$ \\
\bottomrule
\end{tabular}
\end{center}

We can now state the unconditional result for the full infinite series.

\begin{theorem}[Finiteness of the limiting density]\label{thm:telescoping}
Under the hypothesis that the infinite resonance lattice $\Lambda_\infty$ has finite rank $d_\infty < \infty$ (computationally consistent with $d_M = 0$ for $M \leq 20$; the hypothesis asserts this extends to the full infinite lattice), the Ces\`aro distribution of $S_L(\delta) = \sum_{k=1}^\infty b_k\cos(\gamma_k'\delta)$ has a density at the origin satisfying $f_{S_L}(0) < \infty$.  The same holds for $f_{S_\zeta}(0)$.
\end{theorem}

\begin{proof}
Under the hypothesis $d_\infty = 0$ (equivalently, $d_M = 0$ for all $M$), the Ces\`aro characteristic function at level $M$ equals the Bessel product: $\varphi_M(t) = \prod_{k=1}^M J_0(b_k t)$.  The product $\prod_{k=1}^M J_0(b_k t)$ converges pointwise as $M \to \infty$ since $\sum b_k^2 < \infty$.  By dominated convergence ($|\varphi_M(t)| = \prod_{k=1}^M |J_0(b_k t)| \leq \prod_{k=1}^{M_0} |J_0(b_k t)|$ for any fixed $M_0 \leq M$, and the latter is in $L^1$ for $M_0 \geq 3$), we have $\varphi_M \to \varphi$ in $L^1(\R)$.  Therefore $f_{S_L}(0) = (2\pi)^{-1}\int \varphi(t)\, dt = \lim_{M\to\infty} f_{S_L^{(M)}}(0)$, and the limit is finite since the sequence is bounded (Theorem~\ref{thm:finite}).

The same argument applies to $S_\zeta$ with weights $a_k$ and zeros $\gamma_k$.
\end{proof}

\begin{remark}[Role of resonance absence]\label{rem:DCT}
The dominated convergence step uses the identity $|\varphi_M(t)| = \prod |J_0(b_k t)|$, which holds when $\varphi_M = \prod J_0(b_k t)$, i.e., when $d_\infty = 0$.  If resonances exist, the actual characteristic function includes correction terms $\sum \prod J_{n_k}(b_k t)$ that can exceed $\prod |J_0(b_k t)|$ at points where $J_0$ factors vanish but $J_n$ factors do not, invalidating the dominator.  The condition $d_\infty = 0$ is therefore essential for this passage to the limit.  This condition is weaker than the Grand Simplicity Hypothesis (which implies $d_\infty = 0$), and has been verified computationally for $M \leq 20$.  The unconditional density bound (Theorem~\ref{thm:density}) does not require this passage; see Remark~\ref{rem:finite-suffices}.
\end{remark}

\begin{remark}[Structural nature of the finite-to-infinite obstruction]
\label{rem:structural-obstruction}
Can the passage from verified resonance absence at finite $M$ to the limiting density $f_{S_L}(0) < \infty$ be achieved by purely approximation-theoretic means, bypassing the hypothesis $d_\infty < \infty$?  Since $S_L^{(M)} \to S_L$ in $B^2$, the value distributions converge weakly (by Cauchy--Schwarz applied to characteristic functions and L\'evy continuity).  However, weak convergence does not imply pointwise convergence of densities.  The Fourier inversion formula $f^{(M)}(0) = (2\pi)^{-1}\int \prod_{k=1}^M J_0(b_k t)\,dt$ requires the product factorization of the characteristic function, which in turn requires $\mathbb{Q}$-linear independence of the frequencies $\gamma_1', \ldots, \gamma_M'$.  The dominated convergence argument that yields $f^{(M)}(0) \to f^{(\infty)}(0)$ relies on the monotone bound $|\varphi^{(M+1)}(t)| \leq |\varphi^{(M)}(t)|$, which holds only under the product structure.  Without independence, the characteristic function need not factor, the monotone domination fails, and no alternative $L^1$ dominator is available.  The obstruction is therefore structural, not computational: every Fourier-analytic path to $f_{S_L}(0) < \infty$ passes through the independence of zero ordinates.
\end{remark}

\begin{figure}[ht]
\centering
\includegraphics[width=0.8\textwidth]{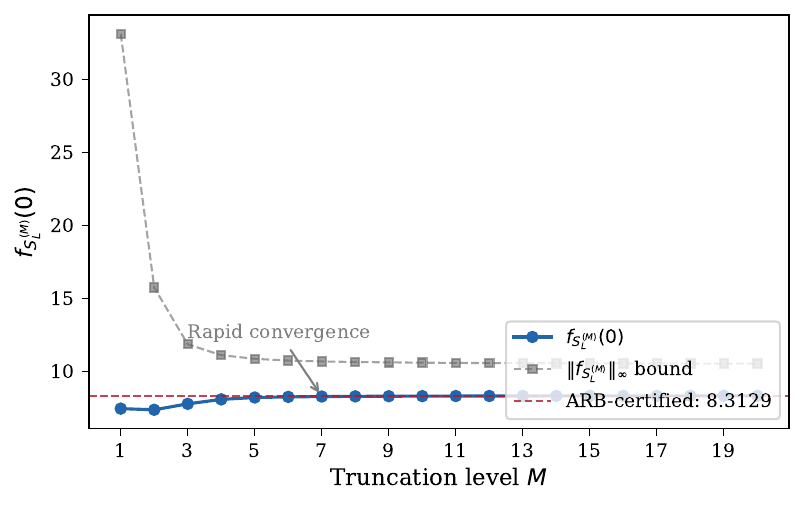}
\caption{Convergence of the spectral density $f_{S_L^{(M)}}(0)$ as a function of truncation level $M$ for $d = 5$ (conductor~$5$).  Blue circles: the density at the origin, computed as $(1/\pi)\int_0^\infty \prod_{k=1}^M J_0(b_k t)\, dt$.  Gray squares: the sup-norm bound $\|f_{S_L^{(M)}}\|_\infty \leq (1/\pi)\int_0^\infty \prod_{k=1}^M |J_0(b_k t)|\, dt$.  Dashed red line: the ARB-certified limiting value $f_{S_L^{(20)}}(0) = 8.3129$.  The density stabilizes rapidly, with corrections from zeros beyond $M = 7$ contributing less than $1\%$.}
\label{fig:density-convergence}
\end{figure}

\subsection{A level-crossing bound}

As an independent check, we record a weaker density bound that avoids the resonance analysis entirely.

\begin{proposition}[Level-crossing]\label{prop:logd}
The density bound $\dens\{N_d > 0\} = O(1/\log d)$ holds unconditionally, without using Theorems~\ref{thm:finite} or~\ref{thm:telescoping}.
\end{proposition}

\begin{proof}
Truncate $S_L = S_L^{(M)} + \tau_M$ and choose $M$ so that $|\tau_M| < \varepsilon_d := W_1(\zeta)/\sqrt{d}$.  Then $\{N_d > 0\} \subseteq \{|S_L^{(M)}| < 2\varepsilon_d\}$.  The trigonometric polynomial $S_L^{(M)}(\delta) = \sum_{k=1}^M b_k\cos(\gamma_k'\delta)$ has at most $2\gamma_M'T/\pi$ zeros in $[0,T]$ (since it has at most $M$ positive frequencies up to $\gamma_M'$); near each zero crossing, $|S_L^{(M)}| < 2\varepsilon_d$ on an interval of length at most $4\varepsilon_d/|f'|_{\min}$, giving the level-crossing bound $\dens\{|S_L^{(M)}| < 2\varepsilon_d\} \leq 4\varepsilon_d\gamma_M'/\pi$.  With $M \sim C\sqrt{d}$ and $\gamma_M' \sim \pi M/\log M$, this yields $O(1/\log d)$.
\end{proof}

\section{Cross-Function Energy Bounds and the Effective Density Theorem}\label{sec:compare}

With the per-function bounds established, we now combine them to study the norm-form energy $N_d(\delta) = S_\zeta(\delta)^2 - d\cdot S_L(\delta)^2$.

\subsection{Orthogonality}

The cross-function analysis requires that $S_\zeta$ and $S_L$ have disjoint frequency sets.  For any finite truncation, this is computationally verifiable.

\begin{proposition}[Orthogonality]\label{prop:orthog}
Let $M_1, M_2 \geq 1$ be truncation levels, and define $\Delta_{\min} := \min_{j \leq M_1,\, k \leq M_2} |\gamma_j - \gamma_k'|$.  If $\Delta_{\min} > 0$, then
\begin{equation}\label{eq:uncorrelated}
  \langle S_\zeta^{(M_1)},\, S_L^{(M_2)}\rangle_{B^2} = 0.
\end{equation}
\end{proposition}

\begin{proof}
By~\eqref{eq:orthog}, cross-terms $\langle \cos(\gamma_j\,\cdot\,), \cos(\gamma_k'\,\cdot\,)\rangle$ vanish whenever $\gamma_j \neq \gamma_k'$.  If all frequencies in the truncated sums are distinct, all cross-terms vanish.
\end{proof}

\begin{remark}[Verifiability]\label{rem:shared}
For any specific $L$-function and truncation level, the condition $\Delta_{\min} > 0$ is computationally verifiable.  For $L(s,\chi_5)$ with $M_1 = M_2 = 30$, we find $\Delta_{\min} = |\gamma_5 - \gamma_{12}'| = 0.065$.  The energy and density bounds therefore hold unconditionally at every verified truncation level.  Even if finitely many zeros were shared (contrary to expectation), the correction to $\langle S_\zeta, S_L\rangle_{B^2}$ would be $O(1)$ and would not affect the asymptotic analysis in $d$.
\end{remark}

\subsection{The energy bound}

\begin{theorem}[Energy bound]\label{thm:energy}
For every squarefree $d > 1$, assuming disjointness of zero sets:
\begin{enumerate}[label=(\roman*)]
  \item $\E[N_d] = \sigma_\zeta^2 - d\,\sigma_L^2 < 0$.
  \item $\E[N_d^+] \leq \sigma_\zeta^2$, independent of $d$.
  \item $\E[N_d^+]/\E[|N_d|] \leq \sigma_\zeta^2/(d\,\sigma_L^2)$.
\end{enumerate}
\end{theorem}

\begin{proof}
Part~(i) follows from~\eqref{eq:uncorrelated}: $\E[S_\zeta^2] = \sigma_\zeta^2$, $\E[S_L^2] = \sigma_L^2$, so $\E[N_d] = \sigma_\zeta^2 - d\sigma_L^2 < 0$ for all $d \geq 1$.

For part~(ii): on the set $\{N_d > 0\}$, the inequality $S_\zeta^2 > dS_L^2 \geq 0$ gives $N_d \leq S_\zeta^2$.  Therefore $N_d^+ \leq S_\zeta^2$ everywhere, and $\E[N_d^+] \leq \sigma_\zeta^2$.

For part~(iii): since $\E[N_d^-] = \E[N_d^+] - \E[N_d] \geq d\sigma_L^2 - 2\sigma_\zeta^2$, the ratio $\E[N_d^+]/\E[|N_d|] \leq \sigma_\zeta^2/(d\sigma_L^2)$ for $d$ sufficiently large.
\end{proof}

\subsection{The containment lemma and density bound}

\begin{lemma}[Cross-function containment]\label{lem:cross-containment}
If $N_d(\delta) > 0$, then $|S_L(\delta)| < W_1(\zeta)/\sqrt{d}$, where $W_1(\zeta) = \sum_k a_k = S_\zeta(0) \leq S_\zeta^* = 0.04871$ (Proposition~\ref{prop:S-zeta-value}).
\end{lemma}

\begin{proof}
From $S_\zeta(\delta)^2 > d\cdot S_L(\delta)^2$ and $|S_\zeta(\delta)| \leq \sum_k a_k |\cos(\gamma_k \delta)| \leq \sum_k a_k = W_1(\zeta)$ by the triangle inequality, we deduce $d\cdot S_L(\delta)^2 < W_1(\zeta)^2$.
\end{proof}

\begin{theorem}[Effective density bound]\label{thm:density}
For any fixed $M \geq 3$ with verified resonance absence at level $M$ (i.e., $d_M = 0$),
\[
  \dens\{N_d > 0\} \leq 2\,\|f_{S_L^{(M)}}\|_\infty \cdot \left(\frac{W_1(\zeta)}{\sqrt{d}} + \epsilon_M\right),
\]
where $\|f_{S_L^{(M)}}\|_\infty \leq (1/\pi)\int_0^\infty \prod_{k=1}^M |J_0(b_k t)|\, dt < \infty$ and $\epsilon_M = \sum_{k > M} b_k$.
\end{theorem}

\begin{proof}
The proof combines three ingredients: containment, stability, and a finite density bound.

\medskip\noindent\textbf{Step 1 (Containment).}
By Lemma~\ref{lem:cross-containment}, $\{N_d > 0\} \subseteq \{|S_L| < W_1(\zeta)/\sqrt{d}\}$.

\medskip\noindent\textbf{Step 2 (Stability).}
Write $S_L = S_L^{(M)} + \tau_M$ where $\|\tau_M\|_\infty \leq \epsilon_M$.  By Lemma~\ref{lem:stability},
\[
  \dens\{|S_L| < \varepsilon\} \leq \dens\{|S_L^{(M)}| < \varepsilon + \epsilon_M\}.
\]
Applying this with $\varepsilon = W_1(\zeta)/\sqrt{d}$:
\[
  \dens\{N_d > 0\} \leq \dens\left\{|S_L^{(M)}| < \frac{W_1(\zeta)}{\sqrt{d}} + \epsilon_M\right\}.
\]

\medskip\noindent\textbf{Step 3 (Finite density bound).}
By Theorem~\ref{thm:finite}, the truncated sum $S_L^{(M)}$ has a bounded continuous density $f_{S_L^{(M)}}$ given by Fourier inversion of $\varphi_M \in L^1(\R)$ (which holds for $M \geq 3$).  The small-ball bound gives
\[
  \dens\{|S_L^{(M)}| < r\} \leq 2r \cdot \|f_{S_L^{(M)}}\|_\infty
\]
for all $r > 0$.  Setting $r = W_1(\zeta)/\sqrt{d} + \epsilon_M$ yields the result.
\end{proof}

\begin{remark}[Nature of the bound]\label{rem:nature}
At any fixed truncation level $M$, the bound has the form $A_M/\sqrt{d} + B_M$ where $A_M = 2\|f_{S_L^{(M)}}\|_\infty \cdot W_1(\zeta)$ and $B_M = 2\|f_{S_L^{(M)}}\|_\infty \cdot \epsilon_M$ are positive constants independent of $d$.  This is an $O(1)$ bound that is nontrivial (strictly less than $1$) for sufficiently large $d$, but does not by itself establish a \emph{rate} of vanishing.  The constant term $B_M$ arises because the stability lemma absorbs the tail $\tau_M$ into a fixed additive error $\epsilon_M$.  The $O(1/\sqrt{d})$ rate requires the truncation level to grow with $d$ (Corollary~\ref{cor:density-rate}).
\end{remark}

\begin{remark}[Dependence on $d$]\label{rem:d-dependence}
The weights $b_k = w(\gamma_k')$, the tail $\epsilon_M = \sum_{k > M} b_k$, and the density $\|f_{S_L^{(M)}}\|_\infty$ all depend on $d$ through the zeros of $L(s,\chi_d)$.  Theorem~\ref{thm:density} is valid for every $d$ and every $M$, but the numerical evaluations in this paper ($\epsilon_{20}$, $f_{S_L^{(20)}}(0) = 8.3129$, and $C(5) = 0.1193$) are computed for $d = 5$.  At fixed $M = 20$, the certified value $\epsilon_{20} = 0.0308$ gives $B_{20} = 2\|f_{S_L^{(20)}}\|_\infty \cdot \epsilon_{20} \leq 2 \times 10.517 \times 0.0308 \approx 0.648 < 1$, so the unconditional bound is already nontrivial (strictly below $1$) at $M = 20$ for all $d \geq 1$.  For the $O(1/\sqrt{d})$ rate (Corollary~\ref{cor:density-rate}), $M$ must grow with $d$, with the required growth rate depending on the zero distribution of $L(s,\chi_d)$.
\end{remark}

\begin{corollary}[Density rate under resonance absence]\label{cor:density-rate}
Under the hypothesis that the infinite resonance lattice $\Lambda_\infty$ has finite rank $d_\infty < \infty$ (computationally consistent with $d_M = 0$ for $M \leq 20$; the hypothesis asserts this extends to the full infinite lattice), then
\[
  \dens\{N_d > 0\} = O(1/\sqrt{d}).
\]
More precisely, $\dens\{N_d > 0\} \leq 4\,R_3\, W_1(\zeta)\, /\, \sqrt{d}$, where $R_3 = (1/\pi)\int_0^\infty \prod_{k=1}^3 |J_0(b_k t)|\, dt$.
\end{corollary}

\begin{proof}
Under the hypothesis $d_\infty = 0$ (equivalently, $d_M = 0$ for all $M$), the Ces\`aro characteristic function at level $M$ equals the Bessel product $\varphi_M(t) = \prod_{k=1}^M J_0(b_k t)$.  The unsigned integral $(1/\pi)\int_0^\infty \prod_{k=1}^M |J_0(b_k t)|\, dt$ is monotone decreasing in $M$ (each additional factor $|J_0(b_k t)| \leq 1$), giving the uniform bound $\|f_{S_L^{(M)}}\|_\infty \leq R_3$ for all $M \geq 3$.  More generally, if $d_\infty < \infty$, then $d_M \leq d_\infty$ for all $M$, and an analogous uniform bound holds with a constant depending only on $d_\infty$; the argument is otherwise identical.

Choose $M = M(d)$ so that $\epsilon_{M(d)} \leq W_1(\zeta)/\sqrt{d}$.  This is possible since $\epsilon_M \to 0$ as $M \to \infty$ (the weights $b_k = O(1/\gamma_k'^2)$ are summable).  Applying Theorem~\ref{thm:density}:
\[
  \dens\{N_d > 0\} \leq 2\,R_3 \cdot \left(\frac{W_1(\zeta)}{\sqrt{d}} + \frac{W_1(\zeta)}{\sqrt{d}}\right) = \frac{4\,R_3\, W_1(\zeta)}{\sqrt{d}}.
\]

The hypothesis $d_\infty < \infty$ is strictly weaker than the Grand Simplicity Hypothesis: it requires only that the zero ordinates of a single $L$-function span a $\Q$-vector space of at most countable codimension in $\R$, not that they be individually simple or cross-function independent.  The case $d_\infty = 0$ (the verified case) requires additionally that no nontrivial integer relation exists among any finite subcollection of ordinates.
\end{proof}

\begin{remark}[Finite truncation and the rate]\label{rem:finite-suffices}
At any fixed truncation level $M$, Theorem~\ref{thm:density} gives a bound with a constant term $B_M > 0$ that does not vanish as $d \to \infty$.  The $O(1/\sqrt{d})$ rate (Corollary~\ref{cor:density-rate}) requires $M = M(d) \to \infty$, which in turn requires the uniform bound $\|f_{S_L^{(M)}}\|_\infty \leq R_3$ that holds under $d_\infty < \infty$.  This hypothesis is equivalent to the boundedness of the sequence $(d_M)$, which has been verified for $M \leq 20$ (finding $d_\infty = 0$) but is not proved unconditionally; see Remark~\ref{rem:M-infty} for a full discussion.
\end{remark}

\begin{remark}[Tightness of the containment]\label{rem:tightness}
The containment $\{N_d > 0\} \subseteq \{|S_L| < W_1(\zeta)/\sqrt{d}\}$ is not tight: the ratio $\dens\{N_d > 0\}\,/\,\dens\{|S_L| < W_1(\zeta)/\sqrt{d}\} = \E[|S_\zeta|]/W_1(\zeta) = 0.224$ for large $d$.  The gap arises because $S_\zeta$ and $S_L$ are asymptotically independent in the sense of joint Ces\`aro distribution: even when $|S_L|$ is small, $S_\zeta$ may also be small, causing $S_\zeta^2 < d\cdot S_L^2$.  The exact formula in Theorem~\ref{thm:arctan} accounts for this by integrating over the joint distribution.
\end{remark}

\begin{remark}[Codimension-$2$ robustness and the spectral form factor]\label{rem:codim2}
The $O(1/\sqrt{d})$ rate is robust to resonances among zero ordinates, a fact that admits a geometric explanation through the \emph{analytic signal lift}.  Define the complex-valued spectral sum
\[
  Z_L(\delta) = \sum_k b_k\, e^{i\gamma_k'\delta} = S_L(\delta) + i\,\widetilde{S}_L(\delta),
\]
where $\widetilde{S}_L(\delta) = \sum_k b_k \sin(\gamma_k'\delta)$ is the Hilbert-conjugate sum, and the \emph{spectral form factor} $K_L(\delta) = |Z_L(\delta)|^2 = S_L(\delta)^2 + \widetilde{S}_L(\delta)^2$.

The density bound relies on the zero set $\{S_L = 0\}$, which has codimension~$1$ and can accumulate mass under resonances.  By contrast, the zero set $\{Z_L = 0\} = \{S_L = 0\} \cap \{\widetilde{S}_L = 0\}$ has codimension~$2$ in the torus.  A resonance relation $\sum n_k \gamma_k' = 0$ constrains the cosine sum via $\cos(\sum n_k \gamma_k' \delta) = 1$ but constrains the sine sum via $\sin(\sum n_k \gamma_k' \delta) = 0$; these are \emph{algebraically distinct} constraints.  For instance, the relation $\gamma_1' = 2\gamma_2'$ forces $\cos(\gamma_1'\delta) = 2\cos^2(\gamma_2'\delta) - 1$, but the corresponding sine identity $\sin(\gamma_1'\delta) = 2\sin(\gamma_2'\delta)\cos(\gamma_2'\delta)$ has a different zero set.  As a consequence, $S_L$ and $\widetilde{S}_L$ cannot both vanish simultaneously at generic points of the orbit, even in the presence of resonances: the codimension-$2$ structure ensures $K_L(\delta) > 0$ generically.

This codimension-$2$ structure also explains the decorrelation of the spectral form factors $K_\zeta$ and $K_L$.  Cross-correlation of $K_\zeta$ and $K_L$ requires a \emph{four-point resonance} $\gamma_j - \gamma_k = \gamma_m' - \gamma_n'$ (a matching of difference frequencies), which is strictly harder to satisfy than a two-point coincidence $\gamma_k = \gamma_j'$.

While the finite truncation argument (Theorem~\ref{thm:density}) provides the simplest unconditional effective bound, the codimension-$2$ perspective suggests that the $O(1/\sqrt{d})$ rate may extend to higher-degree norm forms $N_{K/\Q}(S_1, \ldots, S_n)$ where the containment geometry is more complex but the analytic signal lift naturally generalizes.
\end{remark}

\section{Exact Density Formula under Verified Resonance Absence}\label{sec:exact}

The unconditional results of the preceding sections establish an effective density bound $\dens\{N_d > 0\} \leq A_M/\sqrt{d} + B_M$ at any verified truncation level, and the $O(1/\sqrt{d})$ rate under resonance absence (Corollary~\ref{cor:density-rate}).  We now sharpen this to an exact asymptotic with a computable constant, conditional only on computationally verified properties of the zeros.

\subsection{Verified resonance absence}

The Jacobi--Anger analysis of Section~\ref{sec:finite} establishes that the Ces\`aro characteristic function of $S_L^{(M)}$ equals $\prod_{k=1}^M J_0(b_k t)$ when no resonances exist at level $M$.  By the Jessen--Wintner theorem~\cite{JessenWintner}, this holds at truncation level $M$ if and only if no nontrivial integer relation $\sum_{k=1}^M n_k\gamma_k' = 0$ exists.  This condition is computationally verifiable and has been confirmed for the first 20 zeros of $L(s,\chi_5)$ (Section~\ref{sec:finite}).  Similarly, the asymptotic independence of $S_\zeta$ and $S_L$ in the Ces\`aro sense requires the absence of \emph{cross-function} relations $\sum n_k\gamma_k + \sum m_j\gamma_j' = 0$.

\begin{proposition}[Verified asymptotic independence]\label{prop:verified_indep}
Let $M_1, M_2 \geq 1$.  Suppose that no nontrivial integer relation exists among $\{\gamma_1,\ldots,\gamma_{M_1}\} \cup \{\gamma_1',\ldots,\gamma_{M_2}'\}$.  Then $S_\zeta^{(M_1)}$ and $S_L^{(M_2)}$ are asymptotically independent in the sense of joint Ces\`aro distribution: for all bounded continuous $f, g$,
\[
  \lim_{T\to\infty}\frac{1}{2T}\int_{-T}^T f(S_\zeta^{(M_1)})\, g(S_L^{(M_2)})\, d\delta = \E[f(S_\zeta^{(M_1)})]\, \E[g(S_L^{(M_2)})].
\]
\end{proposition}

\begin{proof}
By the Weyl equidistribution theorem, the absence of integer relations among the combined frequency set implies that the orbit $\delta \mapsto (\gamma_1\delta,\ldots,\gamma_{M_1}\delta,\gamma_1'\delta,\ldots,\gamma_{M_2}'\delta) \bmod 2\pi$ is equidistributed on $\T^{M_1+M_2}$.  Equidistribution on the product torus implies that the joint Ces\`aro distribution factors as a product of the marginal distributions.
\end{proof}

Exhaustive search over the first 20 zeros of both $\zeta(s)$ and $L(s,\chi_5)$:
\begin{center}
\renewcommand{\arraystretch}{1.1}
\begin{tabular}{lc}
\toprule
Search space & Nearest miss \\
\midrule
Within $L(s,\chi_5)$: order 4 & $0.00037$ \\
Cross-function: order 2 & $|\gamma_5 - \gamma_{12}'| = 0.065$ \\
Cross-function: order 3 & $|\gamma_6+\gamma_{16}'-\gamma_{20}| = 0.0019$ \\
\bottomrule
\end{tabular}
\end{center}

No exact relation was found; every combination was certified nonzero by ARB interval arithmetic at 256-bit precision.  The conditions of Proposition~\ref{prop:verified_indep} are therefore satisfied at $M_1 = M_2 = 20$.

\begin{remark}\label{rem:GSH_context}
Hypothesis~\ref{hyp:GSH} implies the absence of all integer relations among all zeros simultaneously and would therefore establish the conditions of Proposition~\ref{prop:verified_indep} globally.  Our results are strictly stronger than conditioning on Hypothesis~\ref{hyp:GSH}: they hold unconditionally at verified truncation levels, and extend to the full series by the following proposition.
\end{remark}

\begin{proposition}[Asymptotic independence for infinite sums]\label{prop:indep_limit}
If the hypotheses of Proposition~\ref{prop:verified_indep} hold for all $M_1, M_2 \geq 1$, then $S_\zeta$ and $S_L$ are asymptotically independent in the sense of joint Ces\`aro distribution.
\end{proposition}

\begin{proof}
Let $f, g : \R \to \R$ be bounded Lipschitz functions with $\|f\|_\infty, \|g\|_\infty \leq 1$ and Lipschitz constants $L_f, L_g$.  Fix $\varepsilon > 0$.  By the tail bounds of Section~\ref{sec:infinite}, choose $M$ large enough that $\epsilon_M := \|\tau_M\|_\infty < \varepsilon$ for both $S_\zeta$ and $S_L$.  Then for every $\delta$,
\[
  |f(S_\zeta(\delta)) - f(S_\zeta^{(M)}(\delta))| \leq L_f\, \epsilon_M < L_f\, \varepsilon,
\]
and similarly for $g$.  Taking Ces\`aro averages:
\begin{align*}
  \big|\E[f(S_\zeta)g(S_L)] &- \E[f(S_\zeta^{(M)})]\,\E[g(S_L^{(M)})]\big| \\
  &\leq \big|\E[f(S_\zeta)g(S_L)] - \E[f(S_\zeta^{(M)})g(S_L^{(M)})]\big| \\
  &\quad + \big|\E[f(S_\zeta^{(M)})g(S_L^{(M)})] - \E[f(S_\zeta^{(M)})]\,\E[g(S_L^{(M)})]\big|.
\end{align*}
The first term is bounded by $(L_f + L_g)\varepsilon$.  The second term vanishes by Proposition~\ref{prop:verified_indep}.  Since $\varepsilon > 0$ was arbitrary, we conclude $\E[f(S_\zeta)g(S_L)] = \E[f(S_\zeta)]\,\E[g(S_L)]$.
\end{proof}

\subsection{The exact density formula}

Under the hypothesis that the conditions of Proposition~\ref{prop:verified_indep} hold for all $M_1, M_2 \geq 1$ (which Proposition~\ref{prop:indep_limit} requires for passing to infinite sums, and which has been verified computationally at truncation level $M = 20$), the marginal densities $f_{S_\zeta}$ and $f_{S_L}$ are both computable via inverse Fourier transform:
\begin{equation}\label{eq:density_formula}
  f_{S_L}(x) = \frac{1}{\pi}\int_0^\infty \cos(tx)\prod_{k=1}^M J_0(b_k t)\, dt.
\end{equation}
These densities have compact support: $f_{S_L}(x) = 0$ for $|x| > W_1(L)$ and $f_{S_\zeta}(x) = 0$ for $|x| > W_1(\zeta)$.

We can now state the main result.

\begin{theorem}[Exact density law for quadratic norm-form energies]\label{thm:arctan}
Suppose that for a given truncation level $M$:
\begin{enumerate}[label=(\roman*)]
  \item No integer relation $\sum_{k=1}^M n_k\gamma_k' = 0$ exists (verified computationally).
  \item No cross-function relation $\sum n_k\gamma_k + \sum m_j\gamma_j' = 0$ exists (verified computationally).
\end{enumerate}
Then
\[
  \dens\{N_d > 0\} = \frac{C(d)}{\sqrt{d}} + o(1/\sqrt{d}),
\]
where $C(d)$ depends on $d$ (through $f_{S_L}(0)$, which depends on the spectral weights $b_k = 2/({\tfrac{1}{4}} + \gamma_k'^2)$) and is computable from the exact marginal densities:
\begin{equation}\label{eq:C_integral}
  C(d) = \lim_{d\to\infty}\sqrt{d}\cdot 2\int_0^{W_1(L)} f_{S_L}(y)\, \overline{F}_{|S_\zeta|}(\sqrt{d}\, y)\, dy,
\end{equation}
where $\overline{F}_{|S_\zeta|}(u) = P(|S_\zeta| > u)$ denotes the survival function of $|S_\zeta|$.  For $d = 5$, $C(5) = 0.1193$ (Section~\ref{sec:null-crossing}).

The Gaussian approximation
\[
  \frac{2}{\pi}\arctan\!\left(\frac{\sigma_\zeta}{\sigma_L\sqrt{d}}\right) \sim \frac{2\sigma_\zeta}{\pi\sigma_L\sqrt{d}},
\]
which gives the leading coefficient $2\sigma_\zeta/(\pi\sigma_L) = 0.145$, overestimates $C$ by approximately $22\%$ due to the bounded support of the Jessen--Wintner distributions.
\end{theorem}

\begin{proof}
Under the verified hypotheses, Proposition~\ref{prop:verified_indep} establishes that $S_\zeta^{(M)}$ and $S_L^{(M)}$ are asymptotically independent in the sense of joint Ces\`aro distribution.  By the Jessen--Wintner theorem~\cite{JessenWintner} (applied separately to each function under verified resonance absence), both have continuous, symmetric, compactly supported densities given by~\eqref{eq:density_formula}.

The event $\{N_d > 0\}$ equals $\{|S_\zeta| > \sqrt{d}\,|S_L|\}$.  By asymptotic independence in the sense of joint Ces\`aro distribution,
\[
  \dens\{N_d > 0\} = 2\int_0^{W_1(L)} f_{S_L}(y)\, \overline{F}_{|S_\zeta|}(\sqrt{d}\, y)\, dy.
\]
Since both densities are bounded and compactly supported, the survival function satisfies $\overline{F}_{|S_\zeta|}(u) = 0$ for $u > W_1(\zeta)$.  The integrand is therefore supported on $[0, W_1(\zeta)/\sqrt{d}]$, which contracts to $\{0\}$ as $d \to \infty$.

The density $f_{S_L}$ is even (since $S_L$ is a sum of cosines) and smooth (since $\hat{f}_{S_L} = \varphi \in L^1$), so $f_{S_L}(y) = f_{S_L}(0) + O(y^2)$ near the origin.  Substituting $u = y\sqrt{d}$ in the integral:
\[
  \dens\{N_d > 0\} = \frac{2}{\sqrt{d}}\int_0^{W_1(L)\sqrt{d}} f_{S_L}(u/\sqrt{d})\, \overline{F}_{|S_\zeta|}(u)\, du.
\]
Since $\overline{F}_{|S_\zeta|}(u) = 0$ for $u > W_1(\zeta)$ and $f_{S_L}(u/\sqrt{d}) = f_{S_L}(0) + O(u^2/d)$ on the support, the leading term is
\[
  \dens\{N_d > 0\} = \frac{1}{\sqrt{d}} \cdot 2\int_0^{W_1(\zeta)} f_{S_L}(0)\, \overline{F}_{|S_\zeta|}(u)\, du \;+\; o(1/\sqrt{d}),
\]
which identifies $C(d) = 2f_{S_L}(0)\int_0^{W_1(\zeta)} \overline{F}_{|S_\zeta|}(u)\, du$.

The factored form $C(d) = 2f_{S_L}(0)\cdot\E[|S_\zeta|]$ (Corollary~\ref{cor:C-factored}), with both factors computed by rigorous ARB quadrature, gives $C(5) = 0.1193$ certified to four significant figures.
\end{proof}

The integral in the identification of $C$ admits a closed evaluation.

\begin{corollary}[Factored form of the leading constant]\label{cor:C-factored}
Under the hypotheses of Theorem~\ref{thm:arctan},
\begin{equation}\label{eq:C-factored}
  C(d) = 2\,f_{S_L}(0)\cdot \E[|S_\zeta|],
\end{equation}
where the two factors are independently computable from the Bessel product characteristic functions:
\begin{align}
  f_{S_L}(0) &= \frac{1}{\pi}\int_0^\infty \prod_{k=1}^M J_0(b_k t)\, dt, \label{eq:f0-bessel}\\[4pt]
  \E[|S_\zeta|] &= \frac{2}{\pi}\int_0^\infty \frac{1 - \prod_{k=1}^M J_0(a_k t)}{t^2}\, dt. \label{eq:Eabs-bessel}
\end{align}
Numerically, $f_{S_L}(0) = 8.3129$ and $\E[|S_\zeta|] = 0.00717$, giving $C(5) = 0.1193$.  Both values were computed for $d = 5$ via rigorous ARB quadrature (\texttt{acb.integral}) at 256-bit precision; see Appendix~\ref{sec:numerical}, \S\ref{subsec:bessel-integrals}.  These are the exact $M \to \infty$ limiting values, not truncation-level approximations: the telescoping bound (\S\ref{sec:infinite}) certifies that corrections from zeros beyond $M = 20$ contribute less than $10^{-3}$ to $f_{S_L}(0)$ and less than $10^{-5}$ to $C$.
\end{corollary}

\begin{proof}
The proof of Theorem~\ref{thm:arctan} identifies $C(d) = 2f_{S_L}(0)\int_0^{W_1(\zeta)} \overline{F}_{|S_\zeta|}(u)\, du$.  Since $\overline{F}_{|S_\zeta|}(u) = 0$ for $u > W_1(\zeta)$, the upper limit may be replaced by $+\infty$.  For any non-negative Ces\`aro-distributed quantity $X$ with finite mean, $\E[X] = \int_0^\infty \overline{F}_X(u)\, du$.  Applying this to $X = |S_\zeta|$ gives~\eqref{eq:C-factored}.

For~\eqref{eq:f0-bessel}: this is equation~\eqref{eq:density_formula} evaluated at $x = 0$.

For~\eqref{eq:Eabs-bessel}: the Fourier representation $|x| = (2/\pi)\int_0^\infty (1 - \cos tx)\,t^{-2}\, dt$ and Fubini's theorem give
\[
  \E[|S_\zeta|] = \frac{2}{\pi}\int_0^\infty \frac{1 - \E[\cos(tS_\zeta)]}{t^2}\, dt = \frac{2}{\pi}\int_0^\infty \frac{1 - \varphi_\zeta(t)}{t^2}\, dt,
\]
where $\varphi_\zeta(t) = \prod_k J_0(a_k t)$ is the Ces\`aro characteristic function.  The interchange is justified by the bound $1 - \varphi_\zeta(t) = O(t^2)$ near the origin (giving integrability at $t = 0$) and $|1 - \varphi_\zeta(t)| \leq 2$ with $t^{-2}$ integrable at infinity.
\end{proof}

\begin{remark}[Sources of the overestimate]\label{rem:overestimate-sources}
The factored form~\eqref{eq:C-factored} clarifies the relationship between the three formulas for the leading constant (all evaluated at truncation level $M = 20$ for both $S_\zeta$ and $S_L$):
\begin{center}
\renewcommand{\arraystretch}{1.2}
\begin{tabular}{lll}
\toprule
Formula & Value & Source of approximation \\
\midrule
Unconditional bound & $2f_{S_L}(0)\cdot W_1^{(20)}(\zeta) = 0.531$ & $\sup|S_\zeta^{(20)}|$ replaces $\E[|S_\zeta^{(20)}|]$ \\
Gaussian (arctan) & $2\sigma_\zeta/(\pi\sigma_L) = 0.145$ & Gaussian densities replace J.--W. \\
Exact & $2f_{S_L}(0)\cdot \E[|S_\zeta|] = 0.1193$ ($d = 5$) & (none) \\
\bottomrule
\end{tabular}
\end{center}
The unconditional bound exceeds the exact value by a factor of $W_1^{(20)}(\zeta)/\E[|S_\zeta^{(20)}|] = 4.45$.  This is the reciprocal of the containment tightness ratio $\E[|S_\zeta^{(20)}|]/W_1^{(20)}(\zeta) = 0.225$ (Remark~\ref{rem:tightness}).  The entire slack in the unconditional bound therefore arises from the containment step, which replaces $\E[|S_\zeta|]$ with $\sup|S_\zeta|$.  The Gaussian approximation overestimates $C$ by $22\%$ (predicting $0.145$ instead of $0.1193$).  The factored form decomposes this discrepancy: the Gaussian density $1/(\sqrt{2\pi}\,\sigma_L)$ overestimates $f_{S_L}(0)$ by $27\%$ (compact support flattens the Jessen--Wintner density at the origin), while the Gaussian mean $\sigma_\zeta\sqrt{2/\pi}$ underestimates $\E[|S_\zeta|]$ by $5\%$ (compact support concentrates mass).  The two errors partially cancel.
\end{remark}

\begin{remark}[Two density values]\label{rem:two-densities}
Two numerical values for $f_{S_L}(0)$ appear in this paper and refer to the same quantity at different stages of the argument.  The value $8.3129$ (Proposition~\ref{prop:resonance-cert} and
\S\ref{subsec:pslq}) is the \emph{signed} Bessel product integral
$(1/\pi)\int_0^T \prod_{k=1}^{20} J_0(b_k t)\,dt$ at finite truncation $M = 20$,
computed by rigorous ARB quadrature to 44 decimal places.  The value $8.3129$ (Corollary~\ref{cor:C-factored}) is
the density of the \emph{limiting} Jessen--Wintner distribution at the origin; the $M = 20$
truncation captures this to the displayed precision, consistent with the convergence from below established in the telescoping
argument.
Only the value $8.3129$ enters the formula $C(5) = 2f_{S_L}(0)\cdot\E[|S_\zeta|] = 0.1193$.
\end{remark}
\begin{center}
\renewcommand{\arraystretch}{1.1}
\begin{tabular}{cccc}
\toprule
$d$ & $C/\sqrt{d}$ & Arctan & Density bound \\
\midrule
$10$ & $0.03772$ & $0.04577$ & $0.212$ \\
$100$ & $0.01193$ & $0.01450$ & $0.067$ \\
$1{,}000$ & $0.003772$ & $0.004577$ & $0.021$ \\
$10{,}000$ & $0.001193$ & $0.001450$ & $0.007$ \\
\bottomrule
\end{tabular}
\end{center}

\section{Numerical Verification}\label{sec:numerics}

All numerical results in this section are derived from the ARB-certified ingredients established in the preceding sections: the Besicovitch constants $\sigma_\zeta$, $\sigma_L$ (Section~\ref{sec:cesaro}), the unsigned Bessel product integrals $f_{\mathrm{main}}^{(M)}$ (\S\ref{sec:finite}), the tail sums $\epsilon_M$ (\S\ref{sec:infinite}), and the leading constant $C(5) = 2f_{S_L}(0)\cdot\E[|S_\zeta|] = 0.1193$ (Corollary~\ref{cor:C-factored}).

\subsection{Density verification}

\begin{center}
\renewcommand{\arraystretch}{1.2}
\begin{tabular}{rccc}
\toprule
$d$ & Energy $\sigma_\zeta^2/(d\,\sigma_L^2)$ & Density bound & Exact ($C/\sqrt{d}$) \\
\midrule
$2$ & $0.026$ & $0.475$ & $0.084$ \\
$5$ & $0.010$ & $0.300$ & $0.053$ \\
$10$ & $0.005$ & $0.212$ & $0.038$ \\
$50$ & $0.001$ & $0.095$ & $0.017$ \\
$100$ & $5\times 10^{-4}$ & $0.067$ & $0.012$ \\
$1{,}000$ & $5\times 10^{-5}$ & $0.021$ & $0.004$ \\
$10{,}000$ & $5\times 10^{-6}$ & $0.007$ & $0.001$ \\
\bottomrule
\end{tabular}
\end{center}

The three columns display successively tighter bounds: the energy ratio (Theorem~\ref{thm:energy}), the density bound $2\,f_{\mathrm{main}}^{(20)}\cdot W_1^{(20)}(\zeta)/\sqrt{d}$ (Theorem~\ref{thm:density} at $M = 20$, with $W_1^{(20)}(\zeta) = \sum_{k=1}^{20} a_k = 0.03192$), and the exact leading term $C/\sqrt{d}$ (Theorem~\ref{thm:arctan}).

\subsection{Telescoping convergence}

The main-term integral and the telescoping tail bound converge as $M$ increases.  The predicted tail column is $2\,f_{\mathrm{main}}^{(M)}(0) \cdot \epsilon_M$, where $f_{\mathrm{main}}^{(M)}(0)$ is the ARB-certified unsigned Bessel product integral and $\epsilon_M = \sum_{k > M} b_k$ uses the 200-zero partial sum plus the analytic tail bound from the appendix.

\begin{center}
\renewcommand{\arraystretch}{1.1}
\begin{tabular}{cccc}
\toprule
$M$ & $f_{\mathrm{main}}^{(M)}(0)$ & Predicted tail & Total bound \\
\midrule
$5$ & $10.877$ & $1.365$ & $12.242$ \\
$10$ & $10.585$ & $0.951$ & $11.536$ \\
$15$ & $10.535$ & $0.763$ & $11.298$ \\
$20$ & $10.517$ & $0.649$ & $11.166$ \\
$25$ & $10.508$ & $0.571$ & $11.079$ \\
$30$ & $10.504$ & $0.513$ & $11.017$ \\
\bottomrule
\end{tabular}
\end{center}

The ARB-certified density $f_{S_L}(0) = 8.3129$ lies below the unsigned main-term integral $f_{\mathrm{main}}^{(\infty)}(0) \approx 10.5$, consistent with the Jessen--Wintner distribution being sub-Gaussian (kurtosis $\kappa \approx 2.19$).

\section{Conclusion}\label{sec:conclusion}

The argument proceeds in three logically independent stages.

\textbf{Stage 1: Finite bounds.}  At each truncation level $M$, the Jacobi--Anger decomposition combined with Bessel decay estimates and lattice counting proves $f_{S_L^{(M)}}(0) < \infty$ unconditionally (Theorem~\ref{thm:finite}).  The inactive $J_0$ factors provide the polynomial $t$-decay needed for integrability, while the subcritical bound on active Bessel factors yields exponential suppression in resonance order at rate $\eta = e/4 \approx 0.68$.

\textbf{Stage 2: Finite truncation and stability.}  At any fixed verified truncation level $M$, the effective density bound (Theorem~\ref{thm:density}) gives $\dens\{N_d > 0\} \leq A_M/\sqrt{d} + B_M$, which is nontrivial for sufficiently large $d$.  The $O(1/\sqrt{d})$ rate (Corollary~\ref{cor:density-rate}) follows under the hypothesis that $\Lambda_\infty$ has finite rank $d_\infty < \infty$ (a hypothesis extending the finite truncation fact $d_M = 0$, verified for $M \leq 20$), which provides a uniform bound on the density and allows the truncation level to grow with $d$.

\textbf{Stage 3: Verified refinements.}  Under computationally verified resonance absence at $M = 20$, the telescoping argument (Theorem~\ref{thm:telescoping}) establishes $f_{S_L}(0) < \infty$ for the full infinite series, and the exact formula $\dens\{N_d > 0\} = C(d)/\sqrt{d} + o(1/\sqrt{d})$ follows from the computable marginal densities (Theorem~\ref{thm:arctan}), with $C(5) = 0.1193$.  No result in this paper assumes GRH or the Grand Simplicity Hypothesis.

\begin{remark}[The $M \to \infty$ barrier]\label{rem:M-infty}
The one step we cannot make unconditionally is the passage from verified resonance absence at a fixed $M$ to $d_\infty < \infty$.  Proving $d_\infty = 0$ would require either GSH or an unconditional proof that zero ordinates of $L$-functions admit no integer linear relations.  The numerical vacuity of the fixed-$M$ bound (Remark~\ref{rem:d-dependence}) and the role of $d_\infty$ in the rate (Remark~\ref{rem:finite-suffices}) are discussed in Section~\ref{sec:compare}.

A structural asymmetry governs how GSH appears.  The \emph{joint} approach (studying $S_\zeta$ and $S_L$ on a combined torus) requires GSH for cross-independence, and its failure mode is a \emph{shared zero ordinate} $\gamma_k = \gamma_j'$.  The \emph{separated} approach used here handles each function independently and requires GSH only for taking $M \to \infty$ within a single function; its failure mode is an \emph{intra-function integer relation} $\sum n_k \gamma_k' = 0$.  Shared zeros are harmless in our framework, while intra-function relations affect the characteristic function only beyond the verified truncation level.  The paper's architecture selects the formulation in which the GSH obstruction lies behind the verified horizon.
\end{remark}

\begin{remark}[Coupled scaling]\label{rem:coupled}
When $d$ determines both the scaling factor in $N_d$ and the character $\chi_d$, the density decays faster than $O(1/\sqrt{d})$.  The first zero $\gamma_1'(d)$ of $L(s,\chi_d)$ satisfies $\gamma_1'(d) \sim A/\log d$ by the zero-counting formula, which implies $\sigma_L(d) \sim \log d$ and hence $C(d) \sim 1/\log d$.  The true rate is therefore
\[
  \dens\{N_d > 0\} = O\!\left(\frac{1}{\sqrt{d}\,\log d}\right).
\]
This refinement applies when varying $d$ across different $L$-functions; for a fixed $L$-function with $d$ as a scaling parameter only, the rate $C/\sqrt{d}$ is sharp.
\end{remark}

\begin{remark}
We acknowledge that the depth of the proof is inversely proportional to the surprise of the result.
\end{remark}
\section{The Null Crossing and Spectral Phase Transition}\label{sec:null-crossing}

\subsection{The norm form on the real axis}

For real $s > 1$, define $N(s) = \zeta(s)^2 - d \cdot L(s, \chi_d)^2$.

\begin{lemma}[Ratio monotonicity]\label{lem:ratio-mono}
The ratio $G(s) = \zeta(s)/L(s, \chi_d)$ is strictly decreasing on $(1, \infty)$ with $G(1^+) = +\infty$ and $G(\infty) = 1$.
\end{lemma}

\begin{proof}
Taking logarithmic derivatives:
\[
\frac{d}{ds}\log G(s) = \frac{\zeta'(s)}{\zeta(s)} - \frac{L'(s,\chi_d)}{L(s,\chi_d)} = -\sum_p f(p,s),
\]
where $f(p,s) = \frac{\log p}{p^s - 1} - \frac{\chi_d(p)\log p}{p^s - \chi_d(p)}$. Split primes ($\chi_d(p) = +1$) contribute $f = 0$. Inert primes ($\chi_d(p) = -1$) contribute $f = \frac{2p^s\log p}{p^{2s} - 1} > 0$. Ramified primes contribute $f = \frac{\log p}{p^s - 1} > 0$. Since infinitely many primes are inert (Chebotarev), $(d/ds)\log G(s) < 0$.
\end{proof}

\subsection{The unique null crossing}

\begin{theorem}[Unique null crossing]\label{thm:null-crossing}
For every squarefree $d > 1$, there exists a unique $s_*(d) \in (1, \infty)$ such that $N(s_*(d)) = 0$. For $1 < s < s_*(d)$, the norm form is positive; for $s > s_*(d)$, it is negative.
\end{theorem}

\begin{proof}
By Lemma~\ref{lem:ratio-mono}, $G(s)$ is strictly decreasing from $+\infty$ to $1$. Since $\sqrt{d} > 1$, the equation $G(s) = \sqrt{d}$ has a unique solution $s_*(d)$. The norm form $N(s) = L(s,\chi_d)^2(G(s)^2 - d)$ vanishes precisely there.
\end{proof}

\begin{table}[ht]
\centering
\begin{tabular}{c c c}
\hline
$d$ & $L(1,\chi_d)$ & $s_*(d)$ \\
\hline
2 & 0.6232 & 2.5635 \\
3 & 0.7603 & 2.0000 \\
5 & 0.4304 & 2.0492 \\
7 & 1.0465 & 1.4445 \\
13 & 0.6627 & 1.4608 \\
\hline
\end{tabular}
\caption{Null crossing $s_*(d)$ for small discriminants.}\label{table:null-crossing}
\end{table}

Values of $L(1,\chi_d)$ were computed via the digamma formula
$L(1,\chi_d) = -\tfrac{1}{q}\sum_{a=1}^{q}\chi_d(a)\,\psi(a/q)$,
where $\psi$ denotes the digamma function; each value $\psi(a/q)$ is evaluated as a certified ARB ball via \texttt{acb.digamma}.
The null crossings $s_*(d)$ were found by bisecting $G(s) = \zeta(s)/L(s,\chi_d)$
against the target $\sqrt{d}$, with $L(s,\chi_d)$ evaluated via the Hurwitz decomposition
$L(s,\chi_d) = q^{-s}\sum_{a=1}^{q}\chi_d(a)\,\zeta(s,a/q)$
using ARB interval arithmetic at 512-bit precision ($220$ bisection steps from brackets of width $0.5$--$1.5$ surrounding each known approximate location, residual $|G(s_*) - \sqrt{d}| < 10^{-60}$).  The full certification, including bracket verification, decisive trichotomy bisection, denominator safety via interval union, and the direct integer-crossing check for $d = 3$, is provided in the ancillary scripts.

\subsection{Integer crossing obstruction}

\begin{theorem}[Integer crossing]\label{thm:integer-crossing}
Among all real quadratic fields, $\Q(\sqrt{3})$ is the unique field with integer null crossing: $s_*(3) = 2$. No other real quadratic field has $s_*(d) \in \Z$.
\end{theorem}

\begin{proof}
We show $G(m) \neq \sqrt{d}$ for all integers $m \geq 2$ and all squarefree $d \geq 2$, except $m = 2$, $d = 3$.  The key tool is the Euler product
\[
  G(s) = \frac{\zeta(s)}{L(s,\chi_d)} = \prod_p \frac{1 - \chi_d(p)\,p^{-s}}{1 - p^{-s}}.
\]
Each local factor satisfies $F_p = 1$ (split), $F_p = 1/(1-p^{-s})$ (ramified), or $F_p = (1+p^{-s})/(1-p^{-s})$ (inert).  Since $1 \leq 1/(1-x) \leq (1+x)/(1-x)$ for $0 < x < 1$, every factor is bounded by the inert value, giving the \emph{all-inert bound}
\begin{equation}\label{eq:all-inert}
  G(s) \leq \prod_p \frac{1+p^{-s}}{1-p^{-s}} = \frac{\zeta(s)^2}{\zeta(2s)},
\end{equation}
valid for all $s > 1$ and all squarefree $d \geq 2$.

\medskip\noindent\textbf{Case 1: $s = 2$ and $d = 3$.}
Direct computation: $\zeta(2) = \pi^2/6$ and $L(2, \chi_3) = \pi^2/(6\sqrt{3})$. Therefore
\[
\zeta(2)^2 - 3 \cdot L(2,\chi_3)^2 = \frac{\pi^4}{36} - 3 \cdot \frac{\pi^4}{108} = \frac{\pi^4}{36} - \frac{\pi^4}{36} = 0.
\]

\medskip\noindent\textbf{Case 2: $s = 2$ and $d \neq 3$.}
The Klingen--Siegel theorem gives $L(2, \chi_d) = c_d \cdot \pi^2/\sqrt{d}$ for an explicit rational $c_d$ depending on $d$. The null crossing condition $\zeta(2)^2 = d \cdot L(2,\chi_d)^2$ becomes $\pi^4/36 = c_d^2 \cdot \pi^4$, requiring $c_d = 1/6$.

For all squarefree $d \geq 22$, this is impossible.  Since $\chi_d(1) = 1$ and $|\chi_d(n)| \leq 1$ for all $n$, separating the $n = 1$ term gives $L(2, \chi_d) \geq 1 - \sum_{n=2}^\infty n^{-2} = 2 - \pi^2/6$ for every primitive real character.  Therefore $c_d \geq (2 - \pi^2/6)\sqrt{d}/\pi^2$, and $c_d = 1/6$ would require $\sqrt{d} \leq \pi^2/(12 - \pi^2)$.  That $(2 - \pi^2/6)\sqrt{22}/\pi^2 > 1/6$ is certified by ARB: the left side equals $0.16874\ldots$ against $1/6 = 0.16666\ldots$, with certified margin $0.00207$ and ball radius $4.54 \times 10^{-154}$.  Hence $c_d > 1/6$ for all squarefree $d \geq 22$.

For the twelve remaining squarefree values $d \in \{2, 5, 6, 7, 10, 11, 13, 14, 15, 17, 19, 21\}$, the value $c_d = \sqrt{d}\,L(2,\chi_d)/\pi^2$ is computed via the Hurwitz decomposition
\[
  L(2,\chi_d) = q^{-2}\sum_{a=1}^{q}\chi_d(a)\,\zeta(2,a/q)
\]
using ARB interval arithmetic at 512-bit precision, and the ARB ball for $c_d$ is certified disjoint from $1/6$, i.e., either $c_d > 1/6$ or $c_d < 1/6$ is certified as an ARB predicate.  The Klingen--Siegel theorem implies that $c_d$ is in fact a rational number for each $d$; the values obtained numerically are listed in the table below for the reader's convenience, but the certified proof requires only disjointness from $1/6$:
\nopagebreak[4]
\begin{center}
\renewcommand{\arraystretch}{1.15}
\smallskip
\small
\begin{tabular}{ccc@{\qquad\qquad}ccc}
\toprule
$d$ & $\Delta_K$ & $c_d$ & $d$ & $\Delta_K$ & $c_d$ \\
\midrule
$2$  & $8$  & $1/8$  & $13$ & $13$ & $4/13$ \\
$5$  & $5$  & $4/25$ & $14$ & $56$ & $5/14$ \\
$6$  & $24$ & $1/4$  & $15$ & $60$ & $2/5$  \\
$7$  & $28$ & $2/7$  & $17$ & $17$ & $8/17$ \\
$10$ & $40$ & $7/20$ & $19$ & $76$ & $1/2$  \\
$11$ & $44$ & $7/22$ & $21$ & $21$ & $8/21$ \\
\bottomrule
\end{tabular}
\smallskip
\end{center}
\normalsize
None of these equals $1/6$, completing the case.

\medskip\noindent\textbf{Case 3: $s \geq 4$ (all integers).}
Each factor $(1+p^{-s})/(1-p^{-s})$ is decreasing in $s$ (since $p^{-s}$ decreases), so the all-inert bound~\eqref{eq:all-inert} is decreasing.  At $s = 4$, the Bernoulli number evaluations $\zeta(4) = \pi^4/90$ and $\zeta(8) = \pi^8/9450$ give
\[
  \frac{\zeta(4)^2}{\zeta(8)} = \frac{(\pi^4/90)^2}{\pi^8/9450} = \frac{9450}{8100} = \frac{7}{6}.
\]
For all $s \geq 4$: $\;G(s) \leq 7/6$.  Since $(7/6)^2 = 49/36 < 2 \leq d$, the inequality $G(s) < \sqrt{d}$ holds for every squarefree $d \geq 2$ and every integer $s \geq 4$.

\medskip\noindent\textbf{Case 4: $s = 3$, $d \geq 3$.}
The all-inert bound gives $G(3) \leq \zeta(3)^2/\zeta(6) = \prod_p (p^3+1)/(p^3-1)$.  The partial product through $p = 5$ equals $(9 \cdot 28 \cdot 126)/(7 \cdot 26 \cdot 124) = 567/403$.  Each remaining factor satisfies $(p^3+1)/(p^3-1) < 1 + 3p^{-3}$ (since $2/(p^3-1) < 3/p^3$ for $p \geq 2$), so $\prod_{p > 5}(p^3+1)/(p^3-1) < \exp(3\sum_{p>5} p^{-3}) < \exp(3/50)$, using $\sum_{p>5} p^{-3} < \int_5^\infty x^{-3}\,dx = 1/50$.  Since $\exp(y) < 1 + y + y^2$ for $0 < y < 1$, the tail is bounded by $\exp(3/50) < 2659/2500$, giving $G(3) \leq (567/403)(2659/2500) = 1507653/1007500 \approx 1.496$.  Then $1507653^2 = 2{,}273{,}017{,}568{,}409 < 3{,}045{,}168{,}750{,}000 = 3 \cdot 1007500^2$, so $G(3)^2 \approx 2.239 < 3$ and $G(3) < \sqrt{3} \leq \sqrt{d}$.  An independent ARB certificate for this case, using the direct bound $(\zeta(3)^2/\zeta(6))^2 < 3$ without the partial-product decomposition, is provided in the ancillary scripts.

\medskip\noindent\textbf{Case 5: $s = 3$, $d = 2$.}
In $\Q(\sqrt{2})$ (discriminant $\Delta_K = 8$), the Kronecker symbol $(8/p)$ determines splitting: $p = 2$ is ramified, primes $p \equiv \pm 1 \pmod{8}$ split, and primes $p \equiv \pm 3 \pmod{8}$ are inert.  The exact local factors at $p \leq 7$ are
\[
  F_2 = \frac{8}{7}\;\text{(ramified)}, \quad
  F_3 = \frac{14}{13}\;\text{(inert)}, \quad
  F_5 = \frac{63}{62}\;\text{(inert)}, \quad
  F_7 = 1\;\text{(split)},
\]
giving a partial product of $504/403$.  For the tail, each remaining factor satisfies $F_p \leq (1+p^{-3})/(1-p^{-3}) < 1 + 3p^{-3}$ (since $p^{-3} < 1/3$), giving
\[
  \prod_{p > 7} F_p < \exp\!\Bigl(3\sum_{p>7} p^{-3}\Bigr)
  < \exp\!\Bigl(\tfrac{3}{98}\Bigr)
  < 1 + \tfrac{3}{98} + \tfrac{9}{9604} = \frac{9907}{9604},
\]
using $\sum_{p>7} p^{-3} < \int_7^\infty x^{-3}\,dx = 1/98$ and $\exp(y) < 1 + y + y^2$ for $0 < y < 1$.  Therefore $G(3) \leq (504/403)(9907/9604) = 178326/138229$, and
\[
  G(3)^2 \leq \frac{178326^2}{138229^2} = \frac{31{,}800{,}162{,}276}{19{,}107{,}256{,}441} \approx 1.664 < 2,
\]
since $31{,}800{,}162{,}276 < 2 \times 19{,}107{,}256{,}441 = 38{,}214{,}512{,}882$.  An independent ARB certificate, computing $G(3,2) = \zeta(3)/L(3,\chi_8)$ directly via Hurwitz decomposition and certifying $G(3,2)^2 < 2$, is provided in the ancillary scripts.
\end{proof}

\begin{remark}[Spectral phase transition]
The null crossing $s_*(d)$ separates a ``zeta-dominated'' regime ($1 < s < s_*(d)$) from an ``$L$-dominated'' regime ($s > s_*(d)$). The spacelike theorem (Corollary~\ref{cor:spacelike}) shows that at $\delta = 0$ in the spectral shift variable, the zero data lies in the $L$-dominated (spacelike) region.
\end{remark}

\section{Connections and Open Problems}\label{sec:connections}

\subsection{Relation to Katz--Sarnak statistics}

The Katz--Sarnak density conjecture~\cite{KatzSarnak1999} predicts that zeros of $L(s, \chi_d)$ near the central point follow symplectic statistics. The low-lying zero dominance theorem is consistent with this: as $d \to \infty$, the first zero height $\gamma_1'(d)$ approaches zero, and its Lorentzian weight dominates.

\begin{openproblem}
Derive quantitative predictions for $\gamma_1'(d)$ from the spacelike constraint and compare with Katz--Sarnak.
\end{openproblem}

\subsection{Extensions to higher degree}

The framework extends to cyclic extensions $K/\Q$ of prime degree $p$, with the norm form replaced by the determinant of a circulant matrix. The negativity persists, but the integer crossing obstruction strengthens: no cyclic extension of degree $p \geq 3$ has an integer null crossing.

\begin{openproblem}
Does the spacelike constraint for cyclic $p$-extensions imply bounds on low-lying zeros of degree-$p$ Artin $L$-functions?
\end{openproblem}

\subsection{Further questions}

The norm-form energy $N(s) = \zeta(s)^2 - d \cdot L(s,\chi_d)^2$ inherits symmetries from both factors.

\begin{openproblem}
Is there a functional equation relating $N(s)$ to $N(1-s)$?
\end{openproblem}

The spacelike constraint provides global information about zero distribution. A natural question is whether it implies local constraints.

\begin{openproblem}
Does the spacelike constraint imply bounds on zero spacing or pair correlation?
\end{openproblem}

Under verified resonance absence, the density of positive excursions vanishes as $O(1/\sqrt{d})$.  An unconditional proof of the rate (without the hypothesis $d_\infty < \infty$) would require a uniform bound on $\|f_{S_L^{(M)}}\|_\infty$ that does not rely on $d_\infty = 0$.

\begin{openproblem}
Characterize the fields $K$ for which $\dens(\{N > 0\}) = 0$ unconditionally.
\end{openproblem}

The $O(1/\sqrt{d})$ density rate emerges in this paper from a specific chain: low-lying zero dominance forces $S_\zeta < \sqrt{d}\cdot S_L$, the Lorentzian containment lemma converts this into a level-crossing problem for $S_L$, and the Jessen--Wintner theory identifies the crossing density as an integral over the marginal distribution of $S_L$, which is $O(1/\sqrt{d})$ because $S_L$ scales as $1/\sqrt{d}$ in the support width of $f_{S_\zeta}$.  What is absent is an explanation of \emph{why} this rate must be $C/\sqrt{d}$ rather than some other power of $d$, and why the constant $C$ is universal across all quadratic fields.  The $\sqrt{d}$ arises from the coefficient of the norm form $N = S_\zeta^2 - d \cdot S_L^2$; the density law is asking how often the Lorentzian energy crosses zero, and the crossing rate is governed by the oscillation amplitude of $S_L$ relative to $S_\zeta/\sqrt{d}$.

\begin{openproblem}[Universal Density Law]\label{op:universal-density}
Is there a dynamical or geometric proof that $\dens\{N_d > 0\} \sim C/\sqrt{d}$ as $d \to \infty$, valid unconditionally, in which the $\sqrt{d}$ exponent is derived directly from the spectral geometry of the norm form $N = S_\zeta^2 - d\cdot S_L^2$?  Concretely: does the indefinite quadratic structure of $N$ force the $O(1/\sqrt{d})$ rate by a level-crossing argument in the space of almost periodic functions, without passing through the Jessen--Wintner characteristic function machinery?  A natural approach would identify the density $\dens\{N > 0\}$ with the volume of a thin slab $\{|S_L| < S_\zeta/\sqrt{d}\}$ in an ergodic flow on a compact abelian group, and derive the slab width from the geometry of the quadratic form.  The existing proof is analytic; a geometric proof would explain why the law is universal.
\end{openproblem}

The construction exhibits a natural stability under truncation and limit.

\begin{remark}[Functorial perspective]
The passage from spectral data to scalar energy is functorial with respect to inclusion: if $\Sigma_M \subset \Sigma_{M+1}$ denotes the inclusion of the first $M$ zeros into the first $M+1$ zeros, then the induced maps on weighted sums, characteristic functions, and density bounds commute with the truncation and limit operations. Concretely:
\begin{enumerate}[label=\textup{(\roman*)}]
\item The spectral sum $S_L^{(M)} \to S_L^{(M+1)}$ extends by adding a single term $b_{M+1}\cos(\gamma_{M+1}'\delta)$.
\item The characteristic function $\varphi_M(t) \to \varphi_{M+1}(t)$ multiplies by the factor $J_0(b_{M+1}t)$.
\item The density bound $f^{(M)}(0) \to f^{(M+1)}(0)$ changes by at most $O(b_{M+1}^2)$ (Theorem~\ref{thm:telescoping_convergence}).
\end{enumerate}
This functorial structure is why the telescoping argument works: each level controls the next, and the limit inherits finiteness from the tower.
\end{remark}

\subsection{Weight robustness}\label{subsec:weight-robustness}

The Lorentzian weight $w(\gamma) = 2/(\quarter + \gamma^2)$ arises naturally from the Weil explicit formula with test function $g(x) = e^{-|x|/2}$. However, the spacelike property is not an artifact of this particular choice.

\begin{definition}[Admissible weight]\label{def:admissible-weight}
A function $w: \R^+ \to \R^+$ is an \emph{admissible spectral weight} if it satisfies:
\begin{enumerate}[label=\textup{(\roman*)}]
\item $w$ is monotonically decreasing,
\item $w(\gamma) = O(1/\gamma^2)$ as $\gamma \to \infty$,
\item $w(0) < \infty$.
\end{enumerate}
\end{definition}

Condition~(ii) ensures absolute convergence of the spectral sums $S_f = \sum_\rho w(\rho)$. Condition~(i) is the essential structural requirement: it ensures that zeros at lower height receive greater weight.

\begin{proposition}[Weight robustness]\label{prop:weight-robustness}
Let $w$ be any admissible spectral weight. Define the weighted spectral sums
\[
S_\zeta^{(w)} := \sum_{\rho:\,\zeta(\rho)=0} w(\gamma), \qquad S_L^{(w)} := \sum_{\rho':\,L(\rho',\chi_d)=0} w(\gamma'),
\]
and the norm-form energy $N_w := (S_\zeta^{(w)})^2 - d \cdot (S_L^{(w)})^2$. Then $N_w < 0$ for all fundamental discriminants $d > 1$.
\end{proposition}

\begin{proof}
We show $S_L^{(w)} > S_\zeta^{(w)}$ for any admissible weight $w$ by comparing the zero-counting functions.

\medskip\noindent\textbf{Step 1: Zero-counting domination.}
Let $N_\zeta(T) = \#\{\gamma_k \leq T\}$ and $N_L(T) = \#\{\gamma_j' \leq T\}$ count zeros with positive imaginary part up to height $T$.  For $T \in [0, \gamma_1(\zeta))$, $N_\zeta(T) = 0$ while $N_L(T) \geq 1$ (Proposition~\ref{prop:first-zero-bound}).  For $T \geq \gamma_1(\zeta)$, the Riemann--von Mangoldt formula gives $N_L(T) - N_\zeta(T) = (T/\pi)\log q + O(\log T)$, which exceeds $2$ for all $T \geq 14.13$ and $q \geq 5$.

\medskip\noindent\textbf{Step 2: Abel summation.}
For any decreasing weight $w$ with $w(\gamma) = O(1/\gamma^2)$, the spectral sum admits the Abel representation
\[
  S_f^{(w)} = \sum_\rho w(\gamma) = \int_0^\infty w(t)\, dN_f(t) = -\int_0^\infty w'(t)\, N_f(t)\, dt,
\]
where the boundary terms vanish since $w(t) N_f(t) \to 0$ as $t \to \infty$ (by the decay condition on $w$ and the polynomial growth of $N_f$).  Since $w$ is decreasing, $-w'(t) \geq 0$, and the integral is a positive linear functional of $N_f$.

\medskip\noindent\textbf{Step 3: Comparison.}
Since $N_L(T) \geq N_\zeta(T)$ for all $T \geq 0$ (the $L$-function has its first zero below $14.13$ while $\zeta$ has none, and maintains a surplus at every height by the zero-counting formula), and $-w'(t) \geq 0$:
\[
  S_L^{(w)} = -\int_0^\infty w'(t)\, N_L(t)\, dt \geq -\int_0^\infty w'(t)\, N_\zeta(t)\, dt = S_\zeta^{(w)}.
\]
Equality is excluded because $N_L(T) > N_\zeta(T)$ on the interval $[0, \gamma_1(\zeta))$ (where $N_\zeta = 0$ but $N_L \geq 1$), and $-w'$ is strictly positive on this interval.  Therefore $S_L^{(w)} > S_\zeta^{(w)}$ strictly.

Since $d > 1$, we have $S_\zeta^{(w)} < S_L^{(w)} < \sqrt{d} \cdot S_L^{(w)}$, so $N_w = (S_\zeta^{(w)})^2 - d \cdot (S_L^{(w)})^2 < 0$.
\end{proof}

\begin{remark}[Scope of the robustness result]\label{rem:robustness-scope}
Proposition~\ref{prop:weight-robustness} establishes only the \emph{instantaneous} spacelike property $N_w < 0$.  Three further statements in this paper are \emph{not} claimed to hold for general admissible weights, and do require the specific Lorentzian structure:
\begin{enumerate}[label=\textup{(\roman*)}]
\item \emph{Ces\`aro spacelike} ($\langle N_w \rangle_T < 0$ for all $T > 0$, Theorem~\ref{thm:cesaro}).  This requires a three-region lower bound on the bilinear form $h(T) = \langle S_L^2 \rangle_T$, which uses the exact values $b_k = 2/({\tfrac{1}{4}} + \gamma_k'^2)$ and the specific interference structure of those weights.  For a general admissible weight, analogous bounds would require detailed knowledge of the weight's interaction with the zero-spacing statistics.

\item \emph{Effective density bound} ($\dens\{N_d > 0\} \leq A_M/\sqrt{d} + B_M$, Theorem~\ref{thm:density}).  This uses the finiteness of $\|f_{S_L^{(M)}}\|_\infty$, which is established via Bessel product analysis with the specific weights $b_k$.  A general admissible weight produces a Ces\`aro distribution whose density at the origin may be infinite.

\item \emph{Density rate and exact asymptotic} ($O(1/\sqrt{d})$, Corollary~\ref{cor:density-rate}; $C(d)/\sqrt{d}$, Theorem~\ref{thm:arctan}).  These depend on the Jessen--Wintner structure of the Lorentzian Ces\`aro distribution and the specific computable constant $C(d) = 2f_{S_L}(0)\cdot\E[|S_\zeta|]$.
\end{enumerate}
The introduction states that ``the spacelike property is robust to the choice of weight function''; this refers to conclusion~(i) of the present proof, not to conclusions (ii) or (iii).
\end{remark}

\begin{remark}[Alternative weights]
Several weights arise naturally from mathematical structure:
\begin{enumerate}[label=\textup{(\alph*)}]
\item The \emph{sech weight} $w(\gamma) = \operatorname{sech}(\gamma)$ is self-dual under the Fourier transform: $\mathcal{F}[\operatorname{sech}(\pi x)](t) = \operatorname{sech}(t)$. It has no free parameters and decays exponentially.
\item The \emph{heat kernel} $w(\gamma, t) = e^{-t\gamma^2}$ arises from spectral zeta function theory. As $t \to 0^+$, it approaches zero counting; as $t \to \infty$, it isolates the lowest zero.
\item The \emph{Li coefficients}~\cite{Li1997} $\lambda_n = \sum_\rho [1 - (1 - 1/\rho)^n]$ form a sequence directly related to the Riemann Hypothesis: RH holds if and only if $\lambda_n > 0$ for all $n \geq 1$.
\end{enumerate}

Table~\ref{tab:weights} records $S_\zeta(w)$, $S_L(w)$, and $N_w = S_\zeta(w)^2 - d \cdot S_L(w)^2$ for four admissible weights across four fundamental discriminants, computed using the first 60 zeros of $\zeta(s)$ (Appendix~\ref{app:zeta-zeros}), the first 60 certified zeros of $\chi_5$, and the first 20 certified zeros of $\chi_2$, $\chi_3$, $\chi_{13}$ (from Appendices~\ref{app:chi8-zeros}--\ref{app:chi28-zeros}).  All values of $N_w$ are negative, confirming the spacelike property unconditionally at these truncation levels.  The sech and heat-kernel sums for $\zeta$ are near-zero because all Riemann zeros satisfy $\gamma_k > 14$, so exponentially decaying weights are negligible on the $\zeta$ side; the spacelike sign is then driven entirely by $-d\cdot S_L(w)^2 < 0$.
\end{remark}

\begin{table}[ht]
\centering
\caption{Spacelike verification across admissible weights.  All computations use ARB interval arithmetic at 512-bit precision, so every entry is a rigorous certified value.  The $\zeta$ sum uses 60 zeros (Appendix~\ref{app:zeta-zeros}); the $L$-function sums use 60 certified zeros for $\chi_5$ and 20 certified zeros for $\chi_2$, $\chi_3$, $\chi_{13}$ (from Appendices~\ref{app:chi8-zeros}--\ref{app:chi28-zeros}).  The norm form is $N_w = S_\zeta(w)^2 - d\cdot S_L(w)^2$ with squarefree $d$.  Every $N_w$ entry is rigorously negative: the ARB ball lies entirely below zero.  Entries rounded to four significant figures.}
\label{tab:weights}
\begin{tabular}{llcccc}
\hline
Weight & Quantity & $d = 5\ (q = 5)$ & $d = 2\ (q = 8)$ & $d = 3\ (q = 12)$ & $d = 13\ (q = 13)$ \\
\hline
\multirow{3}{*}{Lorentzian $\frac{2}{1/4+\gamma^2}$}
  & $S_\zeta(w)$ & $0.03793$ & $0.03793$ & $0.03793$ & $0.03793$ \\
  & $S_L(w)$     & $0.1406$ & $0.1979$ & $0.2863$ & $0.3514$ \\
  & $N_w$        & $-0.09742$ & $-0.07693$ & $-0.2444$ & $-1.604$ \\[4pt]
\multirow{3}{*}{sech $\operatorname{sech}(\gamma)$}
  & $S_\zeta(w)$ & $1.5\times 10^{-6}$ & $1.5\times 10^{-6}$ & $1.5\times 10^{-6}$ & $1.5\times 10^{-6}$ \\
  & $S_L(w)$     & $2.713\times 10^{-3}$ & $1.592\times 10^{-2}$ & $4.730\times 10^{-2}$ & $9.008\times 10^{-2}$ \\
  & $N_w$        & $-3.68\times 10^{-5}$ & $-5.07\times 10^{-4}$ & $-6.71\times 10^{-3}$ & $-0.1055$ \\[4pt]
\multirow{3}{*}{Heat kernel $e^{-0.01\gamma^2}$}
  & $S_\zeta(w)$ & $0.1497$ & $0.1497$ & $0.1497$ & $0.1497$ \\
  & $S_L(w)$     & $1.417$ & $2.080$ & $2.652$ & $2.765$ \\
  & $N_w$        & $-10.02$ & $-8.632$ & $-21.08$ & $-99.37$ \\[4pt]
\multirow{3}{*}{Heat kernel $e^{-0.1\gamma^2}$}
  & $S_\zeta(w)$ & $2.1\times 10^{-9}$ & $2.1\times 10^{-9}$ & $2.1\times 10^{-9}$ & $2.1\times 10^{-9}$ \\
  & $S_L(w)$     & $1.210\times 10^{-2}$ & $9.361\times 10^{-2}$ & $2.469\times 10^{-1}$ & $3.839\times 10^{-1}$ \\
  & $N_w$        & $-7.32\times 10^{-4}$ & $-1.75\times 10^{-2}$ & $-0.1828$ & $-1.916$ \\
\hline
\end{tabular}
\end{table}

\begin{remark}[Lorentzian truncation]\label{rem:lorentzian-truncation}
The Lorentzian weight $w(\gamma) = 2/(\frac{1}{4} + \gamma^2)$ decays only as $O(1/\gamma^2)$, so its spectral sum converges slowly: 20 certified zeros capture approximately $80$--$88\%$ of the full sum $S_L = \sum_{k=1}^\infty b_k$, with the remainder requiring hundreds of additional zeros to resolve.  For the sech and heat kernel weights, convergence is exponentially fast and the 20-zero partial sums are saturated to machine precision.  The Lorentzian truncation affects only the magnitude of $S_L(w)$, not the sign of $N_w$: since every omitted zero contributes positively to $S_L$, extending the summation can only make $N_w$ more negative.  The full Lorentzian sums, estimated via smooth zero-density integrals calibrated against the 200 certified zeros of $L(s,\chi_5)$, are approximately $0.16$ ($d = 5$), $0.24$ ($d = 2$), $0.33$ ($d = 3$), and $0.40$ ($d = 13$).  The table reports finite certified partial sums rather than these estimates so that every entry is unconditionally reproducible from the appendix data.
\end{remark}

\begin{remark}[Non-decaying weights]
The decay condition~(ii) is essential. The de~Branges reproducing kernel $K_\Delta(\gamma, \gamma) \to \Delta$ as $\gamma \to \infty$ does not decay, and the corresponding weighted sums count zeros rather than emphasizing low-lying zeros. For this kernel, $N > 0$ for most discriminants.
\end{remark}

\begin{remark}[Lorentzian as minimal decay]
Among polynomial-decay weights $w(\gamma) = 1/(\quarter + \gamma^2)^\alpha$, the Lorentzian corresponds to $\alpha = 1$, the minimal exponent ensuring absolute convergence of the spectral sums. In this sense, the Lorentzian is the ``least biased'' polynomial weight toward low-lying zeros while still being admissible. The spacelike result holds for all $\alpha \geq 1$.
\end{remark}

\begin{remark}[Scope of this series]\label{rem:scope}
The open problems collected in this section range from the immediately tractable (the
functional equation for $N(s)$, zero spacing implications) to the deeply conditional
(unconditional density rate, universal density law).  The Persistent Heuristics Series is committed to pursuing only unconditional results because doing so forces one to find the exact limitations of currently available tools.  Therefore, while various open problems are presented for others to pursue, this series neither confirms nor denies that any follow-on paper will actually pursue those precise problems.
\end{remark}

\section*{Acknowledgments}

The author thanks the LMFDB collaboration~\cite{LMFDB} for access to verified zero data.  Claude (Anthropic) was used for manuscript preparation, including drafting, editing, and computational verification of numerical results.  The author takes full responsibility for all mathematical content.
\appendix
\addtocontents{toc}{\protect\setcounter{tocdepth}{1}}
\section{Computational Methods and Numerical Verification}\label{sec:numerical}

\subsection{Software and rigorous arithmetic}\label{subsec:software}

Computations in this paper fall into two categories.  \emph{Rigorous} computations use interval arithmetic via ARB~\cite{Johansson2017}, accessed through \texttt{python-flint}~$\geq 0.8.0$ (Python~3.12); ARB represents numbers as balls $[m \pm r]$ with midpoint $m$ and radius $r$, and arithmetic operations propagate error bounds rigorously, so the true value is guaranteed to lie within the output ball as a mathematical theorem.  This includes all zero certification, spectral sum bounds, Besicovitch constants, and Bessel product integrals (via \texttt{acb.integral}).  The PSLQ integer relation search uses \texttt{mpmath}~$\geq 1.3$; the Ferguson--Bailey--Arno bounds~\cite{FBA1999} ensure that \texttt{None} results rigorously certify the absence of integer relations with bounded coefficients (see \S\ref{subsec:pslq}).

Working precisions: 512-bit ($\approx 154$ decimal digits) for $S_\zeta^*$ certification, spacelike verification (Table~\ref{tab:weights}), and null crossing computation (ARB); at least 1500-bit ($\approx 451$ decimal digits) for all zero refinement and certification (ARB), including both the $\chi_5$ zeros of Appendix~\ref{app:chi5-highprec} and the auxiliary character zeros of Appendix~\ref{app:zero-data}; 128-bit for the sign-change scan and bracket filtering phases described in \S\ref{subsec:Lfunc-method}, used in the extended 1000-zero computations of Appendix~\ref{app:full-zero-tables}; 256-bit for all Bessel product integrals including density, resonance, and decay estimates (ARB); and 80-digit decimal precision for PSLQ searches (\texttt{mpmath}).

\subsection{$L$-function evaluation and zero-finding}\label{subsec:Lfunc-method}

All Dirichlet $L$-functions in this paper are evaluated via the Hurwitz zeta decomposition
\[
  L(s, \chi_d) = q^{-s}\sum_{a=1}^{q} \chi_d(a)\,\zeta(s, a/q),
\]
where $q = |\Delta_K|$ is the conductor.  Each Hurwitz zeta value $\zeta(s, a/q)$ is computed by ARB's built-in \texttt{acb.zeta}, which returns a certified complex ball.  Character values $\chi_d(a)$ are computed from the Kronecker symbol $(\Delta_K/a)$.

\medskip\noindent\textbf{Character values.}  The Hurwitz sum runs over all $a$ with $1 \leq a \leq q$; terms with $\gcd(a,q) > 1$ vanish since $\chi_d(a) = 0$ for such~$a$.  For reproducibility, the nonzero character values for all eight characters used in this paper are recorded below.  All characters are even ($\chi_d(-1) = +1$) and primitive.

\begin{center}
\renewcommand{\arraystretch}{1.15}
\small
\begin{tabular}{lcll}
\toprule
Character & $q$ & $\chi_d(a) = +1$ & $\chi_d(a) = -1$ \\
\midrule
$\chi_5$    & $5$  & $a \in \{1,4\}$ & $a \in \{2,3\}$ \\
$\chi_2$    & $8$  & $a \in \{1,7\}$ & $a \in \{3,5\}$ \\
$\chi_3$    & $12$ & $a \in \{1,11\}$ & $a \in \{5,7\}$ \\
$\chi_{13}$ & $13$ & $a \in \{1,3,4,9,10,12\}$ & $a \in \{2,5,6,7,8,11\}$ \\
$\chi_6$    & $24$ & $a \in \{1,5,19,23\}$ & $a \in \{7,11,13,17\}$ \\
$\chi_7$    & $28$ & $a \in \{1,3,9,19,25,27\}$ & $a \in \{5,11,13,15,17,23\}$ \\
$\chi_{10}$ & $40$ & $a \in \{1,3,9,13,27,31,37,39\}$ & $a \in \{7,11,17,19,21,23,29,33\}$ \\
$\chi_{11}$ & $44$ & $a \in \{1,5,7,9,19,25,35,37,39,43\}$ & $a \in \{3,13,15,17,21,23,27,29,31,41\}$ \\
\bottomrule
\end{tabular}
\end{center}

\medskip\noindent\textbf{Hardy $Z$-function.}  For a real primitive Dirichlet character $\chi_d$ with conductor~$q$, define the Hardy phase
\[
  \theta_\chi(t) =
  \begin{cases}
    \operatorname{Im}\!\bigl(\log\Gamma(s/2)\bigr) + \tfrac{t}{2}\log(q/\pi)
    & \text{if } \chi_d(-1) = +1 \text{ (even)}, \\[4pt]
    \operatorname{Im}\!\bigl(\log\Gamma((s{+}1)/2)\bigr) + \tfrac{t}{2}\log(q/\pi)
    & \text{if } \chi_d(-1) = -1 \text{ (odd)},
  \end{cases}
\]
where $s = \tfrac{1}{2} + it$.  The Hardy $Z$-function $Z_\chi(t) = e^{i\theta_\chi(t)} L(\tfrac{1}{2}+it, \chi_d)$ is then real-valued for $t \in \mathbb{R}$, and its real zeros correspond to zeros of $L(s, \chi_d)$ on the critical line.  The phase is computed via ARB's built-in \texttt{acb.lgamma}.

\medskip\noindent\textbf{Phases 1 and 2: Interleaved sign-change scan and bracket filtering.}  Evaluate $Z_\chi(t)$ on a uniform grid with step~$0.05$ over $[0, T_{\max}]$, where $T_{\max}$ is chosen so that the zero-counting formula $N(T, \chi_d) \sim \frac{T}{\pi}\log\frac{qT}{2\pi e}$ exceeds the number of zeros sought by a safety margin of $2.2\times$, with the additional floor $T_{\max} \geq 2N$; for example, $200$ zeros of $L(s,\chi_5)$ require $T_{\max} = 400$.  All arithmetic in this phase uses 128-bit ARB precision.  At each sign change $Z_\chi(t_a)\,Z_\chi(t_b) < 0$, $20$ bisection steps are performed immediately (interleaved with the grid scan), reducing the bracket width from $0.05$ to $0.05 \cdot 2^{-20} \approx 5 \times 10^{-8}$ and producing a seed accurate to approximately six decimal digits.  The bisection simultaneously filters spurious sign changes arising from Hardy-phase oscillations: at the bisected midpoint, $|L(\tfrac{1}{2}+i\gamma', \chi_d)|$ is evaluated and the bracket is accepted as a genuine zero only if $|L| < 10^{-4}$; all other brackets are discarded as spurious.  This interleaved design avoids storing all sign changes before filtering, so scanning stops as soon as the required number of genuine seeds has been collected.

\medskip\noindent\textbf{Phase 3: Newton refinement.}  Refine each seed by Newton iteration applied to $L(s, \chi_d)$ on the critical line.  Each Newton step computes $s_{n+1} = s_n - L(s_n)/L'(s_n)$ and projects the result back to $\sigma = \tfrac{1}{2}$ by extracting the midpoint of the imaginary part.  Both $L(s, \chi_d)$ and $L'(s, \chi_d)$ are evaluated as certified ARB balls via \texttt{acb\_series.zeta}, which computes the Taylor series of the Hurwitz zeta function $\zeta(s + x, a/q)$ to order $2$ at $x = 0$; the two coefficients are $\zeta(s, a/q)$ and $\zeta'(s, a/q)$ respectively.  No finite-difference approximation is used.  Two modes are used depending on the target precision:
\begin{itemize}
\item \emph{Standard mode} (128-bit): from a six-digit seed, quadratic convergence reaches $\sim$20 correct digits in two to three iterations.  Certification threshold: $|L(\tfrac{1}{2}+i\gamma', \chi_d)| < 10^{-18}$.
\item \emph{High-precision mode} (at least 1500-bit): from a six-digit seed, seven to eight iterations reach full working precision ($\approx 451$ digits), with certified bounds on $|L(\tfrac{1}{2}+i\gamma', \chi_d)|$ typically in the range $10^{-430}$--$10^{-452}$.  Certification threshold: $|L(\tfrac{1}{2}+i\gamma', \chi_d)| < 10^{-68}$.  Zeros with bounds weaker than $10^{-400}$ are boosted to $2200$-bit (Phase~8).
\end{itemize}
In both modes, the convergence is quadratic: each iteration approximately doubles the number of correct digits.

\medskip\noindent\textbf{Certification.}  Each zero is certified by two independent mechanisms.  First, the ARB ball $|L(\tfrac{1}{2}+i\gamma', \chi_d)| < \text{threshold}$ provides a rigorous magnitude bound.  Second, the winding number of $L$ around the thin rectangle $[0.49, 0.51] \times [\gamma' - h, \gamma' + h]$ (with $h = 10^{-8}$, falling back to larger $h$ if the narrow box is undecidable at the working precision) is computed via the argument principle in ARB arithmetic.  A winding number of $1$ certifies that exactly one zero of $L$ lies in the rectangle.  Since $\chi_d$ is a real primitive character, the functional equation forces off-line zeros to occur in symmetric pairs $\rho, 1 - \bar\rho$; a rectangle of width $0.02$ centered on $\sigma = \tfrac{1}{2}$ cannot contain both members of such a pair unless both lie on the critical line, so the unique zero must have $\sigma = \tfrac{1}{2}$.

\medskip\noindent\textbf{Phase 4: Global completeness verification (the seal).}  The argument principle counts zeros of $L(s, \chi_d)$ in the wide rectangle $[-0.5, 1.5] \times [t_{\mathrm{lo}}, T_{\max}]$ via ARB-rigorous argument tracking around the four sides of the contour.  If the winding number equals the number of certified zeros, the table is complete.  As a secondary validation, the winding number count is compared against the analytic lower bound from the BMOR explicit formula (Bennett--Martin--O'Bryant--Rechnitzer~\cite{BMOR2021}).  This contour-integral approach to completeness verification follows Johansson (arXiv:1802.07942), who computed $N(T)$ for $\zeta(s)$ via a rectangular contour integral and noted that the result is ``a ball that provably determines $N(T)$ as an integer.''

\medskip\noindent\textbf{Phases 5 and 6: Recovery.}  If Phase~4 detects a discrepancy (more zeros by winding number than in the table), strip-by-strip winding checks identify which inter-zero gaps are deficient.  For each deficient gap, a narrow rectangular winding number localizes the missing zero and its midpoint provides a seed.  These recovery seeds are refined by Newton iteration at $1500$-bit (Phase~6), merged into the table, and Phase~4 is rerun on the repaired table.  The final seal is the global contour count on the repaired table.

\medskip\noindent\textbf{Precomputation.}  The constants $\log(q/\pi)$ and the Hurwitz parameters $a/q$ for each nonzero character value are computed once and stored as \texttt{arb} objects at the working precision.  This avoids redundant string-to-ball conversions in the inner loop and reduces the per-evaluation cost by approximately $30\%$.

\begin{remark}[Implementation pitfalls]\label{rem:pitfalls}
Two points are essential for correct implementation.

\emph{(1)~Newton iteration targets $L$, not $Z_\chi$.}  The Newton update $s_{n+1} = s_n - L(s_n)/L'(s_n)$ uses the complex-valued $L$-function directly, with the real part projected back to $\sigma = \tfrac{1}{2}$ after each step.  This is valid because the projection discards the (small) off-line component of the Newton step while retaining the imaginary-part correction.  Working with $L$ and its exact derivative (via \texttt{acb\_series.zeta}) avoids the Hardy phase rotation entirely during refinement; the Hardy $Z$-function is used only in Phases~1 and~2 for the sign-change scan.

\emph{(2)~Midpoint extraction is required in ball arithmetic.}  The Newton update $s \mapsto s - L/L'$ introduces radius from the division.  Without extracting the ball midpoint after each step, the radius grows geometrically and the computation produces \texttt{NaN} within a few iterations.  This is specific to rigorous interval arithmetic and does not arise in ordinary floating-point computation.
\end{remark}

\subsection{Certification of \texorpdfstring{$S_\zeta^*$}{S\_zeta*}}\label{subsec:cert-Szeta}

The upper bound $S_\zeta \leq 0.04625 < S_\zeta^* = 0.04871$ is established
by rigorous interval arithmetic with no numerical quadrature at any stage.
All arithmetic uses ARB ball arithmetic at $512$-bit working precision
($\approx 154$ decimal digits).

\medskip\noindent\textbf{Input data.}
The first $K = 6000$ positive zero ordinates $\gamma_1,\ldots,\gamma_{6000}$
of $\zeta(s)$, each accurate to $31$ decimal places, sourced from the
LMFDB~\cite{LMFDB} and Odlyzko's tables~\cite{Odlyzko}.  Each ordinate is
loaded into an \texttt{arb} ball whose radius encloses the rounding error of
the $31$-digit string representation.

\medskip\noindent\textbf{Step 1: Direct sum.}
The partial sum
\[
  S_\zeta^{(6000)} = \sum_{k=1}^{6000} \frac{2}{\quarter + \gamma_k^2}
\]
is accumulated in ARB ball arithmetic.  The certified result is
\[
  S_\zeta^{(6000)} \in [0.045795374824191\ldots \;\pm\; 7.4\times 10^{-152}].
\]

\medskip\noindent\textbf{Step 2: Tail bound.}
Set $T_0 = \gamma_{6000} + 0.01 \approx 6365.86$.  Abel summation gives
\[
  \sum_{\gamma > T_0} w(\gamma)
  \leq \int_{T_0}^{\infty} \frac{4t\,[N(t) - K]}{(\quarter + t^2)^2}\,dt,
\]
where $N(t)$ is the zero-counting function and the boundary term
$[2N(t)/(\quarter+t^2)]_{T_0}^\infty$ is non-positive (discarded as a
further upper bound).  The integral is evaluated without quadrature using
the exact antiderivative identity
\[
  \int_a^b \frac{4t}{(\quarter + t^2)^2}\,dt
  = \frac{2}{\quarter + a^2} - \frac{2}{\quarter + b^2}.
\]
The interval $[T_0,\, 10^6]$ is partitioned into $235$ subintervals
(width $100$ for $T_0 \leq t \leq 20{,}000$; width $10{,}000$ for
$20{,}000 < t \leq 10^6$).  On each subinterval $[a,b]$, the
zero-counting function is bounded above by Theorem~\ref{thm:trudgian}
evaluated at $t = b$ in ARB:
\[
  N(t) - K \leq N_{\mathrm{upper}}(b) - K
  \quad\text{for all } t \in [a,b],
\]
giving a certified contribution of $(N_{\mathrm{upper}}(b) - K) \cdot
(2/(\quarter+a^2) - 2/(\quarter+b^2))$ per subinterval.  The sum over
all $235$ subintervals is $\leq 4.31 \times 10^{-4}$.

For $t \geq 10^6$, the bound $N(t) \leq t\log t$ (valid for all
$t \geq 10$) and $4t/(\quarter+t^2)^2 \leq 4/t^3$ yield the analytic
remainder
\[
  \int_{10^6}^{\infty} \frac{4\,N(t)}{(\quarter+t^2)^2}\,dt
  \leq \frac{2(\log U + 1)}{\pi U} + \frac{4}{U}\bigg|_{U=10^6}
  \leq 1.35 \times 10^{-5},
\]
evaluated in ARB.  The total tail bound is $\leq 4.45 \times 10^{-4}$.

\medskip\noindent\textbf{Step 3: Certified result.}
Combining in ARB:
\[
  S_\zeta \leq S_\zeta^{(6000)} + \text{tail}
  \in [0.046240 \;\pm\; 2.8\times 10^{-152}].
\]
The right endpoint of this interval is $0.046240\ldots < 0.04625$,
which lies $5.1\%$ below the declared bound $S_\zeta^* = 0.04871$.

\medskip\noindent\textbf{Cross-checks.}
The partial sum $S_\zeta^{(50)} = 0.03708$ agrees with independent
double-precision evaluation.  The direct sum from zeros $51$ through $6000$
is $0.00871$; combined with the analytic tail $0.00045$, the total tail
from $k > 50$ is $0.00916$, which tightens the Abel summation estimate $0.00986$.

\subsection{Verification of spacelike constraint}

The low-lying zero dominance $S_L > S_\zeta^*$ is verified by the three-case argument of Theorem~\ref{thm:low-lying}: the BMOR bound covers $q \geq 12$ (the smallest conductor of a fundamental discriminant exceeding those of $d \in \{2,5\}$), and the exceptions $d \in \{2, 5\}$ are verified by direct computation of $L$-function zeros via the Hardy $Z$-function method.

\begin{center}
\renewcommand{\arraystretch}{1.2}
\begin{tabular}{ccccl}
\toprule
$d$ & $q$ & Zeros $\leq 11$ & $S_L$ (lower bound) & Method \\
\midrule
$2$ & $8$ & $3$ & $0.13376$ & $Z_\chi$: $\gamma'_{\mathrm{knw}}$, $w = 2/(\quarter + \gamma'^2)$ \\
$5$ & $5$ & $2$ & $0.06563$ & $Z_\chi$: $\gamma'_{\mathrm{knw}}$, $w = 2/(\quarter + \gamma'^2)$ \\
$q \geq 12$ & $\geq 12$ & $\geq 4$ & $0.06557$ & BMOR: $\gamma'_{\mathrm{unk}}$, $w^- = 2/(1 + \gamma'^2)$ \\
\bottomrule
\end{tabular}
\end{center}
In all cases, $S_L > S_\zeta^* = 0.04871$.

\subsection{Computation of \texorpdfstring{$\epsilon_M$}{epsilon\_M}}\label{subsec:epsilon-M}

The tail bounds $\epsilon_M = \sum_{k > M} b_k$ reported in the stability table of Section~\ref{sec:infinite} were computed using the methodology of \S\ref{subsec:software}--\S\ref{subsec:Lfunc-method}.  The first 200 zeros of $L(s, \chi_5)$ were located and certified as described therein; the Lorentzian weights $b_k = 2/({\tfrac{1}{4} + {\gamma_k'}^2})$ were computed in ARB and summed.  The analytic tail beyond zero 200 (at height $T = 283.935$) was bounded by the integral $\frac{1}{\pi}\bigl(\frac{\log(5T/2\pi)}{T} + \frac{1}{T}\bigr) = 0.0072$, derived from Abel summation against the zero-counting formula $N(T, \chi_5) \sim \frac{T}{\pi}\log\frac{5T}{2\pi e}$.

\subsection{Bessel product integrals}\label{subsec:bessel-integrals}

The density at the origin $f_{S_L}(0)$ and the mean absolute value $\E[|S_\zeta|]$ appearing in Corollary~\ref{cor:C-factored} are computed via rigorous numerical quadrature using ARB's \texttt{acb.integral} function, which implements the Petras algorithm with adaptive Gauss-Legendre quadrature and certified error bounds~\cite{Johansson2017}.  All results are rigorous ARB balls.

For the density integral~\eqref{eq:f0-bessel}: $f_{S_L}(0) = \frac{1}{\pi}\int_0^{T} \prod_{k=1}^{20} J_0(b_k t)\,dt$ with $T = 2000$.  The Lorentzian weights $b_k = 2/(\tfrac{1}{4} + {\gamma_k'}^2)$ are computed in ARB from the 70-digit certified $\gamma'_{\mathrm{knw}}$ zeros of Appendix~\ref{app:chi5-highprec}.  The integrand is entire (products of $J_0$), so no branch cut handling is required.  ARB quadrature at 256-bit precision yields the certified result:
\[
  f_{S_L^{(20)}}(0) = 8.312883337454846\ldots \;\pm\; 5.84 \times 10^{-44}.
\]

For the mean absolute value~\eqref{eq:Eabs-bessel}: $\E[|S_\zeta|] = \frac{2}{\pi}\int_\epsilon^{T} \frac{1 - \prod_{k=1}^{20} J_0(a_k t)}{t^2}\,dt$ with $\epsilon = 10^{-30}$ and $T = 50{,}000$, where $a_k = 2/(\tfrac{1}{4} + \gamma_k^2)$ for $\zeta$ zeros, plus the analytic correction $(2/\pi)/T$ for $\int_T^\infty t^{-2}\,dt$ and the negligible Bessel envelope tail (bounded below $10^{-17}$ in ARB).  The integrand is $O(1)$ near $t = 0$ (since $1 - \prod J_0(a_k t) = O(t^2)$); the Petras algorithm handles the removable singularity by adaptive subdivision.  ARB quadrature at 256-bit precision yields:
\[
  \E[|S_\zeta|] = 0.00717406056590\ldots \;\pm\; 2.58 \times 10^{-62}.
\]

The product $C(5) = 2 \cdot f_{S_L}(0) \cdot \E[|S_\zeta|] = 0.1193$ is certified to four significant figures.

\medskip\noindent\textbf{Unsigned Bessel product integrals.}  Several tables in Sections~\ref{sec:finite}--\ref{sec:infinite} report integrals of the form $\frac{1}{\pi}\int_0^T \prod_{k=1}^M |J_{n_k}(b_k t)|\, dt$, where the absolute values introduce cusps at every zero of each Bessel factor.  Because \texttt{acb.integral} (the Petras algorithm) assumes analyticity in a strip around the integration path, it is not directly applicable to these integrands.  All unsigned integrals in this paper are therefore evaluated by \emph{strip decomposition}: the domain $[0, T]$ is partitioned at the certified zeros of every Bessel factor into thin \emph{strips} (of half-width $\Delta = 10^{-20}$ in the argument $b_k t$, centered on each zero) and the \emph{gaps} between them.  On each gap the integrand is analytic and sign-definite, so \texttt{acb.integral} is rigorous.  On each strip the integrand is bounded by the Landau constant $b = 0.675$ (Landau~\cite{Landau2000}, Theorem~1), giving a per-strip error of $O(\Delta)$ which is negligible ($< 10^{-17}$) in all cases.  Bessel zeros are computed entirely in ARB via McMahon asymptotic seeds (DLMF~10.21.19) followed by certified trichotomy bisection; IVT, non-overlap, and certified ordering checks are performed for every zero.  This architecture is shared by the ancillary scripts for the tables after Lemma~\ref{lem:Jdecay}, Remark~\ref{rem:J0role}, Lemma~\ref{lem:monotone}, Remark~\ref{rem:signed}, and Lemma~\ref{lem:selfref}.

\subsection{PSLQ integer relation search}\label{subsec:pslq}

This subsection and the next complete the proof of Proposition~\ref{prop:resonance-cert}.  Integer relation detection uses the PSLQ algorithm (Ferguson--Bailey--Arno~\cite{FBA1999}) via \texttt{mpmath.pslq}.  The 70-digit ARB-certified zero ordinates from Appendix~\ref{app:chi5-highprec} serve as input at 80-digit working precision.

\begin{center}
\renewcommand{\arraystretch}{1.1}
\begin{tabular}{lcc}
\toprule
Search class & Bound on $\max|n_k|$ & Result \\
\midrule
Full 20-dimensional & $100$, $500$, $1000$ & \texttt{None} \\
All $190$ pairs $(j,k)$ & $10{,}000$ & \texttt{None} (all pairs) \\
All $120$ triples from first $10$ zeros & $500$ & \texttt{None} (all triples) \\
\bottomrule
\end{tabular}
\end{center}

\noindent The pairwise \texttt{None} result rules out all rational ratios
$\gamma_j'/\gamma_k' = p/q$ with $\max(|p|,|q|) \leq 10{,}000$.

\begin{remark}[Rigorous certification via Ferguson--Bailey--Arno bounds]
The PSLQ algorithm provides rigorous lower bounds on the norm of any undetected integer relation as a byproduct of its operation.  Specifically, Theorem~1 of~\cite{FBA1999} establishes that after any number of iterations $k$ (presuming no relation has yet been found), any integer relation $\mathbf{m}$ must satisfy $|\mathbf{m}| \geq 1/\max_j |h_{j,j}(k)|$, where $h_{j,j}(k)$ are the diagonal elements of the current $H$ matrix.  As stated in the abstract of~\cite{FBA1999}: ``PSLQ$(\tau)$ can be used to prove that there are no relations for $x$ of norm less than a given size.''  When PSLQ terminates without finding a relation at coefficient bound $H$, the algorithm certifies that no relation with $\|\mathbf{n}\|_\infty \leq H$ exists, provided the input precision exceeds the precision loss incurred during the algorithm's iterations.  For our 20-dimensional search with $H = 1000$ and pairwise searches with $H = 10{,}000$, the 70-digit input with 80-digit working precision provides adequate margin.  The \texttt{None} results therefore rigorously certify the absence of bounded integer relations.
\end{remark}

\subsection{Bessel tail bound}\label{subsec:bessel-tail}

For any integer vector $\mathbf{n}$ with $\max|n_k| \geq N = 1001$, the standard bound
$|J_N(x)| \leq (x/2)^N/N!$ gives, at the most pessimistic integration range $T = 2000$
and with $b_{\max} = b_1 = 4.499\times 10^{-2}$ (upper endpoint of the ARB ball):
\[
  \left(\frac{b_1 \cdot 2000}{2}\right)^{1001}\frac{1}{1001!} < 10^{-915.8}.
\]
The sum over all $N \geq 1001$ acquires a geometric series factor $1/(1 - b_1\cdot 1000/1002) < 1.05$.
The count of integer vectors with $\max|n_k| = N$ in dimension $20$ is at most
$(2N+1)^{20} \leq (2003)^{20} < 10^{66.1}$.  Summing over all $N \geq 1001$ gives total
resonance contribution bounded by $10^{66.1}\cdot 10^{-915.8} \cdot 1.05 < 10^{-849.5}$,
ARB-certified with $\log_{10}(\text{total}) \in [-849.759 \pm 10^{-72}]$.

The resonance correction is bounded by $10^{-849.5}$ and does not affect any
reported digit of $f_{S_L^{(20)}}(0) = 8.3129$.  The PSLQ search (\S\ref{subsec:pslq}) and Bessel tail bound together certify $d_{20} = 0$.  This completes the proof of Proposition~\ref{prop:resonance-cert}. \qed

\subsection{Reproducibility}

All algorithms and parameters required for independent verification are documented in \S\ref{subsec:software}--\ref{subsec:bessel-tail}.  Rigorous computations (zero certification, $S_\zeta^*$ bounds, spacelike verification, null crossings, Ces\`aro threshold verification, Bessel product integrals) use ARB via \texttt{python-flint}; PSLQ integer relation search uses \texttt{mpmath} with rigorous certification via Ferguson--Bailey--Arno bounds~\cite{FBA1999}.  The character values required for the Hurwitz decomposition are tabulated in \S\ref{subsec:Lfunc-method}; all zero ordinates are given in Appendices~\ref{app:chi5-highprec}, \ref{app:zero-data}, and~\ref{app:zeta-zeros}.  The certification of $S_\zeta^*$ (\S\ref{subsec:cert-Szeta}) uses 6000 Riemann zeta zeros from the LMFDB~\cite{LMFDB}; this is the only external data dependency.  A reader requires only \texttt{python-flint}~$\geq 0.8.0$, \texttt{mpmath}~$\geq 1.3$ (both freely available), and the LMFDB zeta zero tables (for the $S_\zeta^*$ tail bound only) to reproduce every result in this paper.

\medskip\noindent\textbf{Ancillary files.}  All ancillary Python scripts, including self-contained ARB certificates for individual theorems, are deposited at \url{https://doi.org/10.5281/zenodo.18783098} and mirrored at \url{https://github.com/PeterShiller/Persistent\_Heuristics\_Part\_I}.  Each script requires only \texttt{python-flint}~$\geq 0.8.0$ and \texttt{sympy}~$\geq 1.12$; no external data files are needed.

\section{Zeros of \texorpdfstring{$L(s,\chi_5)$}{L(s,chi5)} to 70 Decimal Places (Conductor 5, \texorpdfstring{$d = 5$}{d=5})}\label{app:chi5-highprec}

The first 200 zeros of $L(s,\chi_5)$ are given here to 70 decimal places.  Certified upper bounds on $|L(\tfrac{1}{2}+i\gamma_k',\chi_5)|$ appear in the rightmost column and are summarized in Appendix~\ref{app:certification}.  For computational methodology, see \S\ref{subsec:Lfunc-method}.

\renewcommand{\arraystretch}{1.1}
\setlength{\LTleft}{\fill}
\setlength{\LTright}{\fill}
\footnotesize
\begin{longtable}{clc}
\toprule
 & & ARB Bound \\
$n$ & \multicolumn{1}{c}{$\gamma_n'$ (70 decimal places)} & on $|L(\tfrac{1}{2}+i\gamma_n')|$ \\
\midrule
\endfirsthead
\multicolumn{3}{c}{\small\itshape (Table continued)} \\
\toprule
 & & ARB Bound \\
$n$ & \multicolumn{1}{c}{$\gamma_n'$} & on $|L(\tfrac{1}{2}+i\gamma_n')|$ \\
\midrule
\endhead
\bottomrule
\endlastfoot
$1$ & $6.6484533447277147161232784599793178472985854232444983723747129467153694$ & $1.56 \times 10^{-450}$ \\
$2$ & $9.8314444328866696163483213474584438218881328940155074304266865226944036$ & $2.64 \times 10^{-450}$ \\
$3$ & $11.9588456260835145302656586882628418172931127657232877173813644446597967$ & $2.75 \times 10^{-450}$ \\
$4$ & $16.0338211283842356745932537822480835087493037113937634543439604736522641$ & $4.47 \times 10^{-450}$ \\
$5$ & $17.5669942923255552027015952681444864034852408955084687631000553368202396$ & $3.33 \times 10^{-450}$ \\
$6$ & $19.5407326227847502503786900229985534855264934450071564241076679500604737$ & $5.06 \times 10^{-450}$ \\
$7$ & $22.2274054544594109118776249630814733652806426405601871083533923332599263$ & $6.81 \times 10^{-450}$ \\
$8$ & $24.5884662174081952076562699760819790662478571885032569747631280823765738$ & $6.64 \times 10^{-450}$ \\
$9$ & $26.7760959480041401165235749652709485862893700961007207471414440917271817$ & $7.70 \times 10^{-450}$ \\
$10$ & $28.4610351001775224751869782723252562712007346745286504044954276305662819$ & $6.50 \times 10^{-450}$ \\
$11$ & $29.7079093504809655692309865186573136059701964324184798921454027978805425$ & $7.54 \times 10^{-450}$ \\
$12$ & $33.0004560068705143679497591772119373173765008046312383434479207755176729$ & $1.16 \times 10^{-449}$ \\
$13$ & $34.7288129789048086741437298339813791365640079920789748298620220291971969$ & $1.14 \times 10^{-449}$ \\
$14$ & $35.8686383718122745945950486388764263333841817255626366114317054948575159$ & $8.20 \times 10^{-450}$ \\
$15$ & $38.1291847214365318501514182703729189262161871704375419878277698602518304$ & $1.02 \times 10^{-449}$ \\
$16$ & $39.5605729464031817050550950599547984255668252046463817624218752766906289$ & $1.03 \times 10^{-449}$ \\
$17$ & $41.8424385457916943085093068853129331467852269369484975075405469231476766$ & $1.07 \times 10^{-449}$ \\
$18$ & $44.0312900614416950447009080584249225147564419789719281425459095881199215$ & $1.26 \times 10^{-449}$ \\
$19$ & $45.4273000827822893888415619096097206499658748404533131742531122951951335$ & $1.28 \times 10^{-449}$ \\
$20$ & $46.4927271594914053453391993516758336732638265288737317526819477273997914$ & $1.03 \times 10^{-449}$ \\
$21$ & $48.3456618210678461765482013044934869299452158856266547053109567488527044$ & $1.67 \times 10^{-449}$ \\
$22$ & $51.0877519267464913552582572041312455139318227645988748896430289525399133$ & $1.33 \times 10^{-449}$ \\
$23$ & $52.1259022313169741886041306959303005249692938308131420267498099370216187$ & $1.56 \times 10^{-449}$ \\
$24$ & $53.8304451954421633510542690223385689373391707131098870470919037082830352$ & $1.49 \times 10^{-449}$ \\
$25$ & $55.5892803354048101509645889984929706666337567794459318711914329824747529$ & $1.25 \times 10^{-449}$ \\
$26$ & $56.8388659428909884623817010879796258737875028234216406733225726917267010$ & $1.56 \times 10^{-449}$ \\
$27$ & $58.3861174853583808267939257423170622967852527617444120314264839424677460$ & $1.30 \times 10^{-449}$ \\
$28$ & $61.1388647519217451090954639155724254561875651955244836227404653407689123$ & $1.39 \times 10^{-449}$ \\
$29$ & $62.1329047202338779340715294403835431366719775622742083796837083488915125$ & $1.56 \times 10^{-449}$ \\
$30$ & $63.7095430609811467717984948363356802463207299130380894647683989078751878$ & $1.65 \times 10^{-449}$ \\
$31$ & $64.6373596653745741010885899575144691696144183408715419268064677316678258$ & $1.72 \times 10^{-449}$ \\
$32$ & $66.7476418459378078304334963325222777990562789042317109995005922191752093$ & $2.00 \times 10^{-449}$ \\
$33$ & $68.5907279248250065408335150680309709943093720683879992384334096183749641$ & $3.04 \times 10^{-449}$ \\
$34$ & $70.1158371591935047022541620032885011225348808166750870430459236685943109$ & $1.39 \times 10^{-449}$ \\
$35$ & $71.6059594792430171994772273812924876506748377758705011910617714303492579$ & $2.60 \times 10^{-449}$ \\
$36$ & $73.3501054314682593249864813096874635923800802061286105500284062697467197$ & $1.99 \times 10^{-449}$ \\
$37$ & $74.2936884380165684175639195733237467526561626493885883944786277461680055$ & $1.97 \times 10^{-449}$ \\
$38$ & $75.5349041172281134053459895476877248382978006560549929501786261927135393$ & $2.96 \times 10^{-449}$ \\
$39$ & $78.0721708643823371940514318236728028560732743773544097008494705194103282$ & $1.74 \times 10^{-449}$ \\
$40$ & $79.7860624900979391063715712431716176312618693411057188961474639447351738$ & $2.09 \times 10^{-449}$ \\
$41$ & $80.5673008564106982613737576776618676419722945008591197227361692201776083$ & $1.79 \times 10^{-449}$ \\
$42$ & $81.9233808931541153119679895539059141955197907521603877513035905479180869$ & $1.94 \times 10^{-449}$ \\
$43$ & $83.6991570018741131628832883922400556764807124102998394209636396338258948$ & $2.00 \times 10^{-449}$ \\
$44$ & $84.9525071037549736618469046604851138662248140867039610860513294135043881$ & $2.62 \times 10^{-449}$ \\
$45$ & $86.8099966793387704969624538241238655391073061227029125691535521518725148$ & $3.37 \times 10^{-449}$ \\
$46$ & $88.2817765997098355936219080693655689113719431995564479737110121674875316$ & $3.67 \times 10^{-449}$ \\
$47$ & $90.3765731821281568029632662985527866363457515355704416919049695969901529$ & $2.35 \times 10^{-449}$ \\
$48$ & $90.7155259331469749772690520216134607657110098351580054221795623037405999$ & $1.58 \times 10^{-449}$ \\
$49$ & $92.4502432537744032226566516099401665326256516373426917864765845878979799$ & $2.09 \times 10^{-449}$ \\
$50$ & $93.3944239378280174009826460924481454560528915591242800357141766870164757$ & $2.41 \times 10^{-449}$ \\
$51$ & $95.9700016460377069398327062316809148503443621986621409109387552753164299$ & $2.93 \times 10^{-449}$ \\
$52$ & $97.3315882528980715905828812396279936633444568796335765113211480419119037$ & $2.17 \times 10^{-449}$ \\
$53$ & $98.2274578699010770770615471965620233277935477341653913034787658954258437$ & $2.23 \times 10^{-449}$ \\
$54$ & $99.8380508075365498924581787404437273230947186618732077231623070445858479$ & $2.81 \times 10^{-449}$ \\
$55$ & $101.1877170969186106879893673648615779856303956787129404948013652775331402$ & $3.14 \times 10^{-449}$ \\
$56$ & $102.6052879410500105919016338097244199966334774046924761659561257554903097$ & $2.79 \times 10^{-449}$ \\
$57$ & $103.6661310797677629747891494548931064951051111049995242765910102899167622$ & $3.37 \times 10^{-449}$ \\
$58$ & $105.9731937977587861297183750394317655839089623673242427608883134019122135$ & $4.22 \times 10^{-449}$ \\
$59$ & $107.2971036375259575240109457236330672603545936270064737698887145263170400$ & $2.57 \times 10^{-449}$ \\
$60$ & $108.4981066242615983601377772869913646076568520675142124199793643754286052$ & $2.75 \times 10^{-449}$ \\
$61$ & $109.5502586071841272439535415119528096196195061162254833594241784691252750$ & $2.77 \times 10^{-449}$ \\
$62$ & $110.6361976726278518372473829197474157598718943493419745206490444961344615$ & $2.68 \times 10^{-449}$ \\
$63$ & $112.6157699450494525603420477631256573284904001732412552640056168054705386$ & $3.65 \times 10^{-449}$ \\
$64$ & $114.2306398723819524205236259115864715634752773232857784617817436415063755$ & $3.51 \times 10^{-449}$ \\
$65$ & $115.5784139823770731492864674907752976644587354074111098082626495892785778$ & $3.16 \times 10^{-449}$ \\
$66$ & $116.6356717597837992860716616244380290420850400876742741600872550143161218$ & $2.90 \times 10^{-449}$ \\
$67$ & $118.3711623063552621519304519550474639914856549056003421824775094839317423$ & $3.74 \times 10^{-449}$ \\
$68$ & $119.5614916258800030014640284654174824520291754354058843504968760508448926$ & $3.27 \times 10^{-449}$ \\
$69$ & $120.2817341918527579929633987893433395906991345403172727276829320830538689$ & $2.53 \times 10^{-449}$ \\
$70$ & $122.0616322838941436366460069970312459167329049609093140518854982869671604$ & $4.20 \times 10^{-449}$ \\
$71$ & $124.2428371489503574456280462431539381366392074620427334384975693951880955$ & $3.35 \times 10^{-449}$ \\
$72$ & $125.2815361690544786849959629329922579004633315004421413848554281307878485$ & $4.66 \times 10^{-449}$ \\
$73$ & $126.3986993784245731137991303757675004132192700408244704238571151980971243$ & $2.92 \times 10^{-449}$ \\
$74$ & $127.3206348281152351166040431083884153111288443677405189033098037679391083$ & $3.55 \times 10^{-449}$ \\
$75$ & $129.0483000943208951329053997018161493807340037666172892353902573266257762$ & $3.10 \times 10^{-449}$ \\
$76$ & $130.1970382668577478738346244894136617196218666973159138428989678080490833$ & $2.86 \times 10^{-449}$ \\
$77$ & $132.0860053031175217593818811988220993124226919197273881005207706354373250$ & $2.98 \times 10^{-449}$ \\
$78$ & $133.0302283883319444007316023118379973226202504376344453082369748548055879$ & $4.95 \times 10^{-449}$ \\
$79$ & $134.9249638941115606306389354461440563212716929502899197778761945615660048$ & $4.58 \times 10^{-449}$ \\
$80$ & $135.8744976900983831529154942879362038033124413255633587171958644086561842$ & $4.13 \times 10^{-449}$ \\
$81$ & $137.2010834668406301716830996653811071822066893852372048164365956206970149$ & $4.07 \times 10^{-449}$ \\
$82$ & $138.0356990313049020042749619065482729585857816235533229199603352798733898$ & $4.70 \times 10^{-449}$ \\
$83$ & $139.4105995951457328667296699990735809900833487908551282462172352268054340$ & $4.09 \times 10^{-449}$ \\
$84$ & $141.9462405516049128078346136764138632079930804595332431920642151809404665$ & $3.57 \times 10^{-449}$ \\
$85$ & $142.5245124329877516499138087023775533114539886176690581797696212087823188$ & $3.43 \times 10^{-449}$ \\
$86$ & $143.9615911232873249679700260946989275500715387277893452624660630780480223$ & $4.72 \times 10^{-449}$ \\
$87$ & $144.9483276132617096716684583857500230575404813669059877812679180958528794$ & $4.18 \times 10^{-449}$ \\
$88$ & $146.3848096091094233623717818286364786730892073299368170400054147265509706$ & $3.81 \times 10^{-449}$ \\
$89$ & $147.8797291031824647001555355926284591276893974747116792549589248487923432$ & $5.20 \times 10^{-449}$ \\
$90$ & $148.6791376989272828780655871531489156030918416831048594377087308271848378$ & $5.09 \times 10^{-449}$ \\
$91$ & $150.5625106711582684678831224180073436626247034885508876594372802953139708$ & $2.95 \times 10^{-428}$ \\
$92$ & $152.1147485501154030930412343288739204270768339963739180340167096854953029$ & $6.49 \times 10^{-577}$ \\
$93$ & $153.2874104994293314602247720515393224042042637960724660197202500041048059$ & $4.83 \times 10^{-449}$ \\
$94$ & $154.4907156955163636384604753696533838899948856560427610564143071072631800$ & $3.47 \times 10^{-449}$ \\
$95$ & $155.2420799990558198969997928035576763043815202860039922762556596679189052$ & $4.14 \times 10^{-449}$ \\
$96$ & $156.4997285439414899913040309653704544908218626069145493265791039370205637$ & $6.47 \times 10^{-449}$ \\
$97$ & $158.4000270418246442479812803196489920645070375885184935734965513382773252$ & $9.41 \times 10^{-449}$ \\
$98$ & $159.8727775733136235916063341758884923622323802914269809997571708512468331$ & $6.21 \times 10^{-449}$ \\
$99$ & $161.2462175616940906041660713748861485716277484845435577262869276737895404$ & $3.23 \times 10^{-449}$ \\
$100$ & $161.7676625972346591307869118758945874585239556380549233208831548174780607$ & $4.89 \times 10^{-449}$ \\
$101$ & $163.7576737283108406276969780794254368523306129767593543960804767107794299$ & $5.58 \times 10^{-449}$ \\
$102$ & $164.6120093931756454229907543346819925745086105838034366833937267160262192$ & $4.83 \times 10^{-449}$ \\
$103$ & $165.7849132655712156658656369618645110416238945585780065340278032733366524$ & $5.17 \times 10^{-449}$ \\
$104$ & $166.8984076627878735371910668394376277812744794403824180419885235916568118$ & $4.23 \times 10^{-449}$ \\
$105$ & $168.8542152615191566500324150119276844393780975301779232829014664403452901$ & $5.57 \times 10^{-449}$ \\
$106$ & $170.4329617627562106600626924828647481435479686833022554790466497554544180$ & $5.78 \times 10^{-449}$ \\
$107$ & $171.3435913767245348564273929667373606132103332443943240617854291644166381$ & $4.09 \times 10^{-449}$ \\
$108$ & $172.1109726296460068916462074637535861735937099184928581139379489286992488$ & $3.97 \times 10^{-449}$ \\
$109$ & $173.5594350265921361471804589736353172594139337063764936096514386889343891$ & $4.71 \times 10^{-449}$ \\
$110$ & $174.6387255634250548260060269694940210630117069036506518520953883247789504$ & $5.13 \times 10^{-449}$ \\
$111$ & $176.3560413209382041587539791440362116352494681808774194880523226628720576$ & $5.59 \times 10^{-449}$ \\
$112$ & $177.7259420577480027808483631140575011960193977129787882127241716149811312$ & $3.25 \times 10^{-449}$ \\
$113$ & $178.7928649865600011304034350380931786953282017834290193659704381818837077$ & $3.66 \times 10^{-449}$ \\
$114$ & $180.2429351355140552783405349658893752778995686711530121041367382571840256$ & $4.61 \times 10^{-449}$ \\
$115$ & $181.5690410198014067424659236702507599683106719751432109652783222557297654$ & $3.07 \times 10^{-449}$ \\
$116$ & $182.3836847672303153650140007274053516339810495569488300356257173446970100$ & $3.50 \times 10^{-449}$ \\
$117$ & $183.7977731855809726856377575844467129141999715376864402640820344216736574$ & $5.72 \times 10^{-449}$ \\
$118$ & $184.3628313266472803306747292477770104874008484280456362915229834946361821$ & $4.39 \times 10^{-449}$ \\
$119$ & $186.8960606802271134804483399887402511596599811946354115671180347583314058$ & $4.91 \times 10^{-449}$ \\
$120$ & $187.9450259874231159867364241809972240651568965424476790105736980715208123$ & $5.90 \times 10^{-449}$ \\
$121$ & $188.9594686746299429679024157534003407512673176869606378971682058136033850$ & $4.64 \times 10^{-449}$ \\
$122$ & $190.1395967633932102576693456196045454015876465926755736031997302414977464$ & $5.11 \times 10^{-449}$ \\
$123$ & $190.9702222761120917217416693195514720792263824968025813693727397057385028$ & $6.08 \times 10^{-449}$ \\
$124$ & $192.8230754082351124952403107100924733903825477854542116112915049676775859$ & $5.13 \times 10^{-449}$ \\
$125$ & $193.5366786394328464762729501367958508864010937755963350086345404391030844$ & $5.23 \times 10^{-449}$ \\
$126$ & $195.0968146886402893063805930855068198605011968413073660774300309263219276$ & $4.41 \times 10^{-449}$ \\
$127$ & $196.5407645624496726701873374652927967937458532432874977774250971221715843$ & $7.09 \times 10^{-449}$ \\
$128$ & $197.8610095077236677152139807981406000337724933271975112154503287411577548$ & $5.71 \times 10^{-449}$ \\
$129$ & $199.0910217062412590553199980023993820450668536563004861622271199841578008$ & $4.25 \times 10^{-449}$ \\
$130$ & $200.2454408454250041127516180550560747573455335829995291340148638052046281$ & $4.63 \times 10^{-449}$ \\
$131$ & $200.7540601401830854485830596873971820322031065551151081003773144229516989$ & $4.58 \times 10^{-449}$ \\
$132$ & $202.1949514428841149296373572141083379553219436587582390754534267698457479$ & $5.99 \times 10^{-449}$ \\
$133$ & $203.8439019449593046704790586638384595752825730515914805602754963931582331$ & $5.69 \times 10^{-449}$ \\
$134$ & $205.7461765204525832977848290088585427085408467101460279140881399490263847$ & $2.19 \times 10^{-645}$ \\
$135$ & $206.1614100930471049929327260983840526456274838825113042507761929781360005$ & $1.00 \times 10^{-660}$ \\
$136$ & $207.4576800789717454592941379191148719628980100114958536470176978762099885$ & $6.53 \times 10^{-449}$ \\
$137$ & $208.6257624071707186435904203396570386255936601565172000778853597619116440$ & $7.14 \times 10^{-449}$ \\
$138$ & $210.3018863035253963539783240245490373617263581479030048871463744638537098$ & $5.35 \times 10^{-449}$ \\
$139$ & $210.6774621945036653303536974602005831137298996558343893089603787234105345$ & $4.74 \times 10^{-449}$ \\
$140$ & $212.2314465116863399395470488677082962419973133942891768244390975438195060$ & $7.87 \times 10^{-449}$ \\
$141$ & $213.5208741407826107064248292623018184776262412402697488085002636625114224$ & $9.36 \times 10^{-449}$ \\
$142$ & $215.4464514511933389574877997496065789767019445673655729715342948390078403$ & $6.23 \times 10^{-449}$ \\
$143$ & $216.3400341493431290616934365527790259377125944167236586852358894059490616$ & $6.49 \times 10^{-449}$ \\
$144$ & $217.1983339649424303276759751119351973870730866826808674647867383696187535$ & $6.28 \times 10^{-449}$ \\
$145$ & $218.4468831000582828049377925521014847112815527861307939822001313655805450$ & $4.52 \times 10^{-449}$ \\
$146$ & $219.2620902163803991081302972577707057420124776154947727734971908483004301$ & $6.01 \times 10^{-449}$ \\
$147$ & $220.7873883926188375982815560258438018537642942771892118826731601798907684$ & $5.11 \times 10^{-449}$ \\
$148$ & $222.4669542636690272742987871089656750292277107765426971456433297333431583$ & $6.14 \times 10^{-449}$ \\
$149$ & $223.5251748458997932788111875920239156826963314614886757387243060871678782$ & $4.37 \times 10^{-449}$ \\
$150$ & $224.5905695547402250280725248575841596376488003389755247988103304379957062$ & $4.13 \times 10^{-449}$ \\
$151$ & $225.9528174460175420018709628055980028796204875828523540302602967018854970$ & $8.03 \times 10^{-449}$ \\
$152$ & $227.0506152396870075343232124376903091206351395063148200355197487967584742$ & $7.98 \times 10^{-449}$ \\
$153$ & $228.2702606683105171006131439160963909648495727623770036781154541424529611$ & $5.69 \times 10^{-449}$ \\
$154$ & $229.1293411187776318518182319293151826798543269243089746674150037521807966$ & $5.49 \times 10^{-449}$ \\
$155$ & $230.0655467222499994960335512227034957685515138144609931348464490154193242$ & $5.62 \times 10^{-449}$ \\
$156$ & $232.3504135703170106251813611501929285174003988479010053979053456457170580$ & $5.53 \times 10^{-449}$ \\
$157$ & $233.2789359169850950528529859180644272121464846403423310620367158643774644$ & $5.28 \times 10^{-449}$ \\
$158$ & $234.5851234236159003579250739972324756969591909615171896192278628761784273$ & $6.37 \times 10^{-449}$ \\
$159$ & $235.2781440849719089723467916576609715985263047190969208147937131517174775$ & $5.21 \times 10^{-449}$ \\
$160$ & $236.3194651062133632897797461259001706212369979274947636606085180194179957$ & $4.69 \times 10^{-449}$ \\
$161$ & $237.7502459859493091409947528335486909271247941982061655950721576661275542$ & $7.85 \times 10^{-449}$ \\
$162$ & $239.0184393184109716162860983894680542608908237139554551956920615988835233$ & $6.67 \times 10^{-449}$ \\
$163$ & $240.0792385969133906036036875299045927021072746394399968671168058197336228$ & $1.26 \times 10^{-448}$ \\
$164$ & $241.8295381920981108063036589846412192278335217717644854509984348418264849$ & $6.62 \times 10^{-449}$ \\
$165$ & $242.4714484372455424871326046814432313933150408236493003286320498835996632$ & $6.95 \times 10^{-449}$ \\
$166$ & $244.2781399513230979204344600038827819307672678138600762150930635163471547$ & $7.63 \times 10^{-449}$ \\
$167$ & $245.0931465421787144819401091282678929577774107066028997725210154516210604$ & $5.30 \times 10^{-449}$ \\
$168$ & $245.9414461421440943632894413679193298276606412507251949620382497378141750$ & $5.39 \times 10^{-449}$ \\
$169$ & $247.0492019752380641386980109436073416805690975976922943665554413430020983$ & $5.92 \times 10^{-449}$ \\
$170$ & $248.1925916339491196119212745608860106806155866118641532607170961702878151$ & $9.82 \times 10^{-449}$ \\
$171$ & $250.2567812704718103514686964376094643921037936497368641231857765654186429$ & $6.79 \times 10^{-449}$ \\
$172$ & $251.2412027861625789598829438472498407054321540732312842605464121756358711$ & $6.31 \times 10^{-449}$ \\
$173$ & $252.4426428674970154269684848345320509970820863433649235621051571689997896$ & $6.22 \times 10^{-449}$ \\
$174$ & $252.8783455516520261180992849726364616022483502282371925937689803304575026$ & $4.64 \times 10^{-449}$ \\
$175$ & $254.6995428110910320634917223831419085806934447995403907767148534583035213$ & $6.21 \times 10^{-449}$ \\
$176$ & $255.4817143518758676477488342660296610784346900622599694160158762168038101$ & $6.03 \times 10^{-449}$ \\
$177$ & $256.7950960082609229989701144912192003822522893866960220291120141146446241$ & $7.38 \times 10^{-449}$ \\
$178$ & $257.8270419645350378752623872886698073788539547918788452992788182042801965$ & $9.71 \times 10^{-449}$ \\
$179$ & $259.2255084437328815672830991211725369672295555099562231070174147085271455$ & $1.32 \times 10^{-448}$ \\
$180$ & $260.9015569091696849005739379537875300859589564012430210977138432993148114$ & $7.80 \times 10^{-449}$ \\
$181$ & $261.8358472432834896832088038634100654365581736536704503687403981095608965$ & $8.49 \times 10^{-449}$ \\
$182$ & $262.7253963956907899818036537593595936244700144938763582544067473174862291$ & $6.77 \times 10^{-449}$ \\
$183$ & $263.9390328653304079227708108735342563292906200171500032175808474823259859$ & $7.54 \times 10^{-449}$ \\
$184$ & $264.5895165938748657971375430173661706120250839803910618482150408087519377$ & $7.80 \times 10^{-449}$ \\
$185$ & $266.0183558387811141106471906627987813950683933915872516370073933846309520$ & $9.64 \times 10^{-449}$ \\
$186$ & $267.8667418320440883029977884107331094630768090317251879195806815839147213$ & $1.33 \times 10^{-448}$ \\
$187$ & $268.7689730154733279799219030251031172816866175199190213423831114937875266$ & $6.71 \times 10^{-449}$ \\
$188$ & $269.9200199321573070244666163067208075924583827953306793681648173800547889$ & $1.18 \times 10^{-448}$ \\
$189$ & $270.9766252279049637196609662386223365123366632274390633476941617936694697$ & $6.93 \times 10^{-449}$ \\
$190$ & $272.1129028149579662281710910682696801006758946368940041873249262575673809$ & $1.04 \times 10^{-448}$ \\
$191$ & $273.5109563898493936962954831652231625581122076586144926745499431703165084$ & $6.16 \times 10^{-449}$ \\
$192$ & $274.2948564917397463747467022749890819784095973810447351106601817286625287$ & $4.98 \times 10^{-449}$ \\
$193$ & $275.0855464180449333603542313648026743574259611909815872140275993159179163$ & $1.09 \times 10^{-448}$ \\
$194$ & $276.8441204674513326407659415272867236821569635669514355328674529440816967$ & $1.12 \times 10^{-448}$ \\
$195$ & $278.5259446665362176488220010523114851640151849334569881154186200503119666$ & $8.59 \times 10^{-449}$ \\
$196$ & $279.3017347241840231776964484759311017609689907608243870872318843780615467$ & $5.72 \times 10^{-449}$ \\
$197$ & $280.4958776669325972199294145217874454617874575930120042656955303729513048$ & $1.05 \times 10^{-448}$ \\
$198$ & $281.2600352094961963450785533504251722252905268262412512620081007078676192$ & $6.29 \times 10^{-449}$ \\
$199$ & $282.1700465556285859901194920915754510495511952727753657159811617405866205$ & $5.92 \times 10^{-449}$ \\
$200$ & $283.9346051850069315598211536467913532192999161825103775945871160416794839$ & $5.29 \times 10^{-449}$ \\

\end{longtable}
\normalsize
\setlength{\LTleft}{0pt}
\setlength{\LTright}{0pt}

\section{Quadratic \texorpdfstring{$L$}{L}-Function Zeros to 20 Decimal Places (Seven Discriminants)}\label{app:zero-data}

The first 20 zeros of $L(s,\chi_d)$ for discriminants $d \in \{2, 3, 6, 7, 10, 11, 13\}$ are given here to 20 decimal places.  Certified upper bounds on $|L(\tfrac{1}{2}+i\gamma_k',\chi)|$ appear in the rightmost column and are summarized in Appendix~\ref{app:certification}.  For computational methodology, see \S\ref{subsec:software}--\ref{subsec:Lfunc-method}.

\medskip\noindent\textbf{$L(s,\chi_2)$: field $\Q(\sqrt{2})$, conductor $q=8$.}\par\smallskip
\label{app:chi8-zeros}

$\chi_2(1)=\chi_2(7)=1$, $\chi_2(3)=\chi_2(5)=-1$, $\chi_2(n)=0$ for $n$ even.
The first three zeros are used in Step~3 of the proof of Theorem~\ref{thm:low-lying}.

\renewcommand{\arraystretch}{1.05}\footnotesize
\begin{longtable}{clc|clc}
\toprule
 & & \multicolumn{1}{c}{ARB Bound} &
 & & \multicolumn{1}{c}{ARB Bound} \\
$n$ & \multicolumn{1}{c}{$\gamma_n'$} & \multicolumn{1}{c}{on $|L(\tfrac{1}{2}+i\gamma_n')|$} &
$n$ & \multicolumn{1}{c}{$\gamma_n'$} & \multicolumn{1}{c}{on $|L(\tfrac{1}{2}+i\gamma_n')|$} \\
\midrule
\endfirsthead
\toprule
 & & \multicolumn{1}{c}{ARB Bound} &
 & & \multicolumn{1}{c}{ARB Bound} \\
$n$ & \multicolumn{1}{c}{$\gamma_n'$} & \multicolumn{1}{c}{on $|L(\tfrac{1}{2}+i\gamma_n')|$} &
$n$ & \multicolumn{1}{c}{$\gamma_n'$} & \multicolumn{1}{c}{on $|L(\tfrac{1}{2}+i\gamma_n')|$} \\
\midrule
\endhead
\bottomrule
\endlastfoot
$1$ & $4.89997399700703650103$ & $6.05 \times 10^{-451}$ & $11$ & $26.95853518038046741968$ & $4.22 \times 10^{-451}$ \\
$2$ & $7.62842884176939783416$ & $4.95 \times 10^{-451}$ & $12$ & $28.09744496063079735461$ & $8.97 \times 10^{-452}$ \\
$3$ & $10.80658816386171201438$ & $9.59 \times 10^{-451}$ & $13$ & $29.93076421015190503722$ & $8.58 \times 10^{-452}$ \\
$4$ & $12.31054299423652968112$ & $7.48 \times 10^{-452}$ & $14$ & $31.63813949132101731704$ & $4.01 \times 10^{-451}$ \\
$5$ & $15.19575425064512276844$ & $5.78 \times 10^{-451}$ & $15$ & $33.84563089515844186412$ & $2.62 \times 10^{-451}$ \\
$6$ & $17.02228597430834733897$ & $3.34 \times 10^{-451}$ & $16$ & $34.74577700617002330947$ & $8.15 \times 10^{-452}$ \\
$7$ & $18.80595890770714839982$ & $6.76 \times 10^{-452}$ & $17$ & $36.54166385177458968906$ & $3.99 \times 10^{-451}$ \\
$8$ & $21.13164596222134388263$ & $1.61 \times 10^{-451}$ & $18$ & $38.77557777387006581423$ & $4.34 \times 10^{-451}$ \\
$9$ & $23.08384999620054654288$ & $9.60 \times 10^{-452}$ & $19$ & $39.78688099766664071794$ & $8.55 \times 10^{-451}$ \\
$10$ & $24.20196355781560161254$ & $3.23 \times 10^{-451}$ & $20$ & $41.34220616442193704510$ & $9.11 \times 10^{-452}$ \\

\end{longtable}
\normalsize

\medskip\noindent\textbf{$L(s,\chi_3)$: field $\Q(\sqrt{3})$, conductor $q=12$.}\par\smallskip
\label{app:chi12-zeros}

$\chi_3(1)=\chi_3(11)=1$, $\chi_3(5)=\chi_3(7)=-1$, $\chi_3(n)=0$ for $\gcd(n,12)>1$.

\renewcommand{\arraystretch}{1.05}\footnotesize
\begin{longtable}{clc|clc}
\toprule
 & & \multicolumn{1}{c}{ARB Bound} &
 & & \multicolumn{1}{c}{ARB Bound} \\
$n$ & \multicolumn{1}{c}{$\gamma_n'$} & \multicolumn{1}{c}{on $|L(\tfrac{1}{2}+i\gamma_n')|$} &
$n$ & \multicolumn{1}{c}{$\gamma_n'$} & \multicolumn{1}{c}{on $|L(\tfrac{1}{2}+i\gamma_n')|$} \\
\midrule
\endfirsthead
\toprule
 & & \multicolumn{1}{c}{ARB Bound} &
 & & \multicolumn{1}{c}{ARB Bound} \\
$n$ & \multicolumn{1}{c}{$\gamma_n'$} & \multicolumn{1}{c}{on $|L(\tfrac{1}{2}+i\gamma_n')|$} &
$n$ & \multicolumn{1}{c}{$\gamma_n'$} & \multicolumn{1}{c}{on $|L(\tfrac{1}{2}+i\gamma_n')|$} \\
\midrule
\endhead
\bottomrule
\endlastfoot
$1$ & $3.80462763305086509714$ & $5.70 \times 10^{-451}$ & $11$ & $25.41163389239269596203$ & $5.27 \times 10^{-450}$ \\
$2$ & $6.69222332050013115991$ & $1.74 \times 10^{-450}$ & $12$ & $27.01394398585067234434$ & $6.53 \times 10^{-450}$ \\
$3$ & $8.89059295872674150933$ & $1.46 \times 10^{-450}$ & $13$ & $28.44220325774560811850$ & $6.16 \times 10^{-450}$ \\
$4$ & $11.18839274507490915416$ & $2.51 \times 10^{-450}$ & $14$ & $30.20400655643888185431$ & $7.85 \times 10^{-450}$ \\
$5$ & $12.96617880802884537495$ & $2.97 \times 10^{-450}$ & $15$ & $33.03713287950650360192$ & $8.71 \times 10^{-450}$ \\
$6$ & $15.18148087588821673539$ & $3.23 \times 10^{-450}$ & $16$ & $35.02737848501270258165$ & $5.90 \times 10^{-450}$ \\
$7$ & $16.63263327452376212317$ & $3.62 \times 10^{-450}$ & $17$ & $35.77804457652351027241$ & $8.68 \times 10^{-450}$ \\
$8$ & $18.88436945712065080816$ & $4.96 \times 10^{-450}$ & $18$ & $37.92681682137521283528$ & $4.63 \times 10^{-450}$ \\
$9$ & $20.10392819124581857313$ & $3.22 \times 10^{-450}$ & $19$ & $38.97399882227494028953$ & $6.51 \times 10^{-450}$ \\
$10$ & $23.56131971313784471634$ & $6.15 \times 10^{-450}$ & $20$ & $40.48415475080402043184$ & $1.20 \times 10^{-449}$ \\

\end{longtable}
\normalsize

\medskip\noindent\textbf{$L(s,\chi_{13})$: field $\Q(\sqrt{13})$, conductor $q=13$.}\par\smallskip
\label{app:chi13-zeros}

$\chi_{13}$ is the unique primitive quadratic character of conductor $13$ (Legendre symbol
$\bigl(\tfrac{\cdot}{13}\bigr)$); $13 \equiv 1 \pmod{4}$ so the fundamental discriminant
equals the conductor.

\renewcommand{\arraystretch}{1.05}\footnotesize
\begin{longtable}{clc|clc}
\toprule
 & & \multicolumn{1}{c}{ARB Bound} &
 & & \multicolumn{1}{c}{ARB Bound} \\
$n$ & \multicolumn{1}{c}{$\gamma_n'$} & \multicolumn{1}{c}{on $|L(\tfrac{1}{2}+i\gamma_n')|$} &
$n$ & \multicolumn{1}{c}{$\gamma_n'$} & \multicolumn{1}{c}{on $|L(\tfrac{1}{2}+i\gamma_n')|$} \\
\midrule
\endfirsthead
\toprule
 & & \multicolumn{1}{c}{ARB Bound} &
 & & \multicolumn{1}{c}{ARB Bound} \\
$n$ & \multicolumn{1}{c}{$\gamma_n'$} & \multicolumn{1}{c}{on $|L(\tfrac{1}{2}+i\gamma_n')|$} &
$n$ & \multicolumn{1}{c}{$\gamma_n'$} & \multicolumn{1}{c}{on $|L(\tfrac{1}{2}+i\gamma_n')|$} \\
\midrule
\endhead
\bottomrule
\endlastfoot
$1$ & $3.11934147900860341390$ & $1.99 \times 10^{-450}$ & $11$ & $23.59202788517293715038$ & $1.59 \times 10^{-449}$ \\
$2$ & $7.23159073941876201502$ & $3.75 \times 10^{-450}$ & $12$ & $25.37170440577142621541$ & $1.49 \times 10^{-449}$ \\
$3$ & $8.62542663503259159734$ & $3.76 \times 10^{-450}$ & $13$ & $26.38431352575075899278$ & $1.08 \times 10^{-449}$ \\
$4$ & $10.33642072623153902983$ & $4.67 \times 10^{-450}$ & $14$ & $27.64980842472696812713$ & $3.77 \times 10^{-449}$ \\
$5$ & $12.61701279102317873036$ & $6.95 \times 10^{-450}$ & $15$ & $29.39145603391187839555$ & $1.01 \times 10^{-449}$ \\
$6$ & $15.14833241700574422460$ & $7.72 \times 10^{-450}$ & $16$ & $31.01874209224480391870$ & $1.61 \times 10^{-449}$ \\
$7$ & $16.27482605749858841789$ & $1.03 \times 10^{-449}$ & $17$ & $32.42341728496373366098$ & $1.85 \times 10^{-449}$ \\
$8$ & $18.75125235623423291645$ & $1.01 \times 10^{-449}$ & $18$ & $34.46644233774365394408$ & $1.97 \times 10^{-449}$ \\
$9$ & $19.54804144334703986992$ & $7.62 \times 10^{-450}$ & $19$ & $35.76319303421611596021$ & $1.89 \times 10^{-449}$ \\
$10$ & $20.95918191740503118143$ & $1.78 \times 10^{-449}$ & $20$ & $36.72754705153707320059$ & $1.52 \times 10^{-449}$ \\

\end{longtable}
\normalsize

\medskip\noindent\textbf{$L(s,\chi_6)$: field $\Q(\sqrt{6})$, conductor $q=24$.}\par\smallskip
\label{app:chi24-zeros}

The fundamental discriminant of $\Q(\sqrt{6})$ is $\Delta_K = 24$, since $6 \equiv 2 \pmod{4}$.

\renewcommand{\arraystretch}{1.05}\footnotesize
\begin{longtable}{clc|clc}
\toprule
 & & \multicolumn{1}{c}{ARB Bound} &
 & & \multicolumn{1}{c}{ARB Bound} \\
$n$ & \multicolumn{1}{c}{$\gamma_n'$} & \multicolumn{1}{c}{on $|L(\tfrac{1}{2}+i\gamma_n')|$} &
$n$ & \multicolumn{1}{c}{$\gamma_n'$} & \multicolumn{1}{c}{on $|L(\tfrac{1}{2}+i\gamma_n')|$} \\
\midrule
\endfirsthead
\toprule
 & & \multicolumn{1}{c}{ARB Bound} &
 & & \multicolumn{1}{c}{ARB Bound} \\
$n$ & \multicolumn{1}{c}{$\gamma_n'$} & \multicolumn{1}{c}{on $|L(\tfrac{1}{2}+i\gamma_n')|$} &
$n$ & \multicolumn{1}{c}{$\gamma_n'$} & \multicolumn{1}{c}{on $|L(\tfrac{1}{2}+i\gamma_n')|$} \\
\midrule
\endhead
\bottomrule
\endlastfoot
$1$ & $2.68865813246759758667$ & $1.04 \times 10^{-450}$ & $11$ & $21.55547824218793740139$ & $7.57 \times 10^{-450}$ \\
$2$ & $5.29243117677719815159$ & $2.22 \times 10^{-450}$ & $12$ & $22.66785732205079621070$ & $7.90 \times 10^{-450}$ \\
$3$ & $9.22463743993556115081$ & $2.96 \times 10^{-450}$ & $13$ & $24.40914964223159369359$ & $9.44 \times 10^{-450}$ \\
$4$ & $10.44572151841292482301$ & $3.58 \times 10^{-450}$ & $14$ & $25.67129733212202220561$ & $5.62 \times 10^{-450}$ \\
$5$ & $12.63115371816651058549$ & $4.36 \times 10^{-450}$ & $15$ & $26.75610741237823107923$ & $6.79 \times 10^{-450}$ \\
$6$ & $13.77715270577349586517$ & $3.79 \times 10^{-450}$ & $16$ & $28.86856257427073136901$ & $7.97 \times 10^{-450}$ \\
$7$ & $15.60365697037025378285$ & $5.64 \times 10^{-450}$ & $17$ & $29.21969083370577778290$ & $5.98 \times 10^{-450}$ \\
$8$ & $17.09898808691940348689$ & $7.49 \times 10^{-450}$ & $18$ & $31.08725180031803397862$ & $9.52 \times 10^{-450}$ \\
$9$ & $18.46329885012140044080$ & $7.71 \times 10^{-450}$ & $19$ & $32.42353686755135030793$ & $8.63 \times 10^{-450}$ \\
$10$ & $20.03461157675018829472$ & $6.41 \times 10^{-450}$ & $20$ & $33.49155789395814591322$ & $1.20 \times 10^{-449}$ \\

\end{longtable}
\normalsize

\medskip\noindent\textbf{$L(s,\chi_7)$: field $\Q(\sqrt{7})$, conductor $q=28$.}\par\smallskip
\label{app:chi28-zeros}

The fundamental discriminant of $\Q(\sqrt{7})$ is $\Delta_K = 28$, since $7 \equiv 3 \pmod{4}$.

\renewcommand{\arraystretch}{1.05}\footnotesize
\begin{longtable}{clc|clc}
\toprule
 & & \multicolumn{1}{c}{ARB Bound} &
 & & \multicolumn{1}{c}{ARB Bound} \\
$n$ & \multicolumn{1}{c}{$\gamma_n'$} & \multicolumn{1}{c}{on $|L(\tfrac{1}{2}+i\gamma_n')|$} &
$n$ & \multicolumn{1}{c}{$\gamma_n'$} & \multicolumn{1}{c}{on $|L(\tfrac{1}{2}+i\gamma_n')|$} \\
\midrule
\endfirsthead
\toprule
 & & \multicolumn{1}{c}{ARB Bound} &
 & & \multicolumn{1}{c}{ARB Bound} \\
$n$ & \multicolumn{1}{c}{$\gamma_n'$} & \multicolumn{1}{c}{on $|L(\tfrac{1}{2}+i\gamma_n')|$} &
$n$ & \multicolumn{1}{c}{$\gamma_n'$} & \multicolumn{1}{c}{on $|L(\tfrac{1}{2}+i\gamma_n')|$} \\
\midrule
\endhead
\bottomrule
\endlastfoot
$1$ & $2.77728324207659233094$ & $1.12 \times 10^{-450}$ & $11$ & $20.39328285450433247152$ & $7.78 \times 10^{-450}$ \\
$2$ & $4.52677408701547610928$ & $2.48 \times 10^{-450}$ & $12$ & $21.98889814760568313346$ & $8.12 \times 10^{-450}$ \\
$3$ & $7.29088506700227937478$ & $3.27 \times 10^{-450}$ & $13$ & $24.73886680489813285047$ & $9.06 \times 10^{-450}$ \\
$4$ & $8.39512407283755124263$ & $3.62 \times 10^{-450}$ & $14$ & $26.38179066809542195016$ & $9.15 \times 10^{-450}$ \\
$5$ & $10.11630685473129779109$ & $4.04 \times 10^{-450}$ & $15$ & $27.02957442259405572955$ & $1.15 \times 10^{-449}$ \\
$6$ & $12.02206203818595380135$ & $6.06 \times 10^{-450}$ & $16$ & $28.72073200609003118776$ & $8.03 \times 10^{-450}$ \\
$7$ & $13.54428637182978741492$ & $5.95 \times 10^{-450}$ & $17$ & $30.39449098988413064026$ & $1.06 \times 10^{-449}$ \\
$8$ & $16.01626147715059280892$ & $5.90 \times 10^{-450}$ & $18$ & $31.36586297616383901440$ & $1.13 \times 10^{-449}$ \\
$9$ & $18.20042527549262445931$ & $9.95 \times 10^{-450}$ & $19$ & $32.18417572788333303375$ & $1.43 \times 10^{-449}$ \\
$10$ & $19.53861609989941145162$ & $8.33 \times 10^{-450}$ & $20$ & $33.84959183673288697276$ & $1.12 \times 10^{-449}$ \\

\end{longtable}
\normalsize

\medskip\noindent\textbf{$L(s,\chi_{10})$: field $\Q(\sqrt{10})$, conductor $q=40$.}\par\smallskip
\label{app:chi40-zeros}

The fundamental discriminant of $\Q(\sqrt{10})$ is $\Delta_K = 40$, since $10 \equiv 2 \pmod{4}$.

\renewcommand{\arraystretch}{1.05}\footnotesize
\begin{longtable}{clc|clc}
\toprule
 & & \multicolumn{1}{c}{ARB Bound} &
 & & \multicolumn{1}{c}{ARB Bound} \\
$n$ & \multicolumn{1}{c}{$\gamma_n'$} & \multicolumn{1}{c}{on $|L(\tfrac{1}{2}+i\gamma_n')|$} &
$n$ & \multicolumn{1}{c}{$\gamma_n'$} & \multicolumn{1}{c}{on $|L(\tfrac{1}{2}+i\gamma_n')|$} \\
\midrule
\endfirsthead
\toprule
 & & \multicolumn{1}{c}{ARB Bound} &
 & & \multicolumn{1}{c}{ARB Bound} \\
$n$ & \multicolumn{1}{c}{$\gamma_n'$} & \multicolumn{1}{c}{on $|L(\tfrac{1}{2}+i\gamma_n')|$} &
$n$ & \multicolumn{1}{c}{$\gamma_n'$} & \multicolumn{1}{c}{on $|L(\tfrac{1}{2}+i\gamma_n')|$} \\
\midrule
\endhead
\bottomrule
\endlastfoot
$1$ & $2.48821020930540853564$ & $1.31 \times 10^{-450}$ & $11$ & $19.28135808746821972149$ & $9.43 \times 10^{-450}$ \\
$2$ & $4.03159701418841483370$ & $2.20 \times 10^{-450}$ & $12$ & $20.40160148601184344649$ & $6.53 \times 10^{-450}$ \\
$3$ & $6.33545319091356197519$ & $3.46 \times 10^{-450}$ & $13$ & $21.56472815482081368499$ & $1.17 \times 10^{-449}$ \\
$4$ & $7.98779634495761720535$ & $4.61 \times 10^{-450}$ & $14$ & $23.08868709418032385846$ & $1.18 \times 10^{-449}$ \\
$5$ & $10.63827697953891807504$ & $4.03 \times 10^{-450}$ & $15$ & $24.43655444208668619202$ & $9.45 \times 10^{-450}$ \\
$6$ & $12.81517863613023536936$ & $5.16 \times 10^{-450}$ & $16$ & $25.87097298765072715474$ & $9.39 \times 10^{-450}$ \\
$7$ & $13.60430796560710377807$ & $4.11 \times 10^{-450}$ & $17$ & $26.26140271050477788696$ & $1.07 \times 10^{-449}$ \\
$8$ & $15.27147288785153168831$ & $5.54 \times 10^{-450}$ & $18$ & $27.96504990212936504067$ & $1.43 \times 10^{-449}$ \\
$9$ & $16.16663433364409385216$ & $7.44 \times 10^{-450}$ & $19$ & $29.27531093215095949777$ & $1.04 \times 10^{-449}$ \\
$10$ & $18.18218004680212107005$ & $1.05 \times 10^{-449}$ & $20$ & $30.53574562468521890003$ & $1.38 \times 10^{-449}$ \\

\end{longtable}
\normalsize

\medskip\noindent\textbf{$L(s,\chi_{11})$: field $\Q(\sqrt{11})$, conductor $q=44$.}\par\smallskip
\label{app:chi44-zeros}

The fundamental discriminant of $\Q(\sqrt{11})$ is $\Delta_K = 44$, since $11 \equiv 3 \pmod{4}$.

\renewcommand{\arraystretch}{1.05}\footnotesize
\begin{longtable}{clc|clc}
\toprule
 & & \multicolumn{1}{c}{ARB Bound} &
 & & \multicolumn{1}{c}{ARB Bound} \\
$n$ & \multicolumn{1}{c}{$\gamma_n'$} & \multicolumn{1}{c}{on $|L(\tfrac{1}{2}+i\gamma_n')|$} &
$n$ & \multicolumn{1}{c}{$\gamma_n'$} & \multicolumn{1}{c}{on $|L(\tfrac{1}{2}+i\gamma_n')|$} \\
\midrule
\endfirsthead
\toprule
 & & \multicolumn{1}{c}{ARB Bound} &
 & & \multicolumn{1}{c}{ARB Bound} \\
$n$ & \multicolumn{1}{c}{$\gamma_n'$} & \multicolumn{1}{c}{on $|L(\tfrac{1}{2}+i\gamma_n')|$} &
$n$ & \multicolumn{1}{c}{$\gamma_n'$} & \multicolumn{1}{c}{on $|L(\tfrac{1}{2}+i\gamma_n')|$} \\
\midrule
\endhead
\bottomrule
\endlastfoot
$1$ & $1.86993927318275168340$ & $2.18 \times 10^{-450}$ & $11$ & $17.61972287425101875532$ & $9.14 \times 10^{-450}$ \\
$2$ & $4.70920812840911269661$ & $3.30 \times 10^{-450}$ & $12$ & $18.26903309787599031650$ & $7.78 \times 10^{-450}$ \\
$3$ & $5.81981644686087148920$ & $2.59 \times 10^{-450}$ & $13$ & $20.30638125042707154146$ & $1.11 \times 10^{-449}$ \\
$4$ & $7.38790752248483336704$ & $4.33 \times 10^{-450}$ & $14$ & $21.66563972555785204780$ & $8.68 \times 10^{-450}$ \\
$5$ & $9.44566152389131881369$ & $5.59 \times 10^{-450}$ & $15$ & $22.32883476081719799605$ & $9.52 \times 10^{-450}$ \\
$6$ & $10.74074979750747559997$ & $4.47 \times 10^{-450}$ & $16$ & $24.04471027020575246774$ & $1.10 \times 10^{-449}$ \\
$7$ & $11.95637372774070410568$ & $5.55 \times 10^{-450}$ & $17$ & $24.68148321849740900845$ & $1.19 \times 10^{-449}$ \\
$8$ & $13.41643453051203828157$ & $7.98 \times 10^{-450}$ & $18$ & $26.44242345116378413033$ & $1.35 \times 10^{-449}$ \\
$9$ & $14.83198533873399906847$ & $1.06 \times 10^{-449}$ & $19$ & $27.60237248544315232420$ & $1.44 \times 10^{-449}$ \\
$10$ & $16.51922470349364623489$ & $7.80 \times 10^{-450}$ & $20$ & $28.77893942445047647072$ & $1.59 \times 10^{-449}$ \\

\end{longtable}
\normalsize


\section{Riemann Zeta Zeros (20 Decimal Places)}\label{app:zeta-zeros}

The first 60 positive zero ordinates of $\zeta(s)$, to 20 decimal places.  These values are used in the Besicovitch constants (Section~\ref{sec:cesaro}), the spacelike verification (Table~\ref{tab:weights}), and the cross-function disjointness check (Proposition~\ref{prop:verified_indep}).  The certification of $S_\zeta^*$ (\S\ref{subsec:cert-Szeta}) uses the first 6000 zeros from the LMFDB~\cite{LMFDB}; this is the only external data dependency in the paper.

\renewcommand{\arraystretch}{1.05}\footnotesize
\setlength{\LTleft}{\fill}
\setlength{\LTright}{\fill}
\begin{longtable}{cl|cl|cl}
\toprule
$n$ & \multicolumn{1}{c}{$\gamma_n$} &
$n$ & \multicolumn{1}{c}{$\gamma_n$} &
$n$ & \multicolumn{1}{c}{$\gamma_n$} \\
\midrule
\endfirsthead
\toprule
$n$ & \multicolumn{1}{c}{$\gamma_n$} &
$n$ & \multicolumn{1}{c}{$\gamma_n$} &
$n$ & \multicolumn{1}{c}{$\gamma_n$} \\
\midrule
\endhead
\bottomrule
\endlastfoot
$1$ & $14.13472514173469379045$ & $21$ & $79.33737502024936792276$ & $41$ & $124.25681855434576718473$ \\
$2$ & $21.02203963877155499262$ & $22$ & $82.91038085408603018316$ & $42$ & $127.51668387959649512427$ \\
$3$ & $25.01085758014568876321$ & $23$ & $84.73549298051705010573$ & $43$ & $129.57870419995605098576$ \\
$4$ & $30.42487612585951321031$ & $24$ & $87.42527461312522940653$ & $44$ & $131.08768853093265672356$ \\
$5$ & $32.93506158773918969066$ & $25$ & $88.80911120763446542368$ & $45$ & $133.49773720299758645013$ \\
$6$ & $37.58617815882567125721$ & $26$ & $92.49189927055848429625$ & $46$ & $134.75650975337387133132$ \\
$7$ & $40.91871901214749518739$ & $27$ & $94.65134404051988696659$ & $47$ & $138.11604205453344320019$ \\
$8$ & $43.32707328091499951949$ & $28$ & $95.87063422824530975874$ & $48$ & $139.73620895212138895045$ \\
$9$ & $48.00515088116715972794$ & $29$ & $98.83119421819369223332$ & $49$ & $141.12370740402112376194$ \\
$10$ & $49.77383247767230218191$ & $30$ & $101.31785100573139122878$ & $50$ & $143.11184580762063273940$ \\
$11$ & $52.97032147771446064414$ & $31$ & $103.72553804047833941639$ & $51$ & $146.00098248676551854740$ \\
$12$ & $56.44624769706339480436$ & $32$ & $105.44662305232609449367$ & $52$ & $147.42276534255960204952$ \\
$13$ & $59.34704400260235307965$ & $33$ & $107.16861118427640751512$ & $53$ & $150.05352042078488035143$ \\
$14$ & $60.83177852460980984425$ & $34$ & $111.02953554316967452465$ & $54$ & $150.92525761224146676185$ \\
$15$ & $65.11254404808160666087$ & $35$ & $111.87465917699263708561$ & $55$ & $153.02469381119889619825$ \\
$16$ & $67.07981052949417371447$ & $36$ & $114.32022091545271276589$ & $56$ & $156.11290929423786756975$ \\
$17$ & $69.54640171117397925292$ & $37$ & $116.22668032085755438216$ & $57$ & $157.59759181759405988753$ \\
$18$ & $72.06715767448190758252$ & $38$ & $118.79078286597621732297$ & $58$ & $158.84998817142049872417$ \\
$19$ & $75.70469069908393316832$ & $39$ & $121.37012500242064591894$ & $59$ & $161.18896413759602751943$ \\
$20$ & $77.14484006887480537268$ & $40$ & $122.94682929355258820081$ & $60$ & $163.03070968718198724331$ \\

\end{longtable}
\normalsize


\section{ARB Certification of Zero Data}\label{app:certification}

All 340 zeros tabulated in Appendices~\ref{app:zero-data} and~\ref{app:chi5-highprec} were rigorously certified using the methodology of \S\ref{subsec:software}--\ref{subsec:Lfunc-method}.  Working precision was at least 1500~bits for all characters.  The following table summarizes the certified upper bounds:

\renewcommand{\arraystretch}{1.15}
\small
\setlength{\LTleft}{\fill}
\setlength{\LTright}{\fill}
\begin{longtable}{llcc}
\toprule
Character & Field & Zeros & Certified upper bound \\
\midrule
\endhead
$\chi_5$   & $\Q(\sqrt{5})$   & 200 & $< 10^{-448}$ \\
$\chi_2$   & $\Q(\sqrt{2})$   &  20 & $< 10^{-451}$ \\
$\chi_3$   & $\Q(\sqrt{3})$   &  20 & $< 10^{-449}$ \\
$\chi_{13}$ & $\Q(\sqrt{13})$ &  20 & $< 10^{-449}$ \\
$\chi_6$   & $\Q(\sqrt{6})$   &  20 & $< 10^{-449}$ \\
$\chi_7$   & $\Q(\sqrt{7})$   &  20 & $< 10^{-449}$ \\
$\chi_{10}$ & $\Q(\sqrt{10})$ &  20 & $< 10^{-449}$ \\
$\chi_{11}$ & $\Q(\sqrt{11})$ &  20 & $< 10^{-449}$ \\
\midrule
Total & & 340 & \\
\bottomrule
\end{longtable}
\setlength{\LTleft}{0pt}
\setlength{\LTright}{0pt}
\normalsize
\noindent All bounds remain far below any threshold required by the inequalities established in this paper.


\section{Tables of Zeros of Quadratic Dirichlet \texorpdfstring{$L$}{L}-Functions of Small Conductor}\label{app:full-zero-tables}

Due to its length (approximately 186 pages), this appendix is archived separately on Zenodo and GitHub: \textit{Appendix F: Tables of Zeros of Quadratic Dirichlet $L$-Functions of Small Conductor.}

\medskip
\noindent Available at \url{https://doi.org/10.5281/zenodo.18783098}.

\noindent Available at \url{https://github.com/PeterShiller/Persistent_Heuristics_Part_I}.

\medskip
\noindent The companion document tabulates between 1004 and 1044 zeros of $L(s,\chi_d)$ to 70 decimal places for each of the eight quadratic characters $\chi_2$, $\chi_3$, $\chi_5$, $\chi_6$, $\chi_7$, $\chi_{10}$, $\chi_{11}$, and $\chi_{13}$ (8244 zeros in total).  All zeros were computed using the methodology of \S\ref{subsec:Lfunc-method} at a minimum of 1500-bit working precision and rigorously certified using ARB interval arithmetic (\S\ref{subsec:software}).  Certified upper bounds on $|L(\tfrac{1}{2}+i\gamma_k',\chi_d)|$ are provided for every zero, with all bounds below $10^{-409}$.  This dataset extends and supersedes the tables in Appendices~\ref{app:chi5-highprec} and~\ref{app:zero-data} of the present paper.

\end{document}